\documentclass[11pt,reqno]{amsart}%
\usepackage{amsmath,amsfonts,amstext,amssymb,amsbsy,amsopn,amsthm,eucal,amsthm}
\usepackage[colorlinks=true]{hyperref}
\usepackage{enumerate}
\usepackage{amsfonts}
\usepackage{amsmath}
\usepackage{amssymb}
\usepackage{graphicx}%
\usepackage[final]{optional}
\usepackage[centertags]{amsmath}%
\usepackage{amsfonts,amssymb,amsthm, stmaryrd , graphicx, mathrsfs, mathpazo}%
\usepackage[usenames]{color}
\usepackage[toc,page]{appendix}

\usepackage[normalem]{ulem} % either use this (simple) or
\usepackage{soul}

\usepackage{mathrsfs}
\usepackage[abbrev,nobysame]{amsrefs}

\usepackage[usenames,dvipsnames]{xcolor}

\newtheorem{thm}{Theorem}[section]
\newtheorem{lem}[thm]{Lemma}
\newtheorem{cor}[thm]{Corollary}
\newtheorem{prop}[thm]{Proposition}

\newtheorem{conj}[thm]{Conjecture}

\theoremstyle{definition}

\newtheorem{defn}[thm]{Definition}

\newtheorem{clm}[thm]{Claim}

\newcommand{\be}{\begin{equation}}
\newcommand{\ee}{\end{equation}}

\def\Xint#1{\mathchoice
{\XXint\displaystyle\textstyle{#1}}%
{\XXint\textstyle\scriptstyle{#1}}%
{\XXint\scriptstyle\scriptscriptstyle{#1}}%
{\XXint\scriptscriptstyle\scriptscriptstyle{#1}}%
\!\int}
\def\XXint#1#2#3{{\setbox0=\hbox{$#1{#2#3}{\int}$}
\vcenter{\hbox{$#2#3$}}\kern-.5\wd0}}

\def\dashint{\Xint-}
\newcommand{\nocontentsline}[3]{}
\newcommand{\tocless}[2]{\bgroup\let\addcontentsline=\nocontentsline#1{#2}\egroup}

\newcommand{\Pro}{\mathcal{P}}
\newcommand{\R}{\mathbb{R}}

\newcommand{\M}{\mathcal{M}}
\newcommand{\LIP}{{\bf LIP}}

\newcommand{\di}{\mathrm{Diam}\,}
\newcommand{\supp}[1]{\mathrm{supp}\,{#1}}

\newcommand{\RE}{\mathcal{R}}
\newcommand{\We}[1]{\underline{\mathcal{WE}}_{#1}}
\newcommand{\DC}{\mathcal{D}_{C(\cdot)}}
\newcommand{\IM}{\mathrm{Im}\,}
\newcommand{\restr}{\mathrm{restr}}

\newcommand{\calL}{\mathcal{L}}

\newcommand{\bbR}{\mathbb{R}}

\newcommand{\diam}{\operatorname{diam}}

\theoremstyle{remark}

\newtheorem{rem}[thm]{Remark}        % \renewcommand{\therem}{}
\numberwithin{equation}{section}
\title[One dimensional RCD spaces]{Characterization of Low Dimensional $RCD^*(K,N)$ spaces}
\author{Yu Kitabeppu}
\address{Yu Kitabeppu, Kyoto University}
\email{\href{mailto:ybeppu@math.kyoto-u.ac.jp}{ybeppu@math.kyoto-u.ac.jp}}

\author{Sajjad Lakzian}
\address{ Sajjad Lakzian,
Mathematics Department,
Fordham University } 
\urladdr{\href{https://sites.google.com/site/sajjadlakzianmath/}{https://sites.google.com/site/sajjadlakzianmath/}}
\email{\href{mailto:lakzians@gmail.com}{lakzians@gmail.com}}

\thanks{Yu Kitabeppu is Partly supported by the Grant-in-Aid for JSPS Fellows, The Ministry of Education, Culture, Sports, Science and Technology, Japan; Sajjad Lakzian was partly supported by the Hausdorff postdoctoral fellowship at the Hausdorff Center for Mathematics and now is supported by PMC visiting assistant professorship at Fordham University}


\begin{document}
\maketitle
 \begin{abstract}
  In this paper, we give the characterization of metric measure spaces that satisfy synthetic lower Riemannian Ricci curvature bounds (so called $RCD^*(K,N)$ spaces) with \emph{non-empty} one dimensional regular sets. In particular, we prove that the class of Ricci limit spaces with $Ric \ge K$ and Hausdorff dimension $N$ and the class of $RCD^*(K,N)$ spaces coincide for $N < 2$ (They can be either complete intervals or circles). We will also prove a Bishop-Gromov type inequality ( that is ,roughly speaking, a converse to the L\'{e}vy-Gromov's isoperimetric inequality and was previously only known for Ricci limit spaces) which might be also of independent interest.
 \end{abstract}
 	\setcounter{tocdepth}{1}
 	\small
 	\tableofcontents
 	\normalsize
 	\addtocontents{toc}{~\hfill\textbf{Page}\par}

 %%%%%%%%%%%%%%%%%%%%%%%%%%%%%%%%%%%%%%%%%%
 
 						%% Section : Introduction %%
 
 %%%%%%%%%%%%%%%%%%%%%%%%%%%%%%%%%%%%%%%%%%
 %%%%%%%%%%%%%%%%%%%%%%%%%%%%%%%%%%%%%%%%%%
 
 \section{Introduction}
 
 In the past few decades, understanding Ricci limit spaces has been a central theme in geometric analysis. Ricci limit spaces are the metric spaces that are obtained as the pointed Gromov-Hausdorff limits of sequences of Riemannian manifolds with uniform lower Ricci curvature bounds. Studying Ricci limit spaces is a key in understanding the metric and measure properties of Riemannian manifold with lower Ricci curvature bound. A deep theory of these spaces has been developed over the years mostly by the work of Cheeger and Colding (see~\cite{CCwarped,CC1,CC2,CC3}). 
 %%%%%%%%%%%%%%%%%%%%%%%%%%%%%%%%%%%%%%%%%%
 %%%%%%%%%%%%%%%%%%%%%%%%%%%%%%%%%%%%%%%%%%
 
 A very interesting and still unanswered question regarding the Ricci limit spaces is whether they can be characterized solely based on their intrinsic metric (and measure) properties. For a Riemannian manifold $(M^n,g)$, a lower Ricci curvature bound can be characterized solely in terms of the metric measure properties of the induced metric measure space, $\left( M , d_g , dvol_{g} \right)$, where $d_{g}$ is the distance induced on $M^n$ by the Riemannian metric $g$. It is by now well-known that, $Ric_{M^n} \ge K$ is equivalent to metric measure space, $\left( M , d_g , dvol_{g}  \right)$, satisfying $CD(K,n)$ curvature-dimension conditions in the sense of Lott-Sturm-Villani (see the seminal papers~\cite{LV,Stmms1,Stmms2}). The class of $CD(K,N)$ spaces is actually much bigger than the class of Ricci limit spaces (of Riemannian manifolds with dimension at most $N$ and with $Ric \ge K$). In fact, there are Finsler manifolds that satisfy $CD(K,N)$ curvature-dimension conditions (see~Ohta~\cite{OhFinsler}) but from the work of Cheeger-Colding, we know that Finsler manifolds can not arise as Ricci limit spaces. 
%%%%%%%%%%%%%%%%%%%%%%%%%%%%%%%%%%%%%%%%%%
 %%%%%%%%%%%%%%%%%%%%%%%%%%%%%%%%%%%%%%%%%%
 
In order to exclude Finslerian spaces, Ambrosio-Gigli-Savar\'{e}~\cite{AGSRiem} have introduced the notion of dimension-free Riemannian lower Ricci bound for possibly non-compact metric measure spaces with finite measures. Afterwards, Ambrosio-Gigli-Mondino-Rajala extended this notion to the non-compact metric spaces with $\sigma$-finite measures~\cite{AGMR}. The dimensional Riemannian lower Ricci bound for metric measure spaces was later considered and investigated in Erbar-Kuwada-Sturm~\cite{EKS} and also independently in Ambrosio-Mondino-Savar\'{e}~\cite{AMS-2}. 
%%%%%%%%%%%%%%%%%%%%%%%%%%%%%%%%%%%%%%%%%%
 %%%%%%%%%%%%%%%%%%%%%%%%%%%%%%%%%%%%%%%%%%
 
 Roughly speaking, a $CD(K,N)$ metric measure space, $\left(X,d,m\right)$, is said to satisfy the Riemannian curvature-dimension conditions (for short, we will call it an $RCD(K,N)$ space) whenever the associated weak Sobolev space $W^{1,2}$ is a Hilbert space. When $W^{1,2}$ is a Hilbert space, the space is said to be \emph{infinitesimally Hilbertian}. In essence,  infinitesimal Hilbertianity means that the heat flow and the Laplacian on these spaces (defined in~\cite{AGSRiem} ) are Linear. It is readily verified that Ricci limit spaces are in fact infinitesimally Hilbertian. It is also a well-known fact that an infinitesimally Hilbertian Finsler manifold has to be a Riemannian manifold which is a result of the Cheeger energy being a quadratic form. It is yet not known whether every $RCD(K,N)$ space is a Ricci limit space. 
%%%%%%%%%%%%%%%%%%%%%%%%%%%%%%%%%%%%%%%%%%
 %%%%%%%%%%%%%%%%%%%%%%%%%%%%%%%%%%%%%%%%%%
 
Bacher-Sturm~\cite{BS} introduced reduced curvature-dimension conditions $CD^*(K,N)$ in order to get better \emph{local-to-global} and \emph{tensorization} properties. Every $CD(K,N)$ space is also $CD^*(K,N)$; conversely, every $CD^*(K,N)$ space is proven to be a $CD\left( K^* ,  N \right)$ space where $K^* = \frac{(N-1)K}{N}$ for $K\ge0$ (for $K < 0$, a suitable formula can be worked out for $K^*$, see Cavalletti~\cite{Cav-decomp} and Cavaletti-Sturm~\cite{CS-2} for more in this direction. In particular, $CD(0,N) = CD^*(0,N)$. As before, an infinitesimally Hilbertian $CD^*(K,N)$ space is said to be an $RCD^*(K,N)$ space. Recently, a structure theory for $RCD^*(K,N)$ spaces has been developed by Mondino-Naber~\cite{MN}. They prove that the tangent space is unique almost everywhere. Also from Gigli-Mondino-Rajala~\cite{GMR}, we know that almost everywhere, these unique tangent spaces are actually Euclidean namely isomorphic to $\left( \R^k , d_{Euc} , \mathcal{L} \right)$ ($k$ might vary point-wise).  
 %%%%%%%%%%%%%%%%%%%%%%%%%%%%%%%%%%%%%%%%%%
 %%%%%%%%%%%%%%%%%%%%%%%%%%%%%%%%%%%%%%%%%%
 
Our first goal in this paper is to characterize $RCD^*(K,N)$ spaces with 1-dimensional regular set $\RE_1$. The set $\RE_1$ consists of the points where the tangent space is unique and equal to $\R$ (for a precise definition of $\RE_1$, see Definition \ref{def:regular}). We use the structure theory developed by Mondino-Naber~\cite{MN} and arguments similar to Honda~\cite{HBG} to prove the following characterization theorem.
 %%%%%%%%%%%%%%%%%%%%%%%%%%%%%%%%%%%%%%
 \begin{thm}\label{thm:main-1}
 Let $(X,d,m)$ be an $RCD^*(K,N)$ space for $K\in\R$ and $N\in (1,\infty)$. Assume $X$ is not one point and $\supp m=X$. 
  The following are all equivalent to each other: 
  \begin{enumerate}
   \item $\RE_1\neq\emptyset$,
   \item $\RE_j=\emptyset$ for any $j\geq 2$,
   \item $m(\RE_j)=0$ for any $j\geq 2$,
   \item $X$ is isometric to $\R$, to $\R_{\geq 0}$, to $S^1(r):=\{x\in\R^2\;;\;\vert x\vert=r\}$ for $r>0$, or to $[0,l]$ for $l>0$.
  \end{enumerate}
   %%%%%%%%%%%%%%%%%%%%%%%%%%%%%%%%%%%%%%%%%%
 %%%%%%%%%%%%%%%%%%%%%%%%%%%%%%%%%%%%%%%%%%
Moreover 
%\red{when $m$ is locally a reference measure,} then, 
the measure $m$ is equivalent to the 1-dimensional Hausdorff measure $\mathcal{H}^1$ i.e. $m$  can be written in the form $m=e^{-f}\mathcal{H}^1$ for a $(K,N)$-convex function $f$ (see Definition \ref{defn:fconvex}). 
  In particular $\dim_H X\in\mathbb{Z}_{\geq 0}$ if $(X,d,m)$ is an $RCD^*(K,N)$ space that has $\RE_1\neq\emptyset$.
 \end{thm}
  %%%%%%%%%%%%%%%%%%%%%%%%%%%%%%%%%%%%%%%%%%
 %%%%%%%%%%%%%%%%%%%%%%%%%%%%%%%%%%%%%%%%%%
 A direct corollary is the following:
 \begin{cor}\label{cor:lessthan2}
  Let $(X,d,m)$ be an $RCD^*(K,N)$ space for $K\in\R$ and $N\in[1,2)$. Then the statements as in Theorem \ref{thm:main-1} hold. 
 \end{cor}
   %%%%%%%%%%%%%%%%%%%%%%%%%%%%%%%%%%%%%%%%%%
 %%%%%%%%%%%%%%%%%%%%%%%%%%%%%%%%%%%%%%%%%%
\begin{rem}
  On Ricci limit spaces, the conditions in Theorem \ref{thm:main-1} are also equivalent to $1\leq dim_H X<2$ (~\cite{Hlow,CN} ). So far we do not know whether an $RCD^*(K,N)$ space of the Hausdorff dimension $n<N$ has the regular set $\RE_k$, $n<k\leq N$ or not. 
 \end{rem}
 %%%%%%%%%%%%%%%%%%%%%%%%%%%%%%%%%%%%%%%%%%%
 %%%%%%%%%%%%%%%%%%%%%%%%%%%%%%%%%%%%%%%%%%%
 In order to further understand the behaviour of the measure, we first show the following important \emph{Bishop-Gromov type inequality} for $RCD^*(K,N)$ spaces that was previously known for Ricci limit spaces~\cite{HBG}. 
%%%%%%%%%%%%%%%%%%%%%%%%%%%%%%%%%%%%%%%%%%%
 %%%%%%%%%%%%%%%%%%%%%%%%%%%%%%%%%%%%%%%%%%%
\begin{thm}\label{thm:BG}
 Let $(X,d,m)$ be a metric measure space satisfying $RCD^*(K,N)$ condition and $m_{-1}$, the boundary measure. For any point $x_0\in X$ and any $t>0$, we have
  %%%%%%%%%%%%%%%%%%%%%%%%%%%%%%%%%%%%%%%%%%%%%%%%%%%%%%%%%%%%%%%%%%%%%%%%%%%%%%%%%%%%
   %%%%%%%%%%%%%%%%%%%%%%%%%%%%%%%%%%%%%%%%%%%%%%%%%%%%%%%%%%%%%%%%%%%%%%%%%%%%%%
 \begin{align}
  m_{-1}(\partial B_t(x_0))\leq 2 \cdot 5^{N-1} \cdot m(B_t(x_0))\frac{S_{K,N}(t)^{N-1}}{\int_0^tS_{K,N}(r)^{N-1}\,dr}.  \label{eq:BGforbdry}
 \end{align}
\end{thm}
%%%%%%%%%%%%%%%%%%%%%%%%%%%%%%%%%%%%%%%%%%%
 %%%%%%%%%%%%%%%%%%%%%%%%%%%%%%%%%%%%%%%%%%%
 Let $\We{1}$ be the set of points where there exists a tangent space of the form $\R \times W$ for some proper space, $W$ of strictly positive diameter (for a precise definition of $\We{1}$, see Definition~\ref{def:regular}). Using the Bishop-Gromov type inequality (Theorem~\ref{thm:BG}), we will prove the following. 
 %%%%%%%%%%%%%%%%%%%%%%%%%%%%%%%%%%%%%%%%%%%
 %%%%%%%%%%%%%%%%%%%%%%%%%%%%%%%%%%%%%%%%%%
\begin{prop}\label{cor:main2}
 Let
\begin{align}
 \mathcal{M}_{1}:=\left\{x\in X\;;\;\liminf_{r\rightarrow0}\frac{m(B_r(x))}{r}=0\right\}.\notag
\end{align}
Then, 
\be
	\We{1} \subset \mathcal{M}_1.\notag
\ee
and furthermore, if the modulus of continuity of $x\mapsto\frac{m\left(B_r(x)  \right)}{r}$ is independent of the choice of $r \ge 0$ then, $\mathcal{M}_1$ is closed.
\end{prop}
 %%%%%%%%%%%%%%%%%%%%%%%%%%%%%%%%%%%%%%%%%%
 %%%%%%%%%%%%%%%%%%%%%%%%%%%%%%%%%%%%%%%%%%

  \begin{rem}
 It has been brought to our attention that recently, a similar result for Ricci limit spaces has been proven in Chen~\cite{LinaChen}. The proof in Chen~\cite{LinaChen} heavily relies on the H\"older continuity of tangent cones along a minimal geodesic which is a result that is not available in our setting ($RCD^*(K,N)$ metric measure spaces). 
 \end{rem}
 %%%%%%%%%%%%%%%%%%%%%%%%%%%%%%%%%%%%%%

%%%%%%%%%%%%%%%%%%%%%%%%%%%%%%%%%%%%%%%%%%
 
 						%% Section : Preliminaries %%
 
 %%%%%%%%%%%%%%%%%%%%%%%%%%%%%%%%%%%%%%%%%%
 \section{Preliminaries}
 A \emph{metric measure space} is a triple $(X,d,m)$ consisting of a complete separable metric space, $(X,d)$, and  a locally finite complete positive Borel measure, $m$, that is, $m(B)<\infty$ for any bounded Borel set $B$ and $\supp m\neq\emptyset$.
  %%%%%%%%%%%%%%%%%%%%%%%%%%%%%%%%%%%%%%%%%%
  %%%%%%%%%%%%%%%%%%%%%%%%%%%%%%%%%%%%%%%%%%
  
 A curve $\gamma:[0,l]\rightarrow X$ is called a \emph{geodesic} if $d(\gamma(0),\gamma(l))=Length(\gamma)$. We call $(X,d)$ a \emph{geodesic space} if for any two points, there exists a geodesic connecting them. A metric space $(X,d)$ is said to be \emph{proper} if every bounded closed set in $X$ is compact. It is well-known that complete locally compact geodesic metric spaces are proper. 
 %%%%%%%%%%%%%%%%%%%%%%%%%%%%%%%%%%%%%%%%%%
  %%%%%%%%%%%%%%%%%%%%%%%%%%%%%%%%%%%%%%%%%%
  
 We denote the set of all Lipschitz functions in $X$ by $\LIP(X)$. For every $f\in\LIP(X)$, the local Lipschitz constant at $x$, $\vert Df\vert(x)$, is defined by
 %%%%%%%%%%%%%%%%%%%%%%%%%%%%%%%%%%%%%%%%%%
  %%%%%%%%%%%%%%%%%%%%%%%%%%%%%%%%%%%%%%%%%%
 \begin{align}
  \vert Df\vert(x):=\limsup_{y\rightarrow x}\frac{\vert f(x)-f(y)\vert}{d(x,y)}, \notag
 \end{align}
 when $x$ is not isolated, otherwise $\vert Df\vert(x):=\infty$. 
 %%%%%%%%%%%%%%%%%%%%%%%%%%%%%%%%%%%%%%%%%%
  %%%%%%%%%%%%%%%%%%%%%%%%%%%%%%%%%%%%%%%%%%
  
 The \emph{Cheeger energy} of a function $f\in L^2(X,m)$ is defined as 
 \begin{align}
  Ch(f):=\frac{1}{2}\inf\left\{\liminf_{n\rightarrow\infty}\int_X\vert Df_n\vert^2\,dm\;;\;f_n\in\LIP(X),\,f_n\rightarrow f\text{ in }L^2\right\}. \notag
 \end{align}
 %%%%%%%%%%%%%%%%%%%%%%%%%%%%%%%%%%%%%%%%%%
  %%%%%%%%%%%%%%%%%%%%%%%%%%%%%%%%%%%%%%%%%%
 Set $D(Ch):=\{f\in L^2(X,m)\;;\;Ch(f)<\infty\}$. It is known that for any $f\in D(Ch)$, there exists $\vert Df\vert_w\in L^2(X,m)$ such that $2Ch(f)=\int_X\vert Df\vert^2_w\,dm$. We say that $(X,d,m)$ is \emph{infinitesimally Hilbertian} if the Cheeger energy is a quadratic form. The infinitesimal Hilbertianity is equivalent to the Sobolev space $W^{1,2}(X,d,m):=\{f\in L^2\cap D(Ch)\}$ equipped with the norm $\Vert f\Vert^2_{1,2}:=\Vert f\Vert^2_{2}+2Ch(f)$ being a Hilbert space.
 %%%%%%%%%%%%%%%%%%%%%%%%%%%%%%%%%%%%%%%%%%%
   % subsection
 %%%%%%%%%%%%%%%%%%%%%%%%%%%%%%%%%%%%%%%%%%%
 \subsection{The curvature-dimension conditions}
   Let $(X,d,m)$ be a metric measure space and $\Pro(X)$, the set of all Borel probability measures. We denote by $\Pro_2(X)$, the set of all Borel probability measures with finite second moments. 
 %%%%%%%%%%%%%%%%%%%%%%%%%%%%%%%%%%%%%%%%%%
  %%%%%%%%%%%%%%%%%%%%%%%%%%%%%%%%%%%%%%%%%%
  
 For any $\mu_0,\mu_1\in\Pro_2(X)$, the $L^2$-Wasserstein distance is defined as
 %%%%%%%%%%%%%%%%%%%%%%%%%%%%%%%%%%%%%%%%%%
  %%%%%%%%%%%%%%%%%%%%%%%%%%%%%%%%%%%%%%%%%%
\small
 \begin{align}
  W_2(\mu_0,\mu_1):=\inf\left\{\int_{X\times X}d(x,y)^2\,dq(x,y)\;;\;q\text{ is a coupling between }\mu_0,\mu_1\right\}^{\frac{1}{2}}.\label{eq:L2W}
	 \end{align}
	 	\normalsize
	 %%%%%%%%%%%%%%%%%%%%%%%%%%%%%%%%%%%%%%%%%%
  %%%%%%%%%%%%%%%%%%%%%%%%%%%%%%%%%%%%%%%%%%
 A measure $q\in\Pro(X\times X)$ that realizes the infimum in (\ref{eq:L2W}) is called an \emph{optimal coupling between $\mu_0$ and $\mu_1$}. 
 %%%%%%%%%%%%%%%%%%%%%%%%%%%%%%%%%%%%%%%%%%
  %%%%%%%%%%%%%%%%%%%%%%%%%%%%%%%%%%%%%%%%%%
  
	 For every complete separable geodesic space, $(X,d)$, the $L^2$-Wasserstein space, $(\Pro_2(X),W_2)$, is also a complete separable geodesic space. We denote by $Geo(X)$, the space of all constant speed geodesics from $[0,1]$ to $(X,d)$ with the sup norm and by $e_t: Geo(X)\rightarrow X$, the evaluation map for each $t\in[0,1]$. It is known that any geodesic $(\mu_t)_{t\in[0,1]}\subset Geo(\Pro_2(X))$ can be lifted to a measure $\pi\in\Pro(Geo(X))$, so that $(e_t)_\sharp\pi=\mu_t$ for all $t\in[0,1]$. Given two probability measures $\mu_0,\mu_1\in\Pro_2(X)$, we denote by $OptGeo(\mu_0,\mu_1)$ the space of all probability measures $\pi\in\Pro(Geo(X))$ such that $(e_t)_{\sharp}\pi$ is a geodesic and $(e_0,e_1)_\sharp\pi$ is an optimal coupling between $\mu_0$ and $\mu_1$.
 %%%%%%%%%%%%%%%%%%%%%%%%%%%%%%%%%%%%%%%%%%
  %%%%%%%%%%%%%%%%%%%%%%%%%%%%%%%%%%%%%%%%%%
  
 For given $K\in\R$ and $N\in[1,\infty)$, the distortion coefficients, $\sigma^{(t)}_{K,N}(\theta)$, are defined by 
 %%%%%%%%%%%%%%%%%%%%%%%%%%%%%%%%%%%%%%%%%%
  %%%%%%%%%%%%%%%%%%%%%%%%%%%%%%%%%%%%%%%%%%
  
 \begin{align}
  \sigma^{(t)}_{K,N}(\theta):=\begin{cases}
                                               \infty&\text{if }K\theta^2\geq N\pi^2,\\
                                               \frac{\sin(t\theta\sqrt{K/N})}{\sin(\theta\sqrt{K/N})}&\text{if }0<K\theta^2<N\pi^2,\\
                                               t&\text{if }K\theta^2=0,\\
                                               \frac{\sinh(t\theta\sqrt{-K/N})}{\sinh(\theta\sqrt{-K/N})}&\text{if }K\theta^2<0.
                                              \end{cases} \notag
 \end{align}
  %%%%%%%%%%%%%%%%%%%%%%%%%%%%%%%%%%%%%%%%%%
  %%%%%%%%%%%%%%%%%%%%%%%%%%%%%%%%%%%%%%%%%%
 \begin{defn}[$(K,N)-$convexity of functions]\label{defn:fconvex}
 Suppose $\left( X,d  \right)$ is a geodesic space. A function $f: X \to \R \cup \{ \pm \infty \}$ is called $(K,N)-$convex if for any two points $x_0,x_1 \in X$ and a geodesic $x_t$, $0 \le t \le 1$ joining these points, one has
 %%%%%%%%%%%%%%%%%%%%%%%%%%%%%%%%%%%%%%%%%%
  %%%%%%%%%%%%%%%%%%%%%%%%%%%%%%%%%%%%%%%%%%
 \begin{align}
    &\exp\left(-\frac{1}{N}f(x_t)\right) \notag\\
    &\geq \sigma^{(1-t)}_{K,N}(d(x_0,x_1))\exp\left(-\frac{1}{N}f(x_0)\right)+\sigma^{(t)}_{K,N}(d(x_0,x_1))\exp\left(-\frac{1}{N}f(x_1)\right). \notag
   \end{align}
   %%%%%%%%%%%%%%%%%%%%%%%%%%%%%%%%%%%%%%%%%%
  %%%%%%%%%%%%%%%%%%%%%%%%%%%%%%%%%%%%%%%%%%
 \end{defn}
%%%%%%%%%%%%%%%%%%%%%%%%%%%%%%%%%%%%%%%%%%
  %%%%%%%%%%%%%%%%%%%%%%%%%%%%%%%%%%%%%%%%%% 
\begin{defn}[$CD^*(K,N)$ curvature-dimension conditions]
  Let $K\in \R$ and $N\in(1,\infty)$. 
  A metric measure space $(X,d,m)$ is said to be a $CD^*(K,N)$ space if for any two measures $\mu_0,\mu_1\in\Pro(X)$ with bounded support contained in $\supp m$ and with $\mu_0,\mu_1\ll m$, there exists a measure $\pi\in OptGeo(\mu_0,\mu_1)$ such that for every $t\in [0,1]$ and $N'\geq N$ one has,
  %%%%%%%%%%%%%%%%%%%%%%%%%%%%%%%%%%%%%%%%%%
  %%%%%%%%%%%%%%%%%%%%%%%%%%%%%%%%%%%%%%%%%% 
  \begin{align}
   -\int\rho_t^{1-\frac{1}{N'}}\,dm\leq -\int\sigma^{(1-t)}_{K,N'}(d(\gamma_0,\gamma_1))\rho_0^{-\frac{1}{N'}}+\sigma^{(t)}_{K,N'}(d(\gamma_0,\gamma_1))\rho^{-\frac{1}{N'}}_1\,d\pi(\gamma), \notag
  \end{align} 
  %%%%%%%%%%%%%%%%%%%%%%%%%%%%%%%%%%%%%%%%%%
  %%%%%%%%%%%%%%%%%%%%%%%%%%%%%%%%%%%%%%%%%% 
  where $\rho_t$ for $t\in[0,1]$ is the Radon-Nikodym derivative $d(e_t)_\sharp\pi/dm$. 
  %%%%%%%%%%%%%%%%%%%%%%%%%%%%%%%%%%%%%%%%%%
  %%%%%%%%%%%%%%%%%%%%%%%%%%%%%%%%%%%%%%%%%% 
 \end{defn}
  An infinitesimally Hilbertian metric measure space $(X,d,m)$ that also satisfies $CD^*(K,N)$ condition is called an $RCD^*(K,N)$ space. Erbar-Kuwada-Sturm give another characterization of $RCD^*(K,N)$ spaces. 
  %%%%%%%%%%%%%%%%%%%%%%%%%%%%%%%%%%%%%%%%%%
  %%%%%%%%%%%%%%%%%%%%%%%%%%%%%%%%%%%%%%%%%% 

Let the relative Entropy functional, $Ent(\cdot)$ be defined as
\be
	Ent(\mu) := \int_X \; \rho \log \rho \; dm
\ee 
whenever $\mu = \rho m$ is absolutely continuous with respect to the reference measure, $m$ and $\left(  \rho \log \rho \right)_+$ is integrable. Here $\mathcal{D}(Ent)$ denotes the set of all measures $\mu$ with $Ent(\mu)\in\R$.  
  %%%%%%%%%%%%%%%%%%%%%%%%%%%%%%%%%%%%%%%%%%
  %%%%%%%%%%%%%%%%%%%%%%%%%%%%%%%%%%%%%%%%%% 

\begin{defn}
   Let $(X,d,m)$ be a metric measure space. We say that $(X,d,m)$ satisfies the \emph{entropic curvature-dimension condition $CD^e(K,N)$ for $K\in\R$, $N\in(1,\infty)$} if for each pair $\mu_0,\mu_1\in\Pro_2(X,d,m)\cap \mathcal{D}(Ent)$, there exists a constant speed geodesic $(\mu_t)_{t\in[0,1]}$ connecting $\mu_0$ to $\mu_1$ such that for all $t\in[0,1]$:
   %%%%%%%%%%%%%%%%%%%%%%%%%%%%%%%%%%%%%%%%%%
  %%%%%%%%%%%%%%%%%%%%%%%%%%%%%%%%%%%%%%%%%% 
   \begin{eqnarray}
   &&\exp\left(-\frac{1}{N}Ent(\mu_t)\right)\ge \notag \\ 
    &&\sigma^{(1-t)}_{K,N}\left(W_2(\mu_0,\mu_1)\right)\exp\left(-\frac{1}{N}Ent(\mu_0)\right)
    +\sigma^{(t)}_{K,N}\left(W_2(\mu_0,\mu_1)\right)\exp\left(-\frac{1}{N}Ent(\mu_1)\right).\notag
      \end{eqnarray}
   %%%%%%%%%%%%%%%%%%%%%%%%%%%%%%%%%%%%%%%%%%
  %%%%%%%%%%%%%%%%%%%%%%%%%%%%%%%%%%%%%%%%%% 
  \end{defn}
  %%%%%%%%%%%%%%%%%%%%%%%%%%%%%%%%%%%%%%%%%%
  %%%%%%%%%%%%%%%%%%%%%%%%%%%%%%%%%%%%%%%%%% 
  \begin{thm}[Theorem 3.17 in~\cite{EKS}]
   Let $(X,d,m)$ be an infinitesimally Hilbertian metric measure space. Then $(X,d,m)$ is a $CD^*(K,N)$ space for $K\in\R$, $N\in (1,\infty)$ if and only if $(X,d,m)$ is a $CD^e(K,N)$ space. 
  \end{thm}
  %%%%%%%%%%%%%%%%%%%%%%%%%%%%%%%%%%%%%%%%%%
  %%%%%%%%%%%%%%%%%%%%%%%%%%%%%%%%%%%%%%%%%% 
 Since we will use the definition of the dimension-less curvature-dimension conditions (namely, $CD(K,\infty)$ conditions) in a few places in this paper, we will recall it here: 
  \begin{defn}
   Let $(X,d,m)$ be metric measure space. $(X,d,m)$ is said to satisfy the $CD(K,\infty)$ condition if 
   for any $\mu_0,\mu_1 \in\Pro_2(X)\cap\mathcal{D}(Ent)$, there exists a geodesic $(\mu_t)_{t\in[0,1]}$ connecting them such that 
   \begin{align}
    Ent(\mu_t)\leq (1-t)Ent(\mu_0)+tEnt(\mu_1)-\frac{K}{2}W_2^2(\mu_0,\mu_1)\label{def:infty:eq}
   \end{align}
   holds for any $t\in [0,1]$. We say that $(X,d,m)$ satisfies the \emph{strong} $CD(K,\infty)$ condition if (\ref{def:infty:eq}) holds for any geodesic. Moreover $(X,d,m)$ is called an $RCD(K,\infty)$ space if it is infinitesimally Hilbertian and a $CD(K,\infty)$ space. 
  \end{defn}

%%%%%%%%%%%%%%%%%%%%%%%%%%%%%%%%%%%%%%%%%%
  %%%%%%%%%%%%%%%%%%%%%%%%%%%%%%%%%%%%%%%%%% 
An important property of $CD^*(K,N)$ spaces is that the disintegration of the given measure with respect to the radial distance function, can be represented by the one dimensional Lebesgue measure. This fact will be used in the proof of Lemma~\ref{lem:lebesgueae} which in turn is essential in the proof of the characterization theorem. The precise definition and the proof can be found in \cite{CS-2}. 
%%%%%%%%%%%%%%%%%%%%%%%%%%%%%%%%%%%%%%%%%%
  %%%%%%%%%%%%%%%%%%%%%%%%%%%%%%%%%%%%%%%%%% 
  
  \begin{prop}[Disintegration of measure, Cavalletti-Sturm~\cite{CS-2}*{Section 3}]\label{prop:disinteg}
   Under the $CD^*(K,N)$ condition for $K\in\R, N\in(1,\infty)$, for fixed $o\in X$, we are able to disintegrate the given measure $m$ by 
   \begin{align}
    m=\int m_r\,\mathcal{L}^1(dr),
   \end{align}
  where, $m_r$ is a Borel measure supported on the set $\{x\in X\;;\; d(x,o)=r\} = \mathbf{r}^{-1}(r)$ (in which, $\mathbf{r}(\cdot) := d(o , \cdot)$ is the distance function from $o$).  
  \end{prop}

 %%%%%%%%%%%%%%%%%%%%%%%%%%%%%%%%%%%%%%%%%%
     %subsection : convergence
 %%%%%%%%%%%%%%%%%%%%%%%%%%%%%%%%%%%%%%%%%%
 \subsection{Convergence of pointed metric measure spaces}
 A pointed metric measure space is a quadruple $(X,d,m,\bar x)$, comprised of a metric measure space, $(X,d,m)$, and a given reference point $\bar x\in \supp m$. Two pointed metric measure spaces $(X_1,d_1,m_1,\bar x_1)$ and $(X_2,d_2,m_2,\bar x_2)$ are \emph{isomorphic} to each other if there exists an isometry $T:\supp m_1\rightarrow \supp m_2$ such that $T_\sharp m_1=m_2$ and $T\bar x_1=\bar x_2$. We say that a pointed metric measure space, $(X,d,m,\bar x)$, is \emph{normalised} if $\int_{B_1(\bar x)}1-d(\cdot,\bar x)\,dm=1$. 
A measure $m$ is said to be \emph{doubling} if
  %%%%%%%%%%%%%%%%%%%%%%%%%%%%%%%%%%%%%%%%%% 
 %%%%%%%%%%%%%%%%%%%%%%%%%%%%%%%%%%%%%%%%%% 
 \begin{align}
  0<m(B_{2r}(x))\leq C(R)m(B_r(x)),\label{eq:doubling}
 \end{align}
  %%%%%%%%%%%%%%%%%%%%%%%%%%%%%%%%%%%%%%%%%% 
 %%%%%%%%%%%%%%%%%%%%%%%%%%%%%%%%%%%%%%%%%%
 holds for any $0<r\leq R$ and $x\in \supp m$. We denote by $\M_{C(\cdot)}$ the class of all normalised pointed metric measure spaces satisfying (\ref{eq:doubling}) for a given non-decreasing function $C:(0,\infty)\rightarrow(0,\infty)$. We have the following compactness and metrizability theorem.
  %%%%%%%%%%%%%%%%%%%%%%%%%%%%%%%%%%%%%%%%%% 
 %%%%%%%%%%%%%%%%%%%%%%%%%%%%%%%%%%%%%%%%%%
 \begin{thm}[~\cite{GMS,MN}]
  Let $C:(0,\infty)\rightarrow (0,\infty)$ be a non-decreasing function. Then, there exists a distance function $\DC$ on $\M_{C(\cdot)}$ such that $(\M_{C(\cdot)},\DC)$ becomes a compact metric space. Moreover the topology induced from $\DC$ coincides with the one defined by the pointed measured Gromov-Hausdorff convergence on $\M_{C(\cdot)}$. 
 \end{thm} 
  %%%%%%%%%%%%%%%%%%%%%%%%%%%%%%%%%%%%%%%%%% 
 %%%%%%%%%%%%%%%%%%%%%%%%%%%%%%%%%%%%%%%%%%
 For a given pointed metric measure space $(X,d,m,x)$ with $x\in \supp m$ and $r\in(0,1)$, we associate the rescaled and normalised pointed metric measure space $(X,d_r,m^x_r,x)$, where $d_r:=d/r$ and, 
  %%%%%%%%%%%%%%%%%%%%%%%%%%%%%%%%%%%%%%%%%% 
 %%%%%%%%%%%%%%%%%%%%%%%%%%%%%%%%%%%%%%%%%%
 \begin{align}
  m_r^x:=\left(\int_{B_r(x)}1-\frac{1}{r}d(x,\cdot)\,dm\right)^{-1}m. \notag
 \end{align}
  %%%%%%%%%%%%%%%%%%%%%%%%%%%%%%%%%%%%%%%%%% 
 %%%%%%%%%%%%%%%%%%%%%%%%%%%%%%%%%%%%%%%%%%
 \begin{defn}[Tangent space]
  Let $(X,d,m)$ be a metric measure space and $x\in \supp m$. A pointed metric measure space $(Y,d_Y,m_Y,y)$ is called a tangent to $(X,d,m)$ at $x\in X$ if there exists a sequence of positive numbers $r_i\downarrow 0$ such that $(X,d_{r_i},m^x_{r_i},x)\rightarrow (Y,d_Y,m_Y,y)$ as $i\rightarrow \infty$ in the pointed measured Gromov-Haudsdorff topology. We denote by $Tan(X,d,m,x)$ the collection of all tangents to $(X,d,m)$ at $x\in \supp m$. 
 \end{defn}
  %%%%%%%%%%%%%%%%%%%%%%%%%%%%%%%%%%%%%%%%%% 
 %%%%%%%%%%%%%%%%%%%%%%%%%%%%%%%%%%%%%%%%%%
 There exists a non-decreasing function $C:(0,\infty)\rightarrow (0,\infty)$ depending only on $K,N$ such that all $RCD^*(K,N)$ spaces belong to $\M_{C(\cdot)}$ (for instance, see Sturm~\cite{Stmms1}). Hence for $RCD^*(K,N)$ spaces, convergence with respect to $\DC$ and that with respect to the pointed measured Gromov-Hausdorff topology coincide.
  %%%%%%%%%%%%%%%%%%%%%%%%%%%%%%%%%%%%%%%%%% 
 %%%%%%%%%%%%%%%%%%%%%%%%%%%%%%%%%%%%%%%%%% 
 \begin{thm}[~\cite{GMS}]
  Let $K\in \R$ and $N\in(1,\infty)$. Then the class of normalized $RCD^*(K,N)$ pointed metric measure spaces is closed (and therefore compact) with respect to $\DC$. 
 \end{thm}
  %%%%%%%%%%%%%%%%%%%%%%%%%%%%%%%%%%%%%%%%%% 
 %%%%%%%%%%%%%%%%%%%%%%%%%%%%%%%%%%%%%%%%%%
 It is easy to see that for any $\lambda>0$, $(X,\lambda d,m)$ satisfies the $RCD^*(\lambda^{-2}K,N)$ condition provided that $(X,d,m)$ is an $RCD^*(K,N)$ space. This will imply that $Tan(X,d,m,x)$ consists of $RCD^*(0,N)$ spaces for any point $x\in \supp m$.
  %%%%%%%%%%%%%%%%%%%%%%%%%%%%%%%%%%%%%%%%%% 
 %%%%%%%%%%%%%%%%%%%%%%%%%%%%%%%%%%%%%%%%%%
 
 One key tool that is reminiscent of smooth Riemannian setting is the splitting theorem: 
 \begin{thm}[Splitting theorem, Gigli~\cite{Gsplit,Goverview}]\label{thm:splitting}
  Let $(X,d,m)$ be an $RCD^*(0,N)$ space with $1\leq N<\infty$. Suppose that $\supp(m)$ contains a line. Then $(X,d,m)$ is isomorphic to $(X'\times\R,d'\times d_E,m'\times\mathcal{L}^1)$, where $d_E$ is the Euclidean distance, $\mathcal{L}^1$ the Lebesgue measure and $(X',d',m')$ is an $RCD^*(0,N-1)$ space if $N\geq 2$ and a singleton if $1\leq N<2$. 
 \end{thm}
  %%%%%%%%%%%%%%%%%%%%%%%%%%%%%%%%%%%%%%%%%% 
 %%%%%%%%%%%%%%%%%%%%%%%%%%%%%%%%%%%%%%%%%%
From the work of Gigli-Mondino-Rajala~\cite{GMR} and Mondino-Naber~\cite{MN}, it follows that:
 \begin{thm}[\cite{MN},\cite{GMR}]
  Let $(X,d,m)$ be an $RCD^*(K,N)$ space. Then $m$-a.e. $x\in \supp m$, there exists an integer $1\leq k\leq N$ such that $Tan(X,d,m,x)=\{(\R^k,d_E,\mathcal{L}^k,0^k)\}$, where $\mathcal{L}^k$ is the normalized $k$-dimensional Lebesgue measure. 
 \end{thm}
   %%%%%%%%%%%%%%%%%%%%%%%%%%%%%%%%%%%%%%%%%% 
 %%%%%%%%%%%%%%%%%%%%%%%%%%%%%%%%%%%%%%%%%%
 %    subsection : Essentially non-branching
 %%%%%%%%%%%%%%%%%%%%%%%%%%%%%%%%%%%%%%%%%%%

 \subsection{Essentially non-branching property}
 Let $\restr_s^t : Geo(X)\rightarrow Geo(X)$ be a restriction map, which is defined as $\restr_s^t(\gamma)_r:=\gamma_{(1-r)s+rt}$ for $r\in[0,1]$. A subset $\Gamma\subset Geo(X)$ is called \emph{non-branching} if for any $\gamma,\gamma'\in\Gamma$, $\restr_0^t(\gamma)=\restr_0^t(\gamma')$ for some $t\in(0,1]$ implies $\gamma=\gamma'$. Rajala-Sturm~\cite{RS} have proven that branching geodesics in $RCD(K,N)$ spaces are rare. Here we state a special case of their main theorem in~\cite{RS}. 
 \begin{thm}
  Let $(X,d,m)$ be an $RCD(K,\infty)$ space. Then for any $\mu_0,\mu_1\in\Pro_2(X)$ with $\mu_i\ll m$, and any $\pi\in OptGeo(\mu_0,\mu_1)$, there exists a non-branching subset $\Gamma\subset Geo(X)$ such that $\pi(\Gamma)=1$. 
 \end{thm}
 %%%%%%%%%%%%%%%%%%%%%%%%%%%%%%%%%%%%%%%%%% 
 %%%%%%%%%%%%%%%%%%%%%%%%%%%%%%%%%%%%%%%%%%
 \begin{lem}\label{lem:lebesgueae}
  Let $(X,d,m)$ be an $RCD^*(K,N)$ space for $K\in\R, N\in(1,\infty)$ with $\supp m=X$. Let $x,y,z\in X$ be three points such that $d(x,y)=d(x,z)=:l>0$ and $d(y,z)>0$. Set two geodesics $\gamma_1,\gamma_2:[0,1]\rightarrow X$ connecting $x$ and $y$, $x$ and $z$ respectively. Assume that $B_{r_0}(y)\cap B_{r_0}(z)=\emptyset$ for a small $r_0>0$. Let $A, B$ be two Borel sets defined by 
   %%%%%%%%%%%%%%%%%%%%%%%%%%%%%%%%%%%%%%%%%% 
 %%%%%%%%%%%%%%%%%%%%%%%%%%%%%%%%%%%%%%%%%%
  \begin{align}
   A:=\left\{w\in B_{r_0}(y)\;;\;d(x,w)\leq l\right\},\notag\\
   B:=\left\{w\in B_{r_0}(z)\;;\;d(x,w)\leq l\right\}.\notag
  \end{align}
  Then, $m(A)m(B)>0$ and $m_r(A)m_r(B)>0$ for $\mathcal{L}^1$-a.e. $r\in(l-r_0,l)$ where, $m_r$ is the measure obtained from $m$ via disintegration (as in Proposition~\ref{prop:disinteg}). 
 \end{lem}
  %%%%%%%%%%%%%%%%%%%%%%%%%%%%%%%%%%%%%%%%%% 
 %%%%%%%%%%%%%%%%%%%%%%%%%%%%%%%%%%%%%%%%%%
 \begin{proof}
Since $\supp m=X$, every open ball is of positive measure. We are able to take points $y':=\gamma_1(1-r_0/2l)$ and $z':=\gamma_2(1-r_0/2l)$ so that $B_{r_0/2}(y')\subset A$ and $B_{r_0/2}(z')\subset B$. Thus $m(A)\geq m(B_{r_0}(y'))>0$ and $m(B)\geq m(B_{r_0/2}(z'))>0$ holds. By using the disintegration of $m$ (see Proposition \ref{prop:disinteg}), we have 
   %%%%%%%%%%%%%%%%%%%%%%%%%%%%%%%%%%%%%%%%%% 
 %%%%%%%%%%%%%%%%%%%%%%%%%%%%%%%%%%%%%%%%%%
  \begin{align}
   m(A)=\int_{l-r_0}^lm_r(A)\,\mathcal{L}^1(dr)>0,\notag\\
   m(B)=\int_{l-r_0}^lm_r(B)\,\mathcal{L}^1(dr)>0.\notag
  \end{align} 
   %%%%%%%%%%%%%%%%%%%%%%%%%%%%%%%%%%%%%%%%%% 
 %%%%%%%%%%%%%%%%%%%%%%%%%%%%%%%%%%%%%%%%%%
  Suppose there exists a measurable subset $I\subset (l-r_0,l)$ with $\mathcal{L}^1(I)>0$ such that $m_r(A)=0$ for any $r\in I$. The Claim~\ref{clm:closed-interval} below shows that, in virtue of the measure contraction property, this implies that $m_r(A)=0$ for a.e. $r \in I'$ where $I'$ is a closed interval with positive length.
  %%%%%%%%%%%%%%%%%%%%%%%%%%%%%%%%%%%%%%%%%% 
 %%%%%%%%%%%%%%%%%%%%%%%%%%%%%%%%%%%%%%%%%%
 
  \par Therefore we are able to find a point $\tilde{y}\in \IM(\gamma_1)$ and a small number $\eta>0$ such that $\{l:=d(x,w)\in\R\;;\;w\in B_{\eta}(\tilde{y})\}\subset I'$. Hence 
  \begin{align}
   0&<m(B_{\eta}(\tilde{y}))=\int_{l-r_0}^{l}m_r(B_{\eta}(\tilde{y}))\mathcal{L}^1(dr)\notag\\
   &=\int_{I'}m_r(B_{\eta}(\tilde{y}))\,\mathcal{L}^1(dr)\notag\\
   &\leq \int_{I'}m_r(A)\,\mathcal{L}^1(dr)=0.\notag
  \end{align}
  This is a contradiction.

 \end{proof}

   %%%%%%%%%%%%%%%%%%%%%%%%%%%%%%%%%%%%%%%%%% 
 %%%%%%%%%%%%%%%%%%%%%%%%%%%%%%%%%%%%%%%%%%
%%%%%%%%%%%%%%%%%%%%%%%%%%%%%%%%%%%%%%%%%% 
 %%%%%%%%%%%%%%%%%%%%%%%%%%%%%%%%%%%%%%%%%%
\begin{clm}\label{clm:closed-interval}
   Let $I$ be the set,
  \be
  I := \left\{ r \in \left( l-r_0 , l   \right)  :    m_r(A) = 0  \right\}.\label{def:interval}
  \ee 
 Then, if $\mathcal{L}^1(I) > 0$, there exists a closed interval $I'$ with $\mathcal{L}^1(I')>0$ such that $\mathcal{L}^1(I'\setminus I)=0$. 
  \end{clm}
   %%%%%%%%%%%%%%%%%%%%%%%%%%%%%%%%%%%%%%%%%% 
 %%%%%%%%%%%%%%%%%%%%%%%%%%%%%%%%%%%%%%%%%%
\begin{proof}
We will use the regularity of the Lebesgue measure along with the \emph{Measure Contraction Property} to find such a closed interval.
 %%%%%%%%%%%%%%%%%%%%%%%%%%%%%%%%%%%%%%%%%% 
 %%%%%%%%%%%%%%%%%%%%%%%%%%%%%%%%%%%%%%%%%%  
By the regularity of the Lebesgue measure, for any $\epsilon > 0$, one can find a closed set $C$ and an open set $U$ with $C \subset I \subset U$ and such that $\mathcal{L}^1(U \setminus C) < \epsilon$. First of all, this means that we can assume $I$ is closed (otherwise, replace it with $C$ and notice that $C$ has positive measure for $\epsilon$ small enough).
 %%%%%%%%%%%%%%%%%%%%%%%%%%%%%%%%%%%%%%%%%% 
 %%%%%%%%%%%%%%%%%%%%%%%%%%%%%%%%%%%%%%%%%%
\par Claim~\ref{clm:MCP-Lebesgue} below shows that the measure contraction property implies that the set, $I$, is invariant under dilations (in a suitable sense that will be made clear in below). Let  $\mathbf{r}(\cdot):=d(x,\cdot)$ be the distance function from $x$. 
 %%%%%%%%%%%%%%%%%%%%%%%%%%%%%%%%%%%%%%%%%% 
 %%%%%%%%%%%%%%%%%%%%%%%%%%%%%%%%%%%%%%%%%%
  \begin{clm}\label{clm:MCP-Lebesgue} Suppose $J \subset \R$ is any measurable subset with $\mathcal{L}^1(J) > 0$ and
 \be
 	\int_J \,m_r(A)\mathcal{L}^1(dr)>0, \notag %\label{eq:MCP-Lebesque-1}, \notag
 \ee
then for any $0 < t \le 1$ with $A_{tJ}:=A\cap \mathbf{r}^{-1}(tJ)\neq\emptyset$, one has
\be
	\mathcal{L}^1 \left( (tJ) \setminus I  \right) > 0. \notag
\ee
In other words, if $\mathcal{L}^1(J\setminus I)>0$, then for any $0 < t  \le 1$, one has $\mathcal{L}^1(t J \setminus I) >0 $ when $A_{tJ}\neq \emptyset$. 
\end{clm}
 %%%%%%%%%%%%%%%%%%%%%%%%%%%%%%%%%%%%%%%%%% 
 %%%%%%%%%%%%%%%%%%%%%%%%%%%%%%%%%%%%%%%%%%
\begin{proof}
 Let $(X,d,m)$ be an $RCD^*(K,N)$ space for $K\in\R$, $N\in(1,\infty)$. Take two distinct points $x$ and $y$ with $d(x,y)=l$. We denote a geodesic connecting $x$ to $y$ by $\gamma^1$. Let $r_0>0$ be a positive number such that $B_{r_0}(y)\cap B_{r_0}(x)=\emptyset$. We disintegrate $m$ with respect to the distance function, $\mathbf{r}(\cdot):=d(x,\cdot)$, that is, 
\begin{align}
 m=\int_{\bbR_{\geq 0}}m_r\,\calL^1(dr).\notag
\end{align}
 %%%%%%%%%%%%%%%%%%%%%%%%%%%%%%%%%%%%%%%%%%
For $J \subset \R$ and $V\subset X$, let $V_J:=\{w\in V\;;\;d(x,w)\in J\}$. Note that for any measurable subset $V\subset X$, if $m_r(V)>0$ for a.e. $r\in J$ with $\calL^1(J)>0$ then,  $m(V_J)>0$ and obviously, $m(V)>0$. 
 %%%%%%%%%%%%%%%%%%%%%%%%%%%%%%%%%%%%%%%%%%
\par Now, let  $\,I\,$  be the measurable subset defined by (\ref{def:interval}) and assume $\mathcal{L}^1(I)>0$. %assume that there exists a measurable subset $I \subset (l-r_0,l)$ with $\calL^1(I)>0$ such that $m_r(A)=0$ for any $r\in I$. 
 Suppose a measurable subset $J \subset (l-r_0,l)$ with $\calL^1(J)>0$ satisfies $m(A_J)>0$. Let $\tau\in(0,1)$ be a number for which, $\calL^1((\tau J)\setminus I)=0$ and $A_{\tau J}\neq \emptyset$. Without loss of generality, we may assume $m_r(A_{J})=m_r(A)>0$ for all $r\in J$. 
%%%%%%%%%%%%%%%%%%%%%%%%%%%%%%%%%%%%%%%%%%
\par Let $\pi\in OptGeo(\mu,\delta_x)$, where $\mu:=\chi_{A_J}m/(m(A_J))\in\Pro(X)$. Note that by construction, $\mu\ll m$. Hence we are able to find a map $T_t:X\rightarrow X$ such that $(T_t)_*\mu=\mu_t=(e_t)_*\pi$, which is a geodesic from $\mu$ to $\delta_x$ (see Gigli-Rajala-Sturm~\cite{GRS}*{Theorem 1.1}). Since $\calL^1((\tau J)\setminus I)=0$ (i.e. $\tau J$ is a subset of $I$ in a.e. sense), we must have $m_r(A)=0$ for a.e. $r\in \tau J$. Accordingly,  
 \begin{align}
  m(A_{\tau J}\cap T_{1-\tau}(A_{ J}))\leq m(A_{\tau J})=\int_{\tau J}m_r(A)\,dr=0.\label{eq:1}
 \end{align}
 
%%%%%%%%%%%%%%%%%%%%%%%%%%%%%%%%%%%%%%%%%%
%%%%%%%%%%%%%%%%%%%%%%%%%%%%%%%%%%%%%%%%%%
Now, we consider two different cases:
%%%%%%%%%%%%%%%%%%%%%%%%%%%%%%%%%%%%%%%%%%
%%%%%%%%%%%%%%%%%%%%%%%%%%%%%%%%%%%%%%%%%%
\par \textbf{Case I}: 
\par Suppose there exists a measurable subset $B\subset A_J$ with $m(B)>0$ (hence $(e_0)_*\pi(B)=\mu(B)=m(B)/m(A_J)>0$) such that for $\pi$-a.e. geodesic $c^w$ connecting $w\in B$ to $x$, one has $c^w_{1-\tau}\in A$ which readily implies that $c^w_{1-\tau}\in A_{\tau J}$. By the MCP condition, we have $m(T_{1-\tau}(B))>0$. More precisely, since $(e_0)_*\pi(B)$ is positive, so is $(e_{1-\tau})_*\pi(T_{1-\tau}(B))$. This means that $m(A_{\tau J}\cap T_{1-\tau}(A_{ J}))\geq C(e_{1-\tau})_*\pi(T_{1-\tau}(B)) > 0$ which contradicts (\ref{eq:1}).
%%%%%%%%%%%%%%%%%%%%%%%%%%%%%%%%%%%%%%%%%%
%%%%%%%%%%%%%%%%%%%%%%%%%%%%%%%%%%%%%%%%%%
\par \textbf{Case II}: 
\par Suppose for any measurable subset $B\subset A_J$ with $m(B)>0$ one has for $\pi$-a.e. geodesic $c^w$ connecting $w\in B$ to $x$, $c^w_{1-\tau}\in A^c$. This implies that for $\pi$-a.e. $c\in Geo(X)$, one has $c_{1-\tau}\in A^c$. Recall that we denote the geodesic connecting from $x$ to $y$ by $\gamma^1$. 
%%%%%%%%%%%%%%%%%%%%%%%%%%%%%%%%%%%%%%%%%%
%%%%%%%%%%%%%%%%%%%%%%%%%%%%%%%%%%%%%%%%%%
We claim that there exists $s_0>0$ such that $\gamma^1_{s_0}\in A_{\tau J}$ (namely, $\gamma_1$ intersects $A_{\tau J}$). Indeed, suppose $\gamma^1_s\notin A_{\tau J}$ for all $s\in[0,1]$ (equivalently, $\gamma^1_s\in A^c$ whenever $d(x,\gamma^1_s)\in\tau J$). By the assumption $A_{\tau J}\neq \emptyset$, there exists a point $w\in A_{\tau J}$. It is obvious that $d(x,w)\in(l-r_0,l)$ for such $w\in A_{\tau J}$. Since $\gamma^1$ is a geodesic from $x$ to $y$, we can find a point $\gamma^1_s$ such that $d(x,\gamma^1_s)=d(x,w)\in \tau J$. However, $\gamma^1_s\notin A_{\tau J}$ means $d(x,\gamma^1_s)<l-r_0$ which is a contradiction.

%%%%%%%%%%%%%%%%%%%%%%%%%%%%%%%%%%%%%%%%%%
%%%%%%%%%%%%%%%%%%%%%%%%%%%%%%%%%%%%%%%%%% 
 %Since $d(x,w)\in (l-r_0,l)$, we can take $\gamma_s^1$ with $d(x,\gamma_s^1)=d(x,w)$. We just define $\gamma^1_s$ as a point $d(x,\gamma^1_s)=\tau d(x,\gamma^1_t)$. Since we now assume that \textcolor{red}{$\gamma^1_s\notin A_{\tau J}$}, one obtain
 %\begin{align}
 % d(x,c^w_{1-\tau})+d(c^w_{1-\tau},y)=d(x,\gamma_s^1)+d(c^w_{1-\tau},y)<d(x,\gamma_s^1)+r_0<d(x,y),\notag
 %\end{align}
 %which contradicts the triangle inequality. 
We have $\gamma_{s_0}^1\in A_{\tau J}$. Since $\calL^1(J)>0$, we may assume that $\inf\tau J<d(x,\gamma^1_{s_0})$ and in particular $0< d(y , \gamma^1_{s_0}) < r_0$, which also implies $d(A^c,\gamma^1_{s_0})>0$. %\textcolor{red}{Indeed, if not, we can repeat the above argument by replacing $r_0$ with a smaller radius, $r_1$ and find an $s_0$ as in above $<======$ Is this correct??????}. 
%\textcolor{cyan}{It is correct. If we say more rigorously, find a positive number $\kappa>0$ with $\mathcal{L}^1([l-r_0,l-r_0+\kappa])<\mathcal{L}^1(J\cap [l-r_0,l-r_0+\kappa])/2$, and repeat the same argument for $J\cap [l-r_0+\kappa,l]$. How do you think? }\\
This is true since by the Lebesgue density theorem, at almost every point $t\in\tau J$, one has
\be
\lim_{\kappa\rightarrow 0} \frac{\mathcal{L}^1(\tau J\cap(t-\kappa,t+\kappa))}{\mathcal{L}^1((t-\kappa,t+\kappa))}=1. \notag
\ee
Hence, taking a Lebesgue point $t\in\tau J$ greater than $\inf\tau J$, and repeating the above argument for $J' = \tau J \cap \left( t-\epsilon, t+\epsilon    \right) = \tau J'' $ with $\epsilon < t - \inf\tau J $,  we are able to take a point $\gamma^1_{s_0}$ satisfying $\inf\tau J<d(x,\gamma^1_{s_0})$.
 
%%%%%%%%%%%%%%%%%%%%%%%%%%%%%%%%%%%%%%%%%%
%%%%%%%%%%%%%%%%%%%%%%%%%%%%%%%%%%%%%%%%%%
Let $s_0 = t_0 \tau$, then obviously, $\gamma^1_{t_0}\in A_J$. Without loss of generality, we may assume that $t_0 =d(x,\gamma^1_{t_0})\in J$ is a Lebesgue point in $J$ (otherwise, one can repeat the above arguments by replacing $J$ with its Lebesgue points). Let $\xi>0$ be a positive number such that 
%%%%%%%%%%%%%%%%%%%%%%%%%%%%%%%%%%%%%%%%%%
\be
2\xi< \min\left\{ r_0-d(\gamma^1_{s_0},y)  , t_0 - s_0 \right\} . \notag 
\ee
%%%%%%%%%%%%%%%%%%%%%%%%%%%%%%%%%%%%%%%%%%
Consider a ball $B_{\xi}(\gamma^1_{t_0}) \subset A$. By the construction, $B_{\xi}(\gamma^1_{s_0})\cap B_{\xi}(\gamma^1_{t_0})=\emptyset$. Since $B_{ \xi}(\gamma^1_{t_0}) \subset A$ and $t_0 = d(x,\gamma^1_{t_0})\in J$ is a Lebesgue point in $J$, we have $m(A_J\cap B_{\xi}(\gamma^1_{s_0}))>0$. This implies $(e_0)_*\pi(A_J\cap B_{\xi}(\gamma^1_{s_0}))>0$. By the assumption, for $\pi$-a.e. $c\in Geo(X)$, $c_{1-\tau}\in A^c$. That is, for $(e_0)_*\pi=m$-a.e. $w^J\in A_J\cap B_{\xi}(\gamma^1_{s_0})$, one has $w:=T_{1-\tau}(w^J)\in A^c$.
%%%%%%%%%%%%%%%%%%%%%%%%%%%%%%%%%%%%%%%%%%
%%%%%%%%%%%%%%%%%%%%%%%%%%%%%%%%%%%%%%%%%%
\par Note that for all $w^J\in A_J\cap B_{\xi}(\gamma^1_{s_0})$ and their corresponding $w:=T_{1-\tau}(w^J)\in A^c$, one has
\begin{eqnarray}
d \left(w^J,w  \right) = (1 - \tau) d(x , w^J) &\le& (1 - \tau) (d(x,y)t_0 + \xi) \notag\\ &\le& d(x,y)(t_0 - \tau t_0) + \xi =d(x,y)( t_0 - s_0) + \xi \notag\\ &=& d(\gamma^1_{t_0},\gamma^1_{s_0})+\xi. \notag
\end{eqnarray}
 Therefore, we obtain
 \begin{align}
  d(x,y)&\leq d(x,w)+d(w,w^J)+d(w^J,\gamma^1_{t_0})+d(\gamma^1_{t_0},y)\notag\\
  &<l-r_0+d(\gamma^1_{t_0},\gamma^1_{s_0})+\xi+\xi+d(\gamma^1_{t_0},y)\notag\\
  &<l-r_0+2\xi+d(\gamma^1_{s_0},y)\notag\\
  &<l=d(x,y).\notag
 \end{align}
 which is a contradiction. 
\end{proof}
 %%%%%%%%%%%%%%%%%%%%%%%%%%%%%%%%%%%%%%%%%% 
 %%%%%%%%%%%%%%%%%%%%%%%%%%%%%%%%%%%%%%%%%% 
 %\red{Here we want to find a closed interval $J$ with $\mathcal{L}^1(J)>0$ and $\mathcal{L}^1(J\setminus I)=0$. So we assume that every closed interval of positive measure satisfies $\mathcal{L}^1(J\setminus I)>0$. }
 %\textcolor{cyan}{First of all, assume that $t_0:=\min\{t\in I\}>l-r_0$. Put $t_1:=t_0-(l-r_0)>0$ \bl{$====>$ [[[The problem is that $t_1$ might be $0$!]]]}. Consider the closed interval $J:=[t_0,t_0+t_1/2]$. If $\mathcal{L}^1(J\setminus I)=0$, there is nothing to prove. Suppose $\mathcal{L}^1(J\setminus I)>0$. Then by Claim \ref{clm:MCP-Lebesgue}, we have $\mathcal{L}^1(tJ\setminus I)>0$ whenever $A_{tJ}\neq\emptyset$. Since $tJ\subset [l-r_0,t_0]$ for suitable $t>0$, $\mathcal{L}^1(J\setminus I)=0$, which is what we need. Assume that $t_0=l-r_0$. The same argument works if $t_0$ is an isolated point in $I$, so that we may assume $t_0$ is an accumulation point in $I$. }
 %%%%%%%%%%%%%%%%%%%%%%%%%%%%%%%%%%%%%%%%%% 
 %%%%%%%%%%%%%%%%%%%%%%%%%%%%%%%%%%%%%%%%%% 
\begin{clm}[\emph{invariance of $I$ under dilations}]\label{clm:I-invariance}
For any $ 0 < t \le 1$ and up to a set of measure zero, one has
\be
	\left(  \frac{1}{t} I  \right) \cap (l-r_0   , l)  \subset I.   \notag
\ee
In other words, inside the interval $\left( l - r_0 , l  \right)$, $I$ is invariant under dilations. In particular, for $t \ll 1$, we get $t^{-1}I\cap (l-r_0,l)=\emptyset \subset I$ and for $t= 1$, we have $I  \cap (l-r_0   , l)  = I$.
\end{clm}
%%%%%%%%%%%%%%%%%%%%%%%%%%%%%%%%%%%%%%%%%% 
 %%%%%%%%%%%%%%%%%%%%%%%%%%%%%%%%%%%%%%%%%% 
\begin{proof}
Suppose not. Then, there exists $0 < t' <1$ such that $\mathcal{L}^1 \left(   \left(  \frac{1}{t'} I  \right) \cap (l-r_0   , l)  \right) >0 $ and $\left(  \frac{1}{t'} I  \right) \cap (l-r_0   , l) \not  \subset I $ (in a.e. sense). By Claim~\ref{clm:MCP-Lebesgue}, taking $J:=tI$, where $t' = \frac{1}{t} <1 $ and $t = \frac{1}{t'} > 1$, we would have
\be
	\mathcal{L}^1 \left( (t'J) \setminus I  \right) > 0  \quad (\text{and also} \quad  tJ \not \subset I  ),\notag
\ee
%%%%%%%%%%%%%%%%%%%%%%%%%%%%%%%%%%%%%%%%%% 
 %%%%%%%%%%%%%%%%%%%%%%%%%%%%%%%%%%%%%%%%%%
which means
\be
 t' \left( l - r_0    ,   l    \right)  \cap	I   \not \subset I, \notag
\ee
which is obviously a contradiction. 
\end{proof}
Now to finish the proof of the Claim~\ref{clm:closed-interval}, suppose, $s \in I$ is a Lebesgue density point of $I$. This means 
\be
	\liminf_{\delta \to 0} \frac{\mathcal{L}^1 \left( [s -\delta , s + \delta] \setminus I    \right)}{2\delta} = 0.  \notag
\ee
%%%%%%%%%%%%%%%%%%%%%%%%%%%%%%%%%%%%%%%%%% 
 %%%%%%%%%%%%%%%%%%%%%%%%%%%%%%%%%%%%%%%%%% 

For any $\epsilon>0$, choose $\delta>0$ such that
\be
	\mathcal{L}^1 \left( [s -\delta  , s + \delta] \setminus I    \right) < 2 \epsilon \delta. \notag
\ee
%%%%%%%%%%%%%%%%%%%%%%%%%%%%%%%%%%%%%%%%%% 
 %%%%%%%%%%%%%%%%%%%%%%%%%%%%%%%%%%%%%%%%%% 
  Let $I_\epsilon := [s -\delta , s + \delta] \cap I$. Then, by Claim~\ref{clm:I-invariance}, one has
 \be
 	\mathcal{L}^1 \left( (t I_\epsilon ) \cap \left[ (l-r_0 , l) \setminus I \right] \right) = 0, \quad \forall \quad t \ge 1. \notag 
 \ee 
 %%%%%%%%%%%%%%%%%%%%%%%%%%%%%%%%%%%%%%%%%% 
 %%%%%%%%%%%%%%%%%%%%%%%%%%%%%%%%%%%%%%%%%% 
Then for any $t\geq 1$, using the scaling property of the Lebesgue measure, and the scale invariance of $I$, we can compute
 %%%%%%%%%%%%%%%%%%%%%%%%%%%%%%%%%%%%%%%%%% 
 %%%%%%%%%%%%%%%%%%%%%%%%%%%%%%%%%%%%%%%%%%
\be 
	\mathcal{L}^1 \left( \Big(t [s -\delta , s+ \delta] \setminus I \Big)  \cap \left( l-r_0  , l  \right) \right)\le \mathcal{L}^1 \left( t \left( [s -\delta  , s+ \delta] \setminus I  \right)  \cap \left( l-r_0 , l  \right) \right) < t\epsilon\delta ,  \notag
\ee
indeed, by the invariance of $I$ under dilations, we have $tI \cap (l-r_0 , l) \subset I$ and this then would imply that 
\begin{eqnarray}
\left( t  [s -\delta , s+ \delta] \setminus I \right) \cap  \left( l-r_0  , l  \right)  &\subset& \left( t [s -\delta  , s+ \delta] \setminus tI \right) \cap \left( l-r_0 , l  \right) \notag \\ &=&  t \left( [s -\delta , s+ \delta] \setminus I  \right) \cap \left( l-r_0 , l  \right). \notag
\end{eqnarray}
%%%%%%%%%%%%%%%%%%%%%%%%%%%%%%%%%%%%%%%%%% 
 %%%%%%%%%%%%%%%%%%%%%%%%%%%%%%%%%%%%%%%%%
Now for $k$ satisfying
%\begin{align}
% \log\frac{l}{s}<k\log\frac{s+\delta}{s}<\log\frac{l}{s}+\log\frac{s+\delta}{s} \le 2 \log\frac{l}{s}.\notag
%\end{align}
%we have,
%\textcolor{purple}{In below, since  $k$ is a large overkill, we just use subset notation $===>$}
%\textcolor{green}{The equation above, we define $k$ for which satisfies 
\begin{align}
 \frac{(s+\delta)^{k-1}}{s^{k-2}}< l\leq \frac{(s+\delta)^k}{s^{k-1}},\notag
\end{align}
we can write
\begin{eqnarray}
	[s , l ] & \subset& [s , s+\delta] \cup \left[s+\delta , \frac{(s + \delta )^2}{s}\right] \cup \left[\frac{(s + \delta )^2}{s} , \frac{(s + \delta )^3}{s^2}\right]  \cup \dots \left[\frac{(s + \delta )^{k-1}}{s^{k-2}} , \frac{(s + \delta )^{k}}{s^{k-1}}\right], \notag \\ &\subset&  \bigcup_{i = 0}^{k-1}  \frac{(s + \delta )^i}{s^{i}}  [s , s+\delta]   \notag
\end{eqnarray}
%%%%%%%%%%%%%%%%%%%%%%%%%%%%%%%%%%%%%%%%%% 
 %%%%%%%%%%%%%%%%%%%%%%%%%%%%%%%%%%%%%%%%%
Hence,
\begin{eqnarray}\label{eq:closed-interval}
m \left( [s,l] \setminus I \right) &\le& m \left( \bigcup_{i = 0}^{k-1}  \left( \frac{(s + \delta )^i}{s^{i}}  [s , s+\delta]  \setminus I   \right)  \right) \notag \\ &\le& \sum_{i= 0}^{k-1} m   \left(   \frac{(s + \delta )^i}{s^{i}}  [s , s+\delta]  \setminus I      \right) \notag \\ &\le&  \sum_{i= 0}^{k-1} \left(   \frac{s + \delta}{s}  \right)^{i} \epsilon \delta \notag \\ &=& \frac{\left(\frac{s+\delta}{s}\right)^k-1}{\frac{s+\delta}{s}-1}\epsilon\delta\notag\\
  &=&s\left(\left(\frac{s+\delta}{s}\right)^k-1\right)\epsilon\notag\\
  &\leq& s\left(\frac{(s+\delta)l}{s^2}-1\right)\epsilon . 
 \end{eqnarray}
%%%%%%%%%%%%%%%%%%%%%%%%%%%%%%%%%%%%%%%%%% 
 %%%%%%%%%%%%%%%%%%%%%%%%%%%%%%%%%%%%%%%%
In above, the last inequality follows from the definition of $k$ since, 
 \begin{align} 
 \left(\frac{s+\delta}{s}\right)^k&=\frac{(s+\delta)^{k-1}}{s^{k-2}}\cdot\frac{s+\delta}{s^2}\notag\\
 &\leq l\cdot\frac{s+\delta}{s^2}.\notag  
 \end{align}
%%%%%%%%%%%%%%%%%%%%%%%%%%%%%%%%%%%%%%%%%% 
 %%%%%%%%%%%%%%%%%%%%%%%%%%%%%%%%%%%%%%%%%
Therefore, first letting $\delta \to 0$,  and then $\epsilon \to 0$ in (\ref{eq:closed-interval}), we get
\be
	\mathcal{L}^1 \left( [s , l] \setminus I   \right) = 0. \notag
\ee
%%%%%%%%%%%%%%%%%%%%%%%%%%%%%%%%%%%%%%%%%% 
 %%%%%%%%%%%%%%%%%%%%%%%%%%%%%%%%%%%%%%%%%
 This argument can be applied to any Lebesgue density point, $s$, in $I$ (and we know almost every point of $I$ is so). So, with a little bit more work, one can in fact prove that if $s_0 := \inf I$, then
\be
	\mathcal{L}^1 \left( [s_0 , l] \setminus I   \right) = 0. \notag
\ee
%%%%%%%%%%%%%%%%%%%%%%%%%%%%%%%%%%%%%%%%%% 
 %%%%%%%%%%%%%%%%%%%%%%%%%%%%%%%%%%%%%%%%%
\begin{rem}
The conclusion of Claim~\ref{clm:closed-interval} is obviously wrong for arbitrary metric measure spaces (one needs MCP or some sort of curvature conditions. ). The following is a counterexample: Let $C \subset [0,1]$ be a closed nowhere dense Cantor set with positive Lebesgue measure (such sets exist). Take the isometric product $X = [0,1] \times [0,1]$ with measure  $\iota_\sharp\mathcal{L}^1 \times \mathcal{L}^1$  where $\iota: C \hookrightarrow [0,1]$ is the inclusion.
\end{rem}
\end{proof} 
  \begin{rem}
  One can weaken the assumptions in Lemma \ref{lem:lebesgueae}. It is not essential to assume $d(x,y)=d(x,z)$. We just need two sets $A$ and $B$ that are included in $\{w\in X\;;\; r_1\leq d(x,w)\leq r_2\}$ for a pair of numbers $0<r_1<r_2<\infty$. 
 \end{rem}
 %%%%%%%%%%%%%%%%%%%%%%%%%%%%%%%%%%%%%%%%%%%%%
  %%%%%%%%%%%%%%%%%%%%%%%%%%%%%%%%%%%%%%%%%%%%%
 %%%%%%%%%%%%%%%%%%%%%%%%%%%%%%%%%%%%%%%%%%%%%%%%%%%%%%%%%%%%%%%%%%%%%%%%%%%%%%%%%%%%%%%%%%%%%%%%%%%%%
 %
 %   Section : Proof of the characterization theorem
 %
 %%%%%%%%%%%%%%%%%%%%%%%%%%%%%%%%%%%%%%%%%%%%%%%%%%%%%%%%%%%%%%%%%%%%%%%%%%%%%%%%%%%%%%%%%%%%%%%%%%%%
 \section{Proof of the characterization theorem }
   Let $(X,d,m)$ be a metric measure space. Then, the $RCD^*(K,N)$ condition for $K\in\R$ and $N\in(1,\infty)$, or more precisely, the locally doubling condition will imply that $m$ satisfies $m(U)>0$ for any open set $U \subset \supp m$. For brevity, when there is no confusion, we will denote $Tan(X,d,m,x)$ by just $Tan(X,x)$.
   %Let $(X,d,m)$ be an $RCD^*(K,N)$ space with $1\leq N<2$.
   %%%%%%%%%%%%%%%%%%%%%%%%%%%%%%%%%%%%%%%%%% 
 %%%%%%%%%%%%%%%%%%%%%%%%%%%%%%%%%%%%%%%%%% 
   \begin{defn}\label{def:regular}
   We define the following subsets of $X$ based on the point-wise structure of the tangent space:  
    %%%%%%%%%%%%%%%%%%%%%%%%%%%%%%%%%%%%%%%%%% 
 %%%%%%%%%%%%%%%%%%%%%%%%%%%%%%%%%%%%%%%%%%
    \begin{align}
      \We{k}:=&\left\{x\in X\;;\;\text{There exist proper metric measure spaces }(Y,y)\in Tan(X,x),\right.\notag\\
      &\left.\text{ and }(W,w)\text{ with }\di W>0\text{ s.t. }Y=\R^k\times W\right\},\notag\\
      %\Eu_k:=& \left\{x\in X\;;\;\text{for each }(Y,y)\in Tan(X,x),\text{ there exists a} \right.\notag\\
      %& \left.\text{ proper metric measure space }(W,w),\text{ s.t. }(Y,y)=(\R^k\times W,(0^k,w))\right\},\notag\\
      \RE_j:=&\left\{x\in X\;;\;Tan(X,x)=\{(\R^k,0^k)\}\right\}.\notag
    \end{align}
    %%%%%%%%%%%%%%%%%%%%%%%%%%%%%%%%%%%%%%%%%% 
 %%%%%%%%%%%%%%%%%%%%%%%%%%%%%%%%%%%%%%%%%%
    And $\RE:=\bigcup_{j\geq 1}\RE_j$.
   \end{defn}
   %%%%%%%%%%%%%%%%%%%%%%%%%%%%%%%%%%%%%%%%%% 
 %%%%%%%%%%%%%%%%%%%%%%%%%%%%%%%%%%%%%%%%%%
  It is known that $m\left(X\setminus \RE\right)=0$ for an $RCD^*(K,N)$ space $(X,d,m)$ (see \cite{MN}). 
  %%%%%%%%%%%%%%%%%%%%%%%%%%%%%%%%%%%%%%%%%% 
 %%%%%%%%%%%%%%%%%%%%%%%%%%%%%%%%%%%%%%%%%%
  \begin{lem}\label{lem:split}
    Let $(X,d,m)$ be an $RCD^*(K,N)$ space for $K\in\R$ and $N\in(1,\infty)$. 
   Let $x\in X$ be a point and suppose $\gamma$ is a geodesic joining two points $p,q\in X\setminus\{x\}$ that also passes through $x$. Suppose there exists a point $z\notin\IM(\gamma)$ with $d(z,x)=d\left( z,\IM(\gamma) \right)$. Then, there exists a pointed \emph{proper} geodesic metric measure space, $(W,d_W,m_W,w)$, with $\diam W>0$ such that $\R\times W\in Tan(X,x)$.  In fact, every tangent is of the form, $\R \times W$ with $\diam W \ge 0$ and $W$ depending on the tangent.
  \end{lem}
  %%%%%%%%%%%%%%%%%%%%%%%%%%%%%%%%%%%%%%%%%% 
 %%%%%%%%%%%%%%%%%%%%%%%%%%%%%%%%%%%%%%%%%%
  %%%%%%%%%%%%%
  \begin{proof}
   Let $\eta:[0,d(z,x)]\rightarrow Z$ be a geodesic from $z$ to $x$. We have $d(\eta(t),x)=d \left( \eta(t),\IM(\gamma) \right)$ for all $t\in[0,d(z,x)]$. For $n > \frac{1}{d(x,z)}$, set $z_n:=\eta(t_n)$ where $t_n$ is the infimum of the numbers $t$ such that $\eta(t)\in B_{1/n}(x)$. Then obviously, $z_n\in\partial B_{1/n}(x)$. Set $w_n:=\eta(t_n+(d(z,x)-t_n)/2)$ and notice that $d(x,w_n)+d(w_n,z_n)=d(x,z_n)$ holds for any $n\in\mathbb{N}$. Denote by $d_n$, the normalized metric $d/n$.  A simple calculation using the local doubling property implies 
   %%%%%%%%%%%%%%%%%%%%%%%%%%%%%%%%%%%%%%%%%% 
 %%%%%%%%%%%%%%%%%%%%%%%%%%%%%%%%%%%%%%%%%%
   \begin{align}
    m\left(B^{d_n}_{1/2}(w_n)\right)\geq C(K,N)m\left(B^{d_n}_2(w_n)\right)\geq C(K,N)m\left(B^{d_n}_1(x)\right).\notag
   \end{align}
   %%%%%%%%%%%%%%%%%%%%%%%%%%%%%%%%%%%%%%%%%% 
 %%%%%%%%%%%%%%%%%%%%%%%%%%%%%%%%%%%%%%%%%%
  So, there exists a positive constant $C>0$ such that $m^x_n(B^{d_n}_{1/2}(w_n))\geq C$ for any $n\in\mathbb{N}$, where, $m^x_n$ is the normalized measure with respect to $d_n$ at $x$. Thus, in the virtue of the splitting theorem, we deduce that a subsequence of the pointed normalized metric measure spaces $(X,d_n,m^x_n,x)$ converges to a product space $(\R\times W,d_{\R\times W},\mathcal{L}^1\times m_W,(0,w))$ where, $(W,d_W,m_W,w)$ is a proper pointed geodesic metric measure space with $\diam W>0$ and with $m_W\neq 0$. 
   \end{proof}
   %%%%%%%%%%%%%%%%%%%%%%%%%%%%%%%%%%%%%%%%%% 
 %%%%%%%%%%%%%%%%%%%%%%%%%%%%%%%%%%%%%%%%%%
   %%%%%%%%%%%%% key lemma
  \begin{lem}\label{empty}
   For an $RCD^*(K,N)$ space, $(X,d,m)$, with $N\in[1,2)$, we have $\We{1}=\emptyset$.  
  \end{lem}
  %%%%%%%%%%%%%%%%%%%%%%%%%%%%%%%%%%%%%%%%%% 
 %%%%%%%%%%%%%%%%%%%%%%%%%%%%%%%%%%%%%%%%%%
  %%%%%%%%%%%%% 
  \begin{proof}
   Suppose not. Then, by the definition, for $x\in \We{1}$, there exists a proper metric measure space $(W,w)$ with $\diam W>0$ such that $(\R\times W,(0,w))\in Tan(X,x)$. The stability of $RCD^*$ condition under $\DC$ implies that $\R\times W$ is an $RCD^*(0,N)$ space. The splitting theorem then implies that $W$ is one point (see Theorem~\ref{thm:splitting}). This is in contradiction with the assumptions on $W$.  
  \end{proof}
  %%%%%%%%%%%%%%%%%%%%%%%%%%%%%%%%%%%%%%%%%% 
 %%%%%%%%%%%%%%%%%%%%%%%%%%%%%%%%%%%%%%%%%%
  %%%%%%%%%%%%%
  \begin{defn}[interior point] A point $x\in X$ is called an \emph{interior point} if there exists a geodesic $\gamma:[0,l]\rightarrow X$ with $\gamma(t)=x$ for some $t\in(0,l)$.
  \end{defn}
  %%%%%%%%%%%%%%%%%%%%%%%%%%%%%%%%%%%%%%%%%% 
 %%%%%%%%%%%%%%%%%%%%%%%%%%%%%%%%%%%%%%%%%%
  %%%%%%%%%%%%% interior
  \begin{prop}\label{interior}
   Let $x\in\RE_1$. Then $x$ is an interior point.
  \end{prop}
  %%%%%%%%%%%%%%%%%%%%%%%%%%%%%%%%%%%%%%%%%% 
 %%%%%%%%%%%%%%%%%%%%%%%%%%%%%%%%%%%%%%%%%%
  %%%%%%%%%%%%%
  \begin{proof}
   The proof is similar to the proof of Proposition 4.1 in Honda~\cite{Hlow}. Suppose that there exists a point $x\in \RE_1$ such that $x$ is not an interior point on a geodesic. 
   %By the definition of $\RE_1$, we have sequences of decreasing positive numbers $r_i\downarrow 0$ and $\epsilon_i\downarrow0$ such that
    %%%%%%%%%%%%%%%%%%%%%%%%%%%%%%%%%%%%%%%%%% 
 %%%%%%%%%%%%%%%%%%%%%%%%%%%%%%%%%%%%%%%%%%
 %\be  
 %   \DC((X,d_{r_i},m^x_{r_i},x),(\R,d_E,\mathcal{L}^1,0)) <\epsilon_i \rightarrow 0. \notag
 %\ee
  %%%%%%%%%%%%%%%%%%%%%%%%%%%%%%%%%%%%%%%%%% 
  \begin{clm}\label{clm:second-rescaling}  
  For a given sequence of decreasing positive numbers $\{\epsilon_i\}$, there exist sequences of increasing numbers $\{R_i\}$ tending to infinity and decreasing positive numbers $\{r_i\}$ tending to 0 such that one can pick $p_i,q_i\in X$ that satisfy
\be    
     \vert d(p_i,x)- R_ir_i \vert <r_i \epsilon_i \quad, \quad \vert d(q_i,x)- R_ir_i \vert<r_i \epsilon_i, \notag
\ee
  and,
  %%%%%%%%%%%%%%%%%%%%%%%%%%%%%%%%%%%%%%%%%% 
 %%%%%%%%%%%%%%%%%%%%%%%%%%%%%%%%%%%%%%%%%%  
  \be
  d(p_i,x)+d(q_i,x)-d(p_i,q_i)<r_i\epsilon_i. \notag
  \ee    
\end{clm}  
%%%%%%%%%%%%%%%%%%%%%%%%%%%%%%%%%%%%%%%%%% 
 %%%%%%%%%%%%%%%%%%%%%%%%%%%%%%%%%%%%%%%%%%    
  \begin{proof}
 First of all, rescaling the metric if necessary, we may assume that $\di X>1$. 
By Theorem 4.1 in \cite{MN}, there exists a number $\beta=\beta(N)>2$ with the following property: there exists a large number $\tilde{R}_i\gg 1$ such that for any $R\geq \tilde{R}_i$ there exist $0<r_i=r_i(\epsilon_i,R)\ll 1$ and points $\tilde{p}_i,\tilde{q}_i\in B^{d_{r_i}}_{R^{\beta}}(x)\setminus B^{d_{r_i}}_{R^{\beta}/4}(x)$ and also $\xi_i\in B^{d_{r_i}}_{R^{\beta}}(x)$ on a geodesic $c_i$ connecting $\tilde{p}_i$ to $\tilde{q}_i$ with $d(x,\xi_i)< r_i\epsilon_i/2$ that satisfy 
%%%%%%%%%%%%%%%%%%%%%%%%%%%%%%%%%%%%%%%%%% 
 %%%%%%%%%%%%%%%%%%%%%%%%%%%%%%%%%%%%%%%%%%   
 \begin{align}
  d(\tilde{p}_i,x)+d(\tilde{q}_i,x)-d(\tilde{p}_i,\tilde{q}_i)\leq 2d(x,\xi_i)<r_i\epsilon_i.\notag
 \end{align}
 %%%%%%%%%%%%%%%%%%%%%%%%%%%%%%%%%%%%%%%%%% 
 %%%%%%%%%%%%%%%%%%%%%%%%%%%%%%%%%%%%%%%%%%   
Since $R\geq \tilde{R}_i\gg 1$ is arbitrary, we may assume that $R\leq R^{\beta}/8$ (this is always true for $R \ge 4$). Put $R_i$ satisfying $\tilde{R}_i\leq R_i\leq R^{\beta}_i/8$ and $R_i\leq R_{i+1}$. Take points $p_i,q_i\in X$ with the following properties:
\begin{enumerate}
\item $p_i,q_i\in\IM(c_i)$ and $p_i,q_i$ are on \emph{opposite sides} of $\xi_i$, \\

\item  $d(p_i,\tilde{p}_i)\leq d(p_i,\tilde{q}_i)$ and $d(q_i,\tilde{q}_i)\leq d(q_i,\tilde{p_i})$ \\ 

\item $
  \vert d(p_i,x)-R_ir_i\vert <r_i\epsilon_i,\quad \& \quad \vert d(q_i,x)-R_ir_i\vert<r_i\epsilon_i.\notag $\\
 
 \end{enumerate}
 %%%%%%%%%%%%%%%%%%%%%%%%%%%%%%%%%%%%%%%%%% 
 %%%%%%%%%%%%%%%%%%%%%%%%%%%%%%%%%%%%%%%%%%   
 Notice that one can always find such points $p_i$ and $q_i$ on the geodesic $c_i$ since, $d(x,\tilde{\xi}_i)<r_i\epsilon_i/2$, $d(x,\tilde{p}_i)\geq R_i^{\beta}/4$, $d(x,\tilde{q}_i)\geq R_i^{\beta}/4$ and the distance function is continuous.
 %%%%%%%%%%%%%%%%%%%%%%%%%%%%%%%%%%%%%%%%%% 
 %%%%%%%%%%%%%%%%%%%%%%%%%%%%%%%%%%%%%%%%%%  
 
 %Since $R\geq \tilde{R}_i$ is arbitrary, we may assume $R$ is larger than 4, hence so is $R^{\beta}$. 
 %Note that the larger $R$ is taken, the smaller $r_i$ becomes (see~\cite{MN}). 
 %\red{This is the crucial step. Why guarantees that we can find such points on $c_i$? not clear to me!! are you using Abresch-Gromoll? $====>$}
 %\textcolor{cyan}{Take points $p_i,q_i\in X$ that satisfies $p_i,q_i\in\IM(c_i)$ with $d(p_i,\tilde{p}_i)\leq d(p_i,\tilde{q}_i)$ and $d(q_i,\tilde{q}_i)\leq d(q_i,\tilde{p_i})$  such that 
 %\begin{align}
 % \vert d(p_i,x)-1\vert <r_i\epsilon_i,\;\vert d(q_i,x)-1\vert<r_i\epsilon_i.\notag
 %\end{align}
 %Since $d(x,\tilde{\xi}_i)<r_i\epsilon_i/2$, $d(x,\tilde{p}_i)\geq R^{\beta}/4$, $d(x,\tilde{q}_i)\geq R^{\beta}/4$ and the distance function is continuous, we may take such points $p_i,q_i$.} \textcolor{green}{Indeed, consider a function $y\mapsto d(x,y)$, where $y$ is on $c_i$ and from $\tilde{\xi}_i$ to $\tilde{p}_i$. By the assumption, $d(x,\tilde{\xi}_i)<r_i\epsilon_i$ and $d(\tilde{p}_i,x)\geq R^{\beta}/4$, and by the continuity of the distance function, we are able to find a point $p_i$ on $\IM(c_i)$ such that $d(x,p_i)=1$. For $q_i$, we just do the same way. }\red{Here I will add more precise explanation soon. Please wait for more few days! I'm sorry for my late business :-0 }
Since $d(p_i,x)\leq d(p_i,\tilde{\xi}_i)+d(x,\tilde{\xi}_i)$, $d(q_i,x)\leq d(q_i,\tilde{\xi}_i)+d(x,\tilde{\xi}_i)$, and $d(p_i,\tilde{\xi}_i)+d(\tilde{\xi}_i,q_i)=d(p_i,q_i)$, we obtain 
 \begin{align}
  d(p_i,x)+d(q_i,x)-d(p_i,q_i)&\leq d(p_i,\tilde{\xi}_i)+d(q_i,\tilde{\xi}_i)+2d(x,\tilde{\xi}_i)-d(p_i,q_i)\notag\\
  &=2d(x,\tilde{\xi}_i)<r_i\epsilon_i.\notag
 \end{align}

 This is what we wanted to prove in this claim.
 %\be    
 %    \vert d(p_i,x)- r_i \vert <r_i \epsilon_i \quad, \quad \vert d(q_i,x)- r_i \vert<r_i \epsilon_i, \notag
%\ee
%Notice that the choice of $\epsilon_i$ is independent of the sequence $r_i$, so by picking a diagonal subsequence if necessary, we can assume
%\be    
%     \vert d(p_i,x)- r_i \vert <r_i^2 \epsilon_i \quad, \quad \vert d(q_i,x)- r_i \vert<r_i^2 \epsilon_i, \notag
%\ee
%Now, rescaling the metrics once more using the scales , $r_i$, we will get
%\be    
%     \vert d(p_i,x)- 1 \vert <r_i \epsilon_i \quad, \quad \vert d(q_i,x)- 1 \vert<r_i \epsilon_i,\notag
%\ee
%and accordingly,
% \be
%  d(p_i,x)+d(q_i,x)-d(p_i,q_i)<r_i\epsilon_i. \notag
%  \ee  
\end{proof}
  %%%%%%%%%%%%%%%%%%%%%%%%%%%%%%%%%%%%%%%%%% 
 %%%%%%%%%%%%%%%%%%%%%%%%%%%%%%%%%%%%%%%%%%  
  Pick $p_i,q_i\in X$ as in Claim~\ref{clm:second-rescaling}.
%%%%%%%%%%%%%%%%%%%%%%%%%%%%%%%%%%%%%%%%%% 
 %%%%%%%%%%%%%%%%%%%%%%%%%%%%%%%%%%%%%%%%%%    
Let $\gamma_i:[0,d(p_i,q_i)]\rightarrow X$ be a geodesic from $p_i$ to $q_i$. Set $s_i:=d\left( x,\IM(\gamma_i) \right)$. By the assumption, 
\be
0<s_i=d\left( x,\IM(\gamma_i) \right)<r_i\epsilon_i. \notag
\ee
 %%%%%%%%%%%%%%%%%%%%%%%%%%%%%%%%%%%%%%%%%%  
This means that $s_i\rightarrow 0$ as $i \to \infty$. Using the pre-compactness, a subsequence $(X,s_i^{-1}d,m^x_{s_i},x)$, converges to a limit space $(Y,d_Y,m_Y,y)$. Now, our construction implies that there exist a limit point, $z\in \partial B_1(y)$, corresponding to a sequence of points, $z_i\in \partial B_{s_i}(x)$, with $d\left( x,\IM(\gamma_i) \right)=d(x,z_i)$. 
Now, in the $s_i-$ rescaled spaces, 
%\textcolor{cyan}{we may assume} \red{the following formula seems to be not naturally come from the assumption. But as Mondino-Naber's paper said, we can take $p_i,q_i$ far from $x$. So maybe we must change $\vert d(p_i,x)-1\vert<r_i\epsilon$ e.t.c.} 
that %\bl{I think we should not change this since referees were happy with the proof!!}
\be
	s_i^{-1} d (p_i , x) \quad \& \quad s_i^{-1} d (q_i , x) \to \infty \quad \text{as} \quad i \to \infty, \notag
\ee
so that $\gamma_i$ converges to a line in $Y$. Thus,
we get an isometric embedding $L:\R\rightarrow Y$ such that $z\in \IM (L)$ and $y\notin \IM (L)$. This implies that $Y=\R\times W$ for some proper geodesic space with $\diam W>0$ which contradicts $x\in\RE_1$.
  \end{proof}
  %%%%%%%%%%%%%%%%%%%%%%%%%%%%%%%%%%%%%%%%%% 
 %%%%%%%%%%%%%%%%%%%%%%%%%%%%%%%%%%%%%%%%%%
  %%%%%%%%%%%%%%%%%%%%%%%%%%%%%%%%%%%%%%%%%%%%%%%%%%%%%%%%%%%%%%%%%%%%%%%%%%%%%%%%%%%%%%%%%%%%%%%%%%%
  The following theorem is key.
  \begin{thm}\label{local}
   Let $(X,d,m)$ be an $RCD^*(K,N)$ space for $K\in\R$ and $N\in(1,\infty)$. Assume $\mathcal{R}_1\neq \emptyset$. 
   For any $x\in X$, there exists a positive number $\epsilon>0$ such that $(B_{\epsilon}(x),x)$ is isometric to $\left( (-\epsilon,\epsilon) , 0 \right)$ or to $\left( [0 , \epsilon) , 0 \right)$. 
  \end{thm}
 %%%%%%%%%%%%%%%%%%%%%%%%%%%%%%%%%%%%%%%%%% 
 %%%%%%%%%%%%%%%%%%%%%%%%%%%%%%%%%%%%%%%%%%
 
   %%%%%%%%%%%%%%%%%%%%%%%%%%%%%%%%%%%%%%%%%% 
 %%%%%%%%%%%%%%%%%%%%%%%%%%%%%%%%%%%%%%%%%%
\begin{proof}\mbox{}\\*
   \begin{enumerate}
    \item $x\in\RE_1$. 
 %%%%%%%%%%%%%%%%%%%%%%%%%%%%%%%%%%%%%%%%%% 
 %%%%%%%%%%%%%%%%%%%%%%%%%%%%%%%%%%%%%%%%%%    
 Since $x\in \mathcal{R}_1$ is an interior point, there exists a geodesic $\gamma : [-\epsilon,\epsilon]\rightarrow X$ with $\gamma(0)=x$. Suppose that for any $\eta>0$, the set $B_{\eta}(x)\setminus \IM(\gamma)$ is non-empty. Without loss of generality, we may assume 
 \begin{align}
  \eta\leq 10^{-10}\min\left\{\sqrt{\frac{2\log2}{3\vert K\vert+1}},\epsilon\right\}. \notag
 \end{align}
By the assumption, we are able to take $y_n\in B_{\eta/n^2}(x) \setminus \IM(\gamma)$ and $z_n\in \IM(\gamma)$ so that $d(z_n,y_n)=d(y_n,\IM(\gamma))$. By Lemma \ref{lem:split}, $z_n\neq x$ for $n$ large enough. We may assume $z_n\in\mathcal{P}:=\{\gamma_t\;;\;t>0\}$. Now take $w_n\in\mathcal{N}:=\{\gamma_t\;;\;t<0\}$ so that
 \begin{align}
 d(w_n,x)=d(z_n,x). \notag %\label{eq:different}
 \end{align}
 %\begin{align}
 %d(w_n,x)\leq \frac{d(z_n,x)}{10^{10}}.\label{eq:different}
 %\end{align}
 %Notice that for any point $x_n\in B_{\eta/n^2}(z_n)$ and any constant speed geodesic $c_n:[0,1]\rightarrow X$ connecting $w_n$ to $x_n$, one has $d(c_n(1/2),w_n)\gg d(x,w_n)$. 
 Set $l_n:=d(z_n,y_n)$. Then by using the doubling property (also see the proof of Lemma \ref{lem:split}), we have $m(B_{l_n/2}(y_n))>0$ and $B_{l_n/2}(y_n)\cap \IM(\gamma) =\emptyset$. 

  %%%%%%%%%%%%%%%%%%%%%%%%%%%%%%%%%%%%%%%%%%
 %%%%%%%%%%%%%%%%%%%%%%%%%%%%%%%%%%%%%%%%%%
%\textcolor{cyan}{Define
 
 %\begin{align}
 % \eta\leq 10^{-10}\min\left\{\sqrt{\frac{2\log 2}{3\vert K\vert+1}}, \epsilon\right\}.\notag
% \end{align}
%Take $w_n\in \mathcal{N}:=\{\gamma_t\;;\;t<0\}$ so that 
%\begin{align}
% d(w_n,x)\leq \frac{d(z_n,x)-\eta/n^2}{10}.\notag
%\end{align}
%Note that for any point $x_n\in B_{\eta/n^2}(z_n)$ and any constant speed geodesic $c_n:[0,1]\rightarrow X$ connecting from $x_n$ to $w_n$, $d(c_n(1/2),w_n)\geq d(x,w_n)$. 
% }
 %%%%%%%%%%%%%%%%%%%%%%%%%%%%%%%%%%%%%%%%%%
 %%%%%%%%%%%%%%%%%%%%%%%%%%%%%%%%%%%%%%%%%%
Let $\theta$ be a unit speed geodesic from $y_n$ to $z_n$ and set 
 %%%%%%%%%%%%%%%%%%%%%%%%%%%%%%%%%%%%%%%%%%
\be
\alpha_n:=1/2\min\{d(x,z_n),l_n/n^2\}. \notag
\ee
 %%%%%%%%%%%%%%%%%%%%%%%%%%%%%%%%%%%%%%%%%%
Then, for some $k\ge 2$, $l_n > \eta/kn^2$ and $\theta(l_n-\eta/kn^2) \in B_{\alpha_n}(z_n) \setminus \IM(\gamma)$ (otherwise, one can find a point on $\gamma$ that is strictly closer to $y_n$ than $z_n$ is). Therefore, $B_{\alpha_n}(z_n) \setminus \IM(\gamma)$ is non-empty. And by the doubling property, 
\begin{align}
 m\left(B_{\alpha_n}(z_n)\setminus \IM(\gamma)\right)&\geq m\left(B_{\alpha_n/4}(\theta(l_n-\eta/kn^2))\right)\notag\\
 &\geq Cm\left(B_{100\eta}(\theta(l_n-\eta/kn^2))\right) \notag\\
 &\geq Cm(B_{\eta}(z_n))>0.\notag
 \end{align}

%If $B_{\eta/n^2}(z_n) \subset \gamma$ , then, $B_{\eta/n^2}(z_n) = \gamma \left( \left( \gamma^{-1}(z_n) - \eta/n^2 ,    \gamma^{-1}(z_n) + \eta/n^2 \right)  \right)$ (without loss of generality, we are assuming that $\gamma$ is of unit speed). Let $\theta$ be a geodesic from $y_n$ to $z_n$, then we would have either 
% \be
% \gamma \left( \left[ \gamma^{-1}(z_n) - \red{l_n?} , \gamma^{-1}(z_n) \right]  \right) \subset \IM (\theta)
% \ee
%   or 
%\be   
%   \gamma \left( \left[ \gamma^{-1}(z_n) , \gamma^{-1}(z_n) + \red{l_n?}  \right]  \right) \subset \IM (\theta)
% \ee  
%and consequently we would have either $\gamma \left( \gamma^{-1}(z_n) - \eta/n^2 \right)$  or $\gamma \left( \gamma^{-1}(z_n) + \eta/n^2 \right) $ is a point on $\gamma$ that is strictly closer to $y_n$ than $z_n$ is. This is a contradiction.
 %%%%%%%%%%%%%%%%%%%%%%%%%%%%%%%%%%%%%%%%%%
 %%%%%%%%%%%%%%%%%%%%%%%%%%%%%%%%%%%%%%%%%%
 %In fact, $B_{\eta/n^2}(z_n) \setminus \IM(\gamma)$ is of positive measure since otherwise, for some $p'$ and $r'>0$, we would have $B_{r'}(p') \subset B_{\eta/n^2}(z_n) \setminus \gamma$ is a null set and hence, $p'$ is an atom. This can not happen in spaces which satisfy $MCP(K,N)$ property (see Ohta~\cite{Ohmcp}).
%%%%%%%%%%%%%%%%%%%%%%%%%%%%%%%%%%%%%%%%%% 
 %%%%%%%%%%%%%%%%%%%%%%%%%%%%%%%%%%%%%%%%%% 
 \begin{clm}\label{clm:RS}
 There exists a point $x_n\in B_{\alpha_n}(z_n)$ and a geodesic $c_n:[0,1]\rightarrow X$ connecting $w_n$ and $x_n$ such that $d(x,\IM(c_n))>0$. 
 \end{clm}
%%%%%%%%%%%%%%%%%%%%%%%%%%%%%%%%%%%%%%%%%% 
 %%%%%%%%%%%%%%%%%%%%%%%%%%%%%%%%%%%%%%%%%%
 \begin{proof}[Proof of Claim]
 To prove the claim, we are going to use an argument similar to the one in Rajala-Sturm\cite{RS}. From Rajala-Sturm~\cite{RS}, we know that the optimal transport between any two absolutely continuous measures in a space satisfying strong $CD(K,\infty)$ condition is concentrated on non-branching geodesics.
 %%%%%%%%%%%%%%%%%%%%%%%%%%%%%%%%%%%%%%%%%% 
 %%%%%%%%%%%%%%%%%%%%%%%%%%%%%%%%%%%%%%%%%%
 \par The idea of the proof is that if there exists a $\pi \in OptGeo(\mu_1,\mu_0)$ (for absolutely continuous $\mu_i \in \mathcal{P}_2(X)$) that does not live on non-branching geodesics, then via restriction in time and space and using disintegration, one can find  a measure $\mathfrak{M}$ on $Geo(X)\times Geo(X)$ and a family of geodesic pairs $\Gamma_a \subset Geo(X)\times Geo(X)$ ($a \in (0,1)$) with $\mathfrak{M} (\Gamma_a) >0$ and 
\be 
 m\left( e_0 \left( p_1 \left( \Gamma_a \right) \right) \right)m\left( e_0 \left( p_2 \left( \Gamma_a \right) \right) \right)>0, \notag
 \ee
where, the geodesic pairs in $ \Gamma_a$ satisfy the following conditions: 
 there exists a sufficiently small $\xi>0$ such that $\restr_0^a\gamma_1=\restr_0^a\gamma_2$ and $\restr_0^{a+\xi}\gamma_1\neq\restr_0^{a+\xi}\gamma_2$ for any $(\gamma_1,\gamma_2)\in\Gamma_a$. Then, writing down the $K-$convexity conditions for the entropy of the transportation from $e_0(p_1(\Gamma))\cup e_0(p_2(\Gamma))$ to $e_1(p_1(\Gamma))\cup e_1(p_2(\Gamma))$, one proves that the underlying space \emph{fails to satisfy the strong $CD(K,\infty)$ condition}.
 %%%%%%%%%%%%%%%%%%%%%%%%%%%%%%%%%%%%%%%%%% 
 %%%%%%%%%%%%%%%%%%%%%%%%%%%%%%%%%%%%%%%%%%
\par  To prove Claim~\ref{clm:RS}, we are going to prove that the assumption that every geodesic connecting $w_n$ to a point $x_n \in B_{\alpha_n}(z_n)$ passes through $x$ and the fact that $m \left( B_{\alpha_n}(z_n)\setminus \IM(\gamma) \right)> 0$ would provide us with such family of "bad" geodesics and that would lead to a contradiction.
 %%%%%%%%%%%%%%%%%%%%%%%%%%%%%%%%%%%%%%%%%% 
 %%%%%%%%%%%%%%%%%%%%%%%%%%%%%%%%%%%%%%%%%%
 \par Suppose for $w_n\in \mathcal{N}\cap B_{\eta/n^2}(x)$ and for any geodesic $c_n$ connecting $w_n$ to a point $x_n\in B_{\alpha_n}(z_n)$, there exists a time $t\in(0,1)$ such that $c_n(t)=x$. Consider $\mu_0:=\delta_{w_n}$ and $\mu_1:=  \chi_{B_{\alpha_n}(z_n)} m / m(B_{\alpha_n}(z_n))$. By Theorem 1.1 and Corollary 1.6 in Gigli-Rajala-Sturm~\cite{GRS}, one could find a unique $\tilde{\pi}\in OptGeo(\mu_1,\mu_0)$ that is induced by a map. The optimal plan $\tilde{\pi}$ also satisfies $(e_t)_\sharp\tilde{\pi}\ll m$ for any $t\in[0,1)$. Define a map $\sigma: Geo(X)\rightarrow Geo(X)$ by $\sigma(\gamma)_t:=\gamma_{1-t}$. Let $\pi$ denote the measure $\sigma_\sharp\tilde{\pi}$. Then, $\pi$ satisfies $\mu_t:=(e_t)_\sharp\pi\ll m$ for any $t\in (0,1]$ and $\mu_t$ is a geodesic connecting $\mu_0$ to $\mu_1$. 
%%%%%%%%%%%%%%%%%%%%%%%%%%%%%%%%%%%%%%%%%% 
 %%%%%%%%%%%%%%%%%%%%%%%%%%%%%%%%%%%%%%%%%%  
\par  Note that $\pi$ is supported on the branching subset $\Gamma\subset Geo(X)$ of geodesics starting off as $\gamma$. Indeed, since $B_{\alpha_n}(z_n)\setminus \IM(\gamma) \neq\emptyset$, one can pick a small ball $B\subset B_{\alpha_n}(z_n)\setminus \IM(\gamma)$ with $0 < m(B) < \frac{1}{2} m \left( B_{\alpha_n}(z_n)\setminus \IM(\gamma)  \right)$. Let  $g:X\rightarrow \R$ be the distance function, $g(x)=d(w_n,x)$. Thus by the inclusion relation, $g(B)\subset g \left( B_{\alpha_n}(z_n) \setminus \IM(\gamma) \right)$ holds. Now, by assumption we know that for almost every geodesic $\theta$ in the support of $\pi$, there is a time, $t_\theta$, such that $\theta(t_\theta) = x$. We replace $\theta$, up to time $t_\theta$, by $\gamma|_{[0,t_\theta]}$. Also, notice that, for this family of geodesics, the branching time parameters, $a$ and $\xi$ can be chosen as follows:
%%%%%%%%%%%%%%%%%%%%%%%%%%%%%%%%%%%%%%%%%% 
\be
0 < a:= \frac{d(w_n , x)}{d(w_n , z_n) + \alpha_n} \le t_\theta =  \frac{d(w_n , x)}{d(w_n , x_n)} \le \frac{d(w_n , x)}{d(w_n , z_n) - \alpha_n} =:a' =  a + \xi . \notag
\ee
%%%%%%%%%%%%%%%%%%%%%%%%%%%%%%%%%%%%%%%%%% 
 %%%%%%%%%%%%%%%%%%%%%%%%%%%%%%%%%%%%%%%%%%  
Therefore, by the uniqueness of $\pi$, and from the proof of Lemma~\ref{lem:lebesgueae} and the above argument, we deduce that there exist two subsets $\Gamma_1,\Gamma_2\subset\supp\pi$ with $\pi(\Gamma_1)\pi(\Gamma_2)>0$ such that for any $\gamma_1\in\Gamma_1$, there exists $\gamma_2\in\Gamma_2$ with $\restr_0^a\gamma_1=\restr_0^a\gamma_2$ and $\gamma_1(1)\in B$, $\gamma_2(1)\in B_{\alpha_n}(z_n)\setminus \IM(\gamma)$.
%the set $(\gamma_1,\gamma_2)\in \operatorname{supp}(\pi) \subset Geo(X)^2$  with $\restr_0^a \gamma_1=\restr_0^a \gamma_2$ satisfying $\gamma_1(1)\in B$ and $\gamma_2(1)\in B_{\eta/n^2}(z_n) \setminus \IM\gamma$ has $\pi\times\pi$-positive measure.
%%%%%%%%%%%%%%%%%%%%%%%%%%%%%%%%%%%%%%%%%% 
 %%%%%%%%%%%%%%%%%%%%%%%%%%%%%%%%%%%%%%%%%%

\par  By restricting and rescaling $\pi$, we obtain a restricted plan $\pi$ that is supported on branching geodesics (with the abuse of notation, we will also denote this restricted measure by the same character , $\pi$).

  Now, we have at our disposal, all the ingredients needed for the arguments in Rajala-Sturm~\cite{RS} to work. So, employing the exact same arguments as in Rajala-Sturm~\cite{RS}, one obtains two measures $\pi^u,\pi^d$ with the following properties : \\
  \begin{enumerate}
   \item $\beta :=\pi^u(Geo(X))=\pi^d(Geo(X))$. \\
   
   \item There exist a time $a \in(0,1)$ and sufficiently small  $\xi >0$ with $a + \xi<1$ such that $(e_s)_\sharp\pi^u=(e_s)_\sharp\pi^d$ for any $s\in[0,a]$ and $\mu^d_{a + \xi}:=(e_{a + \xi})_\sharp \pi^d/\beta$, $\mu^u_{a + \xi}:=(e_{a + \xi})_\sharp \pi^u/\beta$ are mutually singular with respect to each other. \\
   
   \item For fixed small number $b>0$, there exists a positive number $C>0$ such that  
            \begin{align}
             \frac{d(e_b)_\sharp\pi^d}{dm},\; \frac{d(e_1)_\sharp\pi^d}{dm}, \frac{d(e_1)_\sharp\pi^u}{dm}\leq C. \notag
            \end{align} 
            
    \item Set $\mu^u_{a + \xi}=\rho^u_{a + \xi}m$. Then 
             \begin{align}
              \int \rho^u_{a + \xi}\log\rho^u_{a + \xi}\,dm\geq \beta \log\frac{\xi}{10m(x,\eta/2)}. \label{lower}
             \end{align}
  \end{enumerate}
  Exploiting the $K$-convexity of the entropy along the plan $(\pi^u+\pi^d)/(2 \beta)$ from $b$ to $a + \xi$ (in a similar fashion as in Step 7 in \cite{RS}), we will get a contradiction. See the Appendix~\ref{app:A} for detailed computations.
 \end{proof}
%%%%%%%%%%%%%%%%%%%%%%%%%%%%%%%%%%%%%%%%%% 
 %%%%%%%%%%%%%%%%%%%%%%%%%%%%%%%%%%%%%%%%%%
 The proof of Claim~\ref{clm:RS}, in fact, implies that for $m$-a.e. $x_n \in B_{\alpha_n}(z_n)$ and for $\pi$-almost every geodesic $\theta$, connecting $w_n$ to the point $x_n \in B_{\alpha_n}(z_n)$, we know $\theta$ does not pass through $x$. Thus we find the family of geodesics, $\{c_n\}_{n\in\mathbb{N}}$, from $w_n$ to a point $x_n\in B_{\alpha_n}(z_n)$ with $d(x,\IM(c_n))>0$.
 %%%%%%%%%%%%%%%%%%%%%%%%%%%%%%%%%%%%%%%%%% 
 %%%%%%%%%%%%%%%%%%%%%%%%%%%%%%%%%%%%%%%%%%
\par Moreover, we may assume that $\pi$-a.e. geodesics, $c_n$ do not intersect $\mathcal{P}$ since, otherwise, one could replace the geodesic $c_n$ that intersect $\mathcal{P}$ with the geodesics, $\tilde{c}_n$ given by 
\begin{align}
 \tilde{c}_n(t):=\begin{cases}
                        \gamma(t)&\text{if }c_n(s)\notin \IM (\gamma) \text{ for any }0<s<t\\
                        c_n(t)&\text{otherwise}
                       \end{cases} . \notag
\end{align}
 Now, the collection of $\tilde{c}_n$'s  would form a family of geodesics from $w_n$ to $x_n$ of positive $\pi$-measure and passing through $x$, this is in contradiction with the uniqueness of $\pi$ and the proof of Claim~\ref{clm:RS}. 
%%%%%%%%%%%%%%%%%%%%%%%%%%%%%%%%%%%%%%%%%% 
 %%%%%%%%%%%%%%%%%%%%%%%%%%%%%%%%%%%%%%%%%%
\par So far, we have that $\pi$-almost every geodesic does not pass through $x$ and does not intersect $\mathcal{P}$. Pick one of these good geodesics $c_n$. 
\par  Let $L_n$ denote the distances $d(x,w_n) = d(x,z_n)$. We get 
 \begin{align}
  0<d(x,\IM(c_n))\leq d(x,z_n)=L_n\rightarrow 0.\notag
 \end{align} 
%%%%%%%%%%%%%%%%%%%%%%%%%%%%%%%%%%%%%%%%%% 
 %%%%%%%%%%%%%%%%%%%%%%%%%%%%%%%%%%%%%%%%%%
Let us consider the rescaled metric measure space $(X,d_{L_n},m^x_{L_n},x)$. Since $x\in\mathcal{R}_1$, we have $X_{L_n}\rightarrow \R$ (taking subsequence if necessary). Let $f_n:X_{L_n}\rightarrow\R$ be the approximation maps that realize the convergence $X_{L_n}\rightarrow\R$. Since $(X,d_{L_n},m^x_{L_n},x)\rightarrow (\R,d_E,\mathcal{L}^1,0)$, there exist points on each $\IM(c_n)$ that converge to $0\in\R$ and consequently any sequence of points, $c_n(t_n)$ with $d \left( x , c_n(t_n) \right) = d \left( x , \IM(c_n) \right)$ also has to converge to $x$. Thus, we are able to find a sequence $t_n$ such that $c_n(t_n)$ satisfies $d(x,c_n(t_n))=d(x, \IM(c_n))$ and $f_n(c_n(t_n))\rightarrow 0\in\R$.
Indeed, every point $c_n(t)$ obviously satisfies 
\begin{align}
 d(x,c_n(t))&\leq d(x,w_n)+d(w_n,c_n(t))\leq d(x,w_n)+d(w_n,x_n)+d(x_n,z_n)\notag\\
 &\leq 4L_n.\notag
\end{align}
Therefore, $\limsup d_{L_n}(x,c_n(t))<\infty$. This means that the image of the geodesic $c_n$ approaches to $\IM(\gamma)$ as $n\rightarrow \infty$ in the $L_n$-scale.
%(We kill the possibility the most far point on $c_n$ goes infinity in the $L_n$-scale !). 
Also since $d_{L_n}(x,w_n)=d_{L_n}(x,z_n)=1$, $c_n(t_n)$ does not go closer to neither $w_n$ nor $z_n$ . 
%\red{By the definition of $t_n$, $d_{L_n}(x,c_n(t_n))\rightarrow 0$. However since $d_{L_n}(x,w_n)=d_{L_n}(x,z_n)=1$, $c_n(t_n)$ never close to end points. }

Define $s_n:=d(x,c_n(t_n))$ and consider $(X,d_{s_n},m^x_{s_n},x)$. If 
\begin{align}
 \liminf_{n\rightarrow\infty}d_{s_n}(\IM(\gamma),c_n(t_n))>0,\notag
\end{align}
we find a point in the limit space that is not on the geodesic corresponding to $\IM(\gamma)$. This is a contradiction to $x\in\mathcal{R}_1$. 

On the contrary, suppose 
\begin{align}
 \liminf_{n\rightarrow \infty}d_{s_n}(\IM(\gamma),c_n(t_n))=0.\notag
\end{align}
This means points $c_n(t_n)$ are converging to a point on $\gamma$ in the $s_n-$scale. Assume that $c_n(t_n)$ converges to a point in $\Pro$ in the $s_n$-scale (the case, $c_n(t_n)$ converging to a point in $\mathcal{N}$ in the $s_n$-scale can be ruled out in a similar fashion). Pick times $t_n'$ such that $t_n'\leq t_n$ and $d(c_n(t_n'),c_n(t_n))=s_n$. It is easy to see that we can find such a point $c_n(t_n')$ since the assumption that $c_n(t_n)$ converges to a point in $\Pro$ implies $d(w_n,c_n(t_n)) > s_n$ for $n$ large enough. By the construction, $d(x,c_n(t_n'))\geq d(x,c_n(t_n))=s_n$. Hence $d_{s_n}(x,c_n(t_n'))\geq 1$. Since $x\in \mathcal{R}_1$ and $d_{s_n}(c_n(t_n'),c_n(t_n))=1$, $c_n(t_n')$ converges to a point on $\IM(\gamma)$ in the $s_n$-scale. 

Let $a:=\lim_nh_n(c_n(t_n'))\in\R$, where $h_n:X_{s_n}\rightarrow \R$ are approximation maps. Since $d_{s_n}(c_n(t_n'),c_n(t_n))=1$, $a=0$ or $a=2$. If $a=2$, this contradicts the minimality of $c_n$. Thus $a=0$. 
Note that 
\begin{align}
 d(x,c_n(t_n'))\leq d(x,c_n(t_n))+d(c_n(t_n),c_n(t'_n))\leq 2s_n.\notag
\end{align}
Hence $K_n:=d(x,c_n(t_n'))$ satisfies $s_n\leq K_n\leq 2s_n$. Consider $(X,d_{K_n},m^x_{K_n},x)$. Taking a subsequence if necessary, we know $X_{K_n}\rightarrow \R$ via the approximation maps $g_n:X_{K_n}\rightarrow \R$. Since $x\in\mathcal{R}_1$, $d_{K_n}(x,c_n(t_n'))=1$, and $s_n\leq K_n\leq 2s_n$, $d_{K_n}(c_n(t_n'),\IM(\gamma))\rightarrow 0$ and $g_n(c_n(t_n'))\rightarrow -1\text{ or }1\in\R$. However, again by $s_n\leq K_n\leq 2s_n$, $d_{K_n}(x,c_n(t_n'))\leq 2d_{s_n}(x,c_n(t_n'))\rightarrow 0$. This is a contradiction.

 \par Now, Consider pointed normalized metric measure spaces $(X,s_n^{-1}d,m^x_{s_n},x)$ that converge to $(Y,d_Y,m_Y,y)\in Tan(X,x)$ in the measured Gromov-Hausdorff sense. By the rescaling, it is clear that $(Y,y)$ is not isomorphic to $(\R,0)$. This contradicts $x\in \mathcal{R}_1$.

  %%%%%%%%%%%%%%%%%%%%%%%%%%%%%%%%%%%%%%%%%% 
 %%%%%%%%%%%%%%%%%%%%%%%%%%%%%%%%%%%%%%%%%% 
  \mbox{}\\*  \item $x\in X\setminus \RE_1$. 
%%%%%%%%%%%%%%%%%%%%%%%%%%%%%%%%%%%%%%%%%% 
%%%%%%%%%%%%%%%%%%%%%%%%%%%%%%%%%%%%%%%%%% 
 %%%%%%%%%%%%%%%%%%%%%%%%%%%%%%%%%%%%%%%%%% 
Since $\mathcal{R}_1\neq\emptyset$, one can find a point $y\in\mathcal{R}_1$. By the proof of (1) above, a neighbourhood of any such $y$ is isometric to an open interval. 
 Therefore, $\RE_1$ is an open set. If $\RE_1$ is closed, then $X$ must be $\RE_1$ itself. This contradicts the existence of $x\in X\setminus \RE_1$.
 %and this, in fact, implies that any geodesic from a point $w$ to $y$ can be extended past $y$. 
Note that $\RE_1$ is an open $1-$ dimensional manifold. If the open set $\RE_1$ is a circle, take a point, $p$ in the circle that is the closest point from $x$,  Lemma~\ref{lem:split} implies that there exist a tangent cone at $p$ that is not isometric to $\R$. This is a contradiction (one can also see the contradiction by noticing that a circle is closed).
 %%%%%%%%%%%%%%%%%%%%%%%%%%%%%%%%%%%%%%%%%% 
 %%%%%%%%%%%%%%%%%%%%%%%%%%%%%%%%%%%%%%%%%%
\par The maximal connected open subset in $\RE_1$, which contains $y\in\RE_1$, is a locally minimizing curve $\gamma:(-a,b)\rightarrow X$, $a,b\in(0,\infty]$, which satisfies $\gamma_0=y$. Furthermore, $\gamma_{-a}:=\lim_{t\rightarrow -a}\gamma_t$ and $\gamma_b:=\lim_{t\rightarrow b}\gamma_t$ when $a,b\neq\infty$, do not belong to $\mathcal{R}_1$. Locally, a neighbourhood of each point in $\RE_1$ is isometric to $(-\epsilon,\epsilon)$. This means the maximal connected subset in $\RE_1$ should be a local minimizing unit speed geodesic.
%%%%%%%%%%%%%%%%%%%%%%%%%%%%%%%%%%%%%%%%%% 
 %%%%%%%%%%%%%%%%%%%%%%%%%%%%%%%%%%%%%%%%%%
%%%%%%%%%%%%%%%%%%%%%%%%%%%%%%%%%%%%%%%%%% 
 %%%%%%%%%%%%%%%%%%%%%%%%%%%%%%%%%%%%%%%%%% 
\par  Just to make it more clear, we can argue as follows: Let $\gamma: (-a,b) \to \RE_1 \subset X$ be a locally minimizing curve with $\gamma(0) = y \in \RE_1$. If $p = \gamma(t_1) = \gamma(t_2)$ for some $t_1 , t_2 \in (-a , b)$ and $t_1 \neq t_2$, then since a neighbourhood of $p \in \RE_1$ is isometric to an interval, we deduce that $\gamma$ has to be periodic (after trivially expanding its domain to $\R$) so $\gamma \subset \RE_1$ is a circle. But as we previously showed, this can not happen.

 %%%%%%%%%%%%%%%%%%%%%%%%%%%%%%%%%%%%%%%%%% 
 %%%%%%%%%%%%%%%%%%%%%%%%%%%%%%%%%%%%%%%%%% 
Therefore, from the argument above, we can assume $\gamma$ has no self-intersections and can be extended from either end in a locally minimizing fashion as long as $a$ (or $b$) stays finite. Suppose $(-a , b)$ ($a,b \in \R \cup \{\infty\}$) is the maximal domain for the locally minimizing curve $\gamma$. Then, if $b < \infty$ (respectively $a < \infty$), we must have $\gamma_b:=\lim_{t\rightarrow b}\gamma_t \not \in \RE_1$ (respectively $\gamma_{-a}:=\lim_{t\rightarrow - a}\gamma_t \not \in \RE_1$) since otherwise, one can extend $\gamma$ further and in a locally minimizing fashion.
 %%%%%%%%%%%%%%%%%%%%%%%%%%%%%%%%%%%%%%%%%% 
 %%%%%%%%%%%%%%%%%%%%%%%%%%%%%%%%%%%%%%%%%% 
 \par When both $a$ and $b$ are $\infty$, consider a point , $z$ on $\gamma$ with $d(x,z) = d(\gamma , x)$. Then, Lemma~\ref{lem:split} implies that $z \notin \RE_1$ which is a contradiction.
 %%%%%%%%%%%%%%%%%%%%%%%%%%%%%%%%%%%%%%%%%% 
 %%%%%%%%%%%%%%%%%%%%%%%%%%%%%%%%%%%%%%%%%% 
 \par Without loss of generality, we assume $b < \infty$.
%(\red{This needs to be argued similar to the above case where $\gamma$ could not be a circle. We know in this case $\gamma = \R$. We need to find a closest point $z \in \IM(\gamma)$ to $x$ and argue that $z \not \in \RE_1$ and get a contradiction).} 
 Consider a geodesic $\theta:I\rightarrow X$ from $x$ to a point $z\in\overline{\IM(\gamma)}$ that satisfies $d(x,z)=d\left( x,\IM(\gamma) \right)$. If $z=\gamma_t$ for $t\in(-a,b)$, we will get a contradiction by part (1) or by using Lemma~\ref{lem:split} (in this case, there exists a tangent cone at $z$ which is not $\R$). 
 %%%%%%%%%%%%%%%%%%%%%%%%%%%%%%%%%%%%%%%%%% 
 %%%%%%%%%%%%%%%%%%%%%%%%%%%%%%%%%%%%%%%%%% 
 % (\red{this is exactly the argument we have to also use in above cases}) 
Without loss of generality, we may assume that $z=\gamma_b \notin \RE_1$. 
Suppose $x\neq z$. Notice that for any $\eta > 0$, $B_{\eta}(z)\setminus \left( \IM(\gamma)\cup\IM(\theta) \right)\neq\emptyset$ since, otherwise a neighbourhood of $z$ would be isometric to an open interval. Indeed,  $B_{\eta}(z)\setminus \left( \IM(\gamma)\cup\IM(\theta) \right) = \emptyset$ implies that a neighbourhood of $z$ is just the concatenation of two minimal geodesics, $\gamma$ and $\theta$; also every geodesic joining two points in $B_{\frac{\eta}{10}}(z)$ is included in $B_{\eta}(z)$ therefore, $B_{\frac{\eta}{10}}(z)$ is \emph{isometric} to $(-\frac{\eta}{10} , \frac{\eta}{10})$. This means $z \in \RE_1$ which we know is not the case.
%%%%%%%%%%%%%%%%%%%%%%%%%%%%%%%%%%%%%%%%%% 
 %%%%%%%%%%%%%%%%%%%%%%%%%%%%%%%%%%%%%%%%%% 
\par In particular, the above argument ensures us that if $x \neq z$, one must have $B_{\eta}(z)\setminus \left( \IM(\gamma)\cup\IM(\theta) \right)\neq\emptyset$ for any $\eta >0$. Take a point $w\in B_{\eta}(z)\setminus (\IM(\gamma)\cup\IM(\theta))$ and consider a geodesic $\alpha$ from $w$ to the point $v\in \overline{\IM(\gamma)}$ that satisfies $d(w,v)=d\left( w,\overline{\IM(\gamma}) \right)$. 
 Since $\gamma_t\in\RE_1$ for $t\in(-a,b)$, $v\neq \gamma_t$ for $t\in(-a,b)$; this means $v = z = \gamma_b$. 
%%%%%%%%%%%%%%%%%%%%%%%%%%%%%%%%%%%%%%%%%% 
 %%%%%%%%%%%%%%%%%%%%%%%%%%%%%%%%%%%%%%%%%%
\par From now on, we just repeat a similar argument as in the case (1). For the sake of completeness, we give an outline of the proof. Take a point $z'\in\mathcal{R}_1$, which is close enough to $z$. In order to apply the argument in (1), we may assume that 
\begin{align}
 d(y,z'),d(x,z'),d(w,z')\ll 10^{-10}\sqrt{\frac{2\log 2}{3\vert K\vert+1}}.\notag
\end{align}
Take 
\be
	r:=\min\{d(x, z'),d(w,z'),d(x,w), d(w,\IM(\theta)) \}/4 , \notag
\ee
and define $A:=B_r(x)\cup B_r(w)$. Note that $B_r(x)\cap B_r(w)=\emptyset$. By considering the optimal transportation between $\mu_0:=m\vert_A/m(A)$ and $\mu_1:=\delta_y$, we are able to find a curve $c$ from $y$ to a point in $A$ not passing through $z'$. This means that there exists a point in $z'=\gamma_t=c_s$ for $t\in (0,b)$, $s\in (0,1)$ such that $c_{s'}\notin \IM(\gamma)$ for any $s'>s$. This contradicts that $z'=\gamma_t\in\mathcal{R}_1$. Therefore $x=z$.

This means that $x$ has to be the end point (after taking the closure of the geodesic) of the geodesic, $\gamma \subset \RE_1$. 
 % If $x$ is the end point of more than one geodesic, a similar argument as in above leads to a contradiction. 
 Hence $B_{\epsilon}(x)$ is isometric to $[0,\epsilon)$ for sufficiently small $\epsilon>0$.
\end{enumerate}
  \end{proof}

\begin{rem}
We note that in the proof of Claim~\ref{clm:RS}, the geodesics \emph{are not} branching at the same time but they are all branching within a tiny time interval $[a,a']$ the length of which going to zero as $n \to \infty$ and that is enough to get a contradiction. Another possible approach would perhaps be to \emph{non-linearly} contract the geodesics toward $w_n$ so that all branch at the same time and then use the measure contraction property to get a contradiction. The difficulty in this approach is that since all geodesics are of constant speed and parametrized on $[0,1]$, when we perform such a non-linear contraction, we will end up with a family of geodesics that branch at the same time but their end points will all be on a sphere with center $w_n$. This contradicts the measure contraction property or the spherical Bishop-Gromov inequality in, for example, non-collapsed Ricci limit spaces of dimensions strictly larger than $1$. But in the setting of $RCD^*(K,N)$ spaces, it is unclear to the authors how to derive a contradiction having a family of branching constant speed geodesics parametrized on $[0,1]$ (all branching at the same time) with end points on a sphere. To the best of authors' knowledge, a spherical Bishop-Gromov volume comparison or measure contraction property (i.e. a volume comparison or measure contraction property for the co-dimension $1$ measures) is yet not available in this setting. Also notice that even in the simplest example of the \emph{letter "Y" space (the tripod)}, the geodesics emanating from one point on one branch and going to other two branches, once parametrized on $[0,1]$ , are branching at different times (depending on their lengths).
\end{rem}

%%%%%%%%%%%%%%%%%%%%%%%%%%%%%%%%%%%%%%%%%%%%%%%%%%%%%%%%%%%%%%%%%%%%%%%%%%%%%%%%%%%%
   %%%%%%%%%%%%%%%%%%%%%%%%%%%%%%%%%%%%%%%%%%%%%%%%%%%%%%%%%%%%%%%%%%%%%%%%%%%%%
\begin{rem}
In Theorem~\ref{local}, we have in fact proven the stronger fact that in any $RCD^*(K,N)$ metric measure space, $\RE_1$ is an open and \emph{convex} (convexity follows from arguments in the proof of part (2)) subset. In Ricci limit spaces, the convexity of all the regular sets follow from the recent developments by Naber and Colding but to the best of our knowledge, in the metric measure setting, this is not known (at least for $\RE_k$, $k \ge 2$).
\end{rem}
%%%%%%%%%%%%%%%%%%%%%%%%%%%%%%%%%%%%%%%%%%%%%%%%%%%%%%%%%%%%%%%%%%%%%%%%%%%%%%%%%%%%
   %%%%%%%%%%%%%%%%%%%%%%%%%%%%%%%%%%%%%%%%%%%%%%%%%%%%%%%%%%%%%%%%%%%%%%%%%%%%%%%%%%%%
\begin{defn}\label{defn:reference-measure}
 Let $(X,d)$ be a geodesic, proper complete separable metric space. 
 A positive Radon measure $\mu$ on $X$ is a \emph{reference measure} (in the sense of Cavalletti-Mondino~\cite{FM}) for $(X,d)$ provided it is non-zero, and $\mu$-a.e. $z\in X$ there exists $\pi^z$, which is a positive Radon measure on $X\times X$, such that 
 \begin{align}
  (p_1)_\sharp\pi^z=\mu,\;\pi^z(X\times X\setminus H(z))=0,\;(p_2)_\sharp\pi^z\ll\mu,\notag
 \end{align}
 where $p_i:X\times X\rightarrow X$ is the natural $i$-th projection maps $i=1,2$ and 
 \begin{align}
  H(z):=\left\{(x,y)\in X\times X\;;\;d(x,y)=d(x,z)+d(z,y)\right\}.\notag
 \end{align} 
 The measure $\pi^z$ is called an inversion plan. $\mathrm{lp}(\mu)$ is the set of all points $z\in X$ that has an inversion plan $\pi^z$. 
\end{defn}
%%%%%%%%%%%%%%%%%%%%%%%%%%%%%%%%%%%%%%%%%%%%%%%%%%%%%%%%%%%%%%%%%%%%%%%%%%%%%%%%%%%%%%%%%%%%%%%%%%%%%%
\begin{prop}\label{prop:reference}
   Let $(M,g)$ be a complete Riemannian manifold of dimension $1$ and let $d_g, m_g$ be the Riemannian distance function and the Riemannian volume measure associated with $g$ (respectively). Let $\mu$ be a locally finite Borel measure on $M$ satisfying $RCD^*(K,N)$ condition for $K\in\R$ and $N\in(1,\infty)$. Assume that $\supp (\mu)=M$. Then, $\mu$ and $m_g$ are reference measures for $(M,d_g)$ and $\mu\sim m_g$, and $\mu = e^{-V} m_g$ for some locally integrable function $V$.
  \end{prop}
  
   %%%%%%%%%%%%%%%%%%%%%%%%%%%%%%%%%%%%%%%%%%%%%%%%%%%%%%%%%%%%%%%%%%%%%%%%%%%%%%%%%%%%
   %%%%%%%%%%%%%%%%%%%%%%%%%%%%%%%%%%%%%%%%%%%%%%%%%%%%%%%%%%%%%%%%%%%%%%%%%%%%%%%%%%%%
 %\bl{  \begin{proof}
 %  If $(M,g)$ has no boundary, the first claim for $m_g$ follows from Proposition 8.1 in \cite{FM} (also see \cite{BC}). Even if $(M,g,m_g)$ is $[0,l]$ for some $l>0$ or $\R_{\geq 0}$, a similar proof as in the aforementioned paper works well. This means that we can write $\mu = e^{-V}m_g$. To see that $V$ is locally integrable, see the proof of Lemma~\ref{lem:KNconvex} in below. 
 % \end{proof} }
  \begin{proof}
   The fact that $m_g$ is a reference measure follows from Cavalletti-Mondino~\cite{FM}. We present an argument as to why $\mu$ is also a reference measure.  First of all, the measure $\mu$ is a Radon measure (\cite{Fo}*{Theorem 7.8}). 
   Since $(M,d_g,\mu)$ satisfies the $RCD^*(K,N)$ condition for $K\in\R$, $N\in(1,\infty)$, $\mu$ does not have atoms, that is, $\mu(\{x\})=0$ for any $x\in M$.  
   Assume $(M,d_g)$ is isometric to $(\R_{\geq 0},d_E)$ (the other cases can be dealt with in a similar way).
   First of all, by Proposition 3.4 in \cite{FM}, we have $\mu\ll m_g$. Take $z\in \R_{>0}$ and fix it. 
   \subsection*{Step 1.} We find a family of compact sets $K_n\subset M$, $n\in\mathbb{N}$ and bi-Lipschitz maps $\Phi_n$ (the so called local inversion maps), such that, 
   \begin{itemize}
    \item $\mu(M\setminus \cup_{n\in\mathbb{N}}K_n)=0$,
    \item For every $x\in K_n$, there exists a unique constant speed minimal geodesic $\gamma_{xz}:[0,1]\rightarrow M$ from $x$ to $z$, which can be extended to $[0,1+2/n]\rightarrow M$ as a minimal geodesic,
    \item The map $\Phi_n:K_n\rightarrow M$ defined by $\Phi_n(x):=\gamma_{xz}(1+1/n)$ is bi-Lipschitz onto its image.  
   \end{itemize}
   Set $d_z:=\vert z\vert>0$ and $I:=\{x\in\R_{\geq 0}\;;\;\vert x\vert\leq d_z\}$. Define 
   \begin{align}
    \tilde{K}_n:=\left\{x\in\R_{\geq 0}\;;\; d_z\leq x\leq \left(1+\frac{n}{2}\right)d_z\right\}\notag,
   \end{align}
   and $K_n:=\tilde{K}_n\cup I$. It is clear that $\R_{\geq 0}\setminus \cup_{n\in\mathbb{N}}K_n=\emptyset$. This means $m_g(\R_{\geq 0}\setminus \cup_{n\in\mathbb{N}})=0$, accordingly $\mu(\R_{\geq 0}\setminus \cup_{n\in\mathbb{N}})=0$. Since the map $\Phi_n:K_n:=[0,(1+n/2)d_z]\rightarrow[d_z/2,(1+1/n)d_z]$ is bi-Lipschitz, we have required properties. 
   \subsection*{Step 2.} Define a map $\Phi:M\rightarrow M$ as 
   \begin{itemize}
    \item if $x\in \cup_{n\in\mathbb{N}}K_n$, $\Phi(x):=\Phi_{n_x}(x)$, where $n_x:=\min\{n\in\mathbb{N}\;;\; x\in K_n\}$;
    \item if $x\in M\setminus \cup_{n\in\mathbb{N}}K_n$, $\Phi(x)=z$. 
   \end{itemize} 
   Take the measure $\pi^z:=(Id,\Phi)_\sharp \mu$. We claim that $\pi^z$ satisfies all the properties required in Definition~\ref{defn:reference-measure}. It is clear that $(p_1)_\sharp \pi^z=\mu$ and $\pi^z(X\times X\setminus H(z))=0$ by the construction. The last property $(p_2)_\sharp \pi^z\ll m_g$ is proven as follows. Let $E\in \R_{\geq 0}$ be a Lebesgue negligible set, that is, $m_g(E)=0$. Since $m_g$ is also Hausdorff measure, $m_g(\phi(E))=0$ for any bi-Lipschitz map $\phi:M\rightarrow M$ (see for instance \cite{AT}*{Proposition 3.1.4}). Therefore, we obtain 
   \begin{align}
    (p_2)_\sharp \pi^z(E)&=\pi^z(p^{-1}_2(E))=\mu(\Phi^{-1}(E))=\mu\left(\Phi^{-1}(E)\cap \left(\bigcup_{n\in\mathbb{N}}K_n\right)\right)\notag\\
    &=\mu\left(\bigcup_{n\in\mathbb{N}}\Phi^{-1}(E)\cap\left(K_n\setminus \cup_{1\leq j\leq n-1}K_j\right)\right)\notag\\
    &\leq \sum_{n\in\mathbb{N}}\mu(\Phi_n^{-1}(E))=0, \notag
   \end{align} 
   Notice that the last equality follows since the sets ,$\phi_n^{-1}(E)$, are Lebesgue negligible sets and $\mu\ll m_g$. Hence, $\mu$ is a reference measure for $(M,d_g)$. The same proof shows that $m_g$ is also a reference measure for $(M,d_g)$. Theorem 5.3 in \cite{FM} tells us $\mu\sim m_g$. For other one-dimensional spaces, namely, $M=\R,[0,l],S^1(r)$, similar arguments work with slight modifications. To see that $V$ is locally integrable, see the proof of Lemma~\ref{lem:KNconvex} in below.
  \end{proof}
   %%%%%%%%%%%%%%%%%%%%%%%%%%%%%%%%%%%%%%%%%%%%%%%%%%%%%%%%%%%%%%%%%%%%%%%%%%%%%%%%%%%%
   %%%%%%%%%%%%%%%%%%%%%%%%%%%%%%%%%%%%%%%%%%%%%%%%%%%%%%%%%%%%%%%%%%%%%%%%%%%%%%%%%%%%
  
   %%%%%%%%%%%%%%%%%%%%%%%%%%%%%%%%%%%%%%%%%%%%%%%%%%%%%%%%%%%%%%%%%%%%%%%%%%%%%%%%%%%%
 \begin{lem}\label{lem:KNconvex}
   Assume there exists a measurable function $V:\R\rightarrow[-\infty,\infty]$ such that a metric measure space $(\R,d_E,m:=e^{-V}\mathcal{H}^1)$ satisfies $RCD^*(K,N)$ for $K\in\R$, $N\geq 1$. Then there exists a $(K,N)$-convex function $W:\R\rightarrow \R$ such that $\mathcal{H}^1(\{x\in\R\;;\;V(x)\neq W(x)\})=0$. In particular, $W$ is continuous and $(\R,d_E,e^{-W}\mathcal{H}^1)$ is an $RCD^*(K,N)$ space, which is isomorphic to $(\R,d_E,m)$. 
  \end{lem}
  %%%%%%%%%%%%%%%%%%%%%%%%%%%%%%%%%%%%%%%%%%%%%%%%%%%%%%%%%%%%%%%%%%%%%%%%%%%%%%%%%%%%
   %%%%%%%%%%%%%%%%%%%%%%%%%%%%%%%%%%%%%%%%%%%%%%%%%%%%%%%%%%%%%%%%%%%%%%%%%%%%%%
\begin{proof}
Since $m \sim \mathcal{H}^1$, $e^{-V}$ is locally integrable in $\R$ and the set %\red{[only $\{x\in X\;;\;V(x)= - \infty \}$]} 
$\{x\in X\;;\;V(x)=\{-\infty,\infty\}\}$ is $\mathcal{H}^1$-negligible. First of all, we notice that for a given bounded Borel set $\Omega\subset\R$, the integral of the negative part of $V$ on $\Omega$ is finite. Indeed, decompose $V$ into the positive and the negative parts $V=(V)_+-(V)_-$. Then, using, $x \le e^{x}$, we get
 \begin{align}
\int_{\Omega}(V)_-\,d\mathcal{H}^1\leq \int_{\Omega}e^{(V)_-}\,d\mathcal{H}^1\leq\int_{\Omega}e^{-V}\,d\mathcal{H}^1=m(\Omega)<\infty.\notag
\end{align}
%%%%%%%%%%%%%%%%%%%%%%%%%%%%%%%%%%%%%%%%%%%%%%%%%%%%%%%%%%%%%%%%%%%%%%%%%%%%%%%%%%%%
   %%%%%%%%%%%%%%%%%%%%%%%%%%%%%%%%%%%%%%%%%%%%%%%%%%%%%%%%%%%%%%%%%%%%%%%%%%%%%%
%Suppose that there exists a point $x\in\R^n$ such that 
%\begin{align}
% \int_{B_r(x)}(V)_{+}\,d\mathcal{H}^n=\infty\label{lem:KNconvex:infty}
%\end{align}
%for any $r>0$. Set $\Omega_r^k:=B_r(x)\cap\{V\geq k\}$ and $\Omega^{\infty}_r:=\cap_k\Omega^k_r$. By (\ref{lem:KNconvex:infty}), we have $\mathcal{H}^n(\Omega^k_r)>0$ for any $r>0$ and $k>0$. 
%Let $F\subset \R$ be a \red{connected} compact subset, that is, a closed interval $[a,b]$. 
%By Lusin's theorem (\cite{Fo}*{Exercise 44}), for any $\epsilon>0$, there exists a compact subset $F_{\epsilon}$ with $\mathcal{H}^1(F\setminus F_{\epsilon})<\epsilon$ on which $(V)_+$ is continuous. $F_{\epsilon}$ consists of a family of finite closed intervals $\{[a_i,b_i]\}_{i=1}^k$. On each closed interval $[a_i,b_i]$, $(V)_+$ is integrable. Thus $V\in L^1([a_i,b_i],\mathcal{H}^1)$.\red{$<----$ this is not completely true, for example, $\cup_{n=1}^\infty [1/(2n+1) , 1/2n] \cup \{0\}$ is compact!. I agree that a similar argument should work but there are lots of compact subsets that are not of this form. .....   } 
%}
%\red{Maybe something like the following also makes sense: $=--->$}
For $k>>0$, take $V^k = \min \{ V , k\}$. $V^k$ is integrable w.r.t $\mathcal{H}^1$ and any other absolutely continuous measure. Fix a closed interval $[a,b]$ and denote by $\mathcal{H}^1$ even which is restricted on $[a,b]$.
%%%%%%%%%%%%%%%%%%%%%%%%%%%%%%%%%%%%%%%%%%%%%%%%%%%%%%%%%%%%%%%%%%%%%%%%%%%%%%%%%%%%
   %%%%%%%%%%%%%%%%%%%%%%%%%%%%%%%%%%%%%%%%%%%%%%%%%%
\begin{clm}\label{clm:integrability}
For any measure $\mu \in\Pro([a,b])$ that is absolutely continuous with respect to $\mathcal{H}^1$ and $\mu\in \mathcal{D}(Ent(\cdot\vert \mathcal{H}^1))$, we have
\be
	\int V^k\,d\mu \le Ent(\mu\vert e^{-V}\mathcal{H}^1) - Ent(\mu\vert\mathcal{H}^1) < \infty. \notag
\ee
\end{clm}
%%%%%%%%%%%%%%%%%%%%%%%%%%%%%%%%%%%%%%%%%%%%%%%%%%%%%%%%%%%%%%%%%%%%%%%%%%%%%%%%%%%%
   %%%%%%%%%%%%%%%%%%%%%%%%%%%%%%%%%%%%%%%%%%%%%%%%%%
\begin{proof}
Note that the equivalence $\mathcal{H}^1\sim e^{-V}\mathcal{H}^1$ and the integrability of the negative part of $V$ imply $\mu\in\mathcal{D}(Ent(\cdot\vert e^{-V}\mathcal{H}^1))$. 
Since $V^k$ is integrable, we can write
\be\label{eq:integ-1}
Ent(\mu\vert e^{-V^k}\mathcal{H}^1)=Ent(\mu\vert\mathcal{H}^1)+\int V^k\,d\mu . \notag
\ee
%%%%%%%%%%%%%%%%%%%%%%%%%%%%%%%%%%%%%%%%%%%%%%%%%%%%%%%%%%%%%%%%%%%%%%%%%%%%%%%%%%%%
   %%%%%%%%%%%%%%%%%%%%%%%%%%%%%%%%%%%%%%%%%%%%%%%%%%
Let $U_N(r) := -N r^{1 - \frac{1}{N}}$ defined on $\R_{ \ge 0}$. Then, on $\R_{>0}$, $U_N$ is negative valued, decreasing and convex. Let
\be
	S_N(\nu \vert m) := N + \int \; U_N(\rho) \; dm , \notag
\ee 
then, from Sturm~\cite{Stmms1} and the fact $\mathcal{H}^1([a,b])<\infty$, we know that for any $\nu \in \mathcal{P}_2([a,b])$
\be
	Ent(\nu \vert m) = \lim_{N \to \infty} S_N(\nu \vert m) \left( = \sup_{N}	 S_N(\nu \vert m) \right). \notag
\ee
%%%%%%%%%%%%%%%%%%%%%%%%%%%%%%%%%%%%%%%%%%%%%%%%%%%%%%%%%%%%%%%%%%%%%%%%%%%%%%%%%%%%
   %%%%%%%%%%%%%%%%%%%%%%%%%%%%%%%%%%%%%%%%%%%%%%%%%%
   Now, for the problem in hand, we have $\mu = \rho_1 e^{-V^k}\mathcal{H}^1 = \rho_2 e^{-V} \mathcal{H}^1$ which means $\rho_1 \le \rho_2$ $\mathcal{H}^1$-a.e. (since $V^k \le V$). We can now compute
\be
	\int V^k\,d\mu = Ent(\mu\vert e^{-V_k}\mathcal{H}^1) - Ent(\mu\vert\mathcal{H}^1), \notag
\ee
and
\begin{eqnarray}
	Ent(\mu\vert e^{-V_k}\mathcal{H}^1) = \sup_{N}	 S_N (\mu \vert e^{-V_k}\mathcal{H}^1) &=& \sup_{N} \left[ N +	\int\;  -N \left(  \frac{d \mu}{d e^{-V^k}\mathcal{H}^1}  \right)^{1 - \frac{1}{N}} d e^{-V^k}\mathcal{H}^1 \right] \notag \\ 
	&=& \sup_{N} \left[ N +	\int\;  -N \left(  \underbrace{ \frac{d \mu}{d e^{-V^k}\mathcal{H}^1} }_{\rho_1} \right)^{- \frac{1}{N}} d \mu \right] \notag  \\ &\le& \sup_{N} \left[ N +	\int\;  -N \left(  \underbrace{\frac{d \mu}{d e^{-V}\mathcal{H}^1}}_{\rho_2} \right)^{- \frac{1}{N}} d \mu \right]  \notag \\ = Ent(\mu\vert e^{-V}\mathcal{H}^1), \notag
\end{eqnarray}
hence we get the desired result.
\end{proof}
So now, let $\mu := \frac{1}{b-a}\mathcal{H}^1$ be the normalized Hausdorff measure on $[a,b]$ and notice that we have $\int V^k\,d\mathcal{H}^1 $ is increasing and bounded above. Hence by monotone convergence theorem and since $Ent(\mathcal{H}^1 \vert e^{-V}\mathcal{H}^1) - Ent(\mathcal{H}^1 \vert\mathcal{H}^1) = Ent(\mathcal{H}^1 \vert e^{-V}\mathcal{H}^1) < \infty $, we get
\be
	\int\; V = \lim \int \; V_k < \infty. \notag
\ee
%Suppose there exists a bounded Borel set $\Omega\subset\R^n$ on which $V \ge 0$ and with $\mathcal{H}^n(\Omega)>0$ \red{[with $m(\Omega)>0$ which also implies $\mathcal{H}^n(\Omega)>0$ ]}such that 
%\begin{align}
% \int_{\Omega}(V)_+\,d\mathcal{H}^n=\infty.\label{lem:KNconvex:eq} 
%\end{align}
%Set $\Omega^k_+:=\Omega\cap\{V\geq 1/k\}$ for $k\in\mathbb{N}$ and $\Omega_+:=\Omega\cap\{V\geq0\}$, which are Borel measurable. Then by using the Jensen's inequality, we have
%\begin{align}
% m(\Omega^k_+)=\int_{\Omega^k_+}e^{-(V)_+}\,d\mathcal{H}^n\leq e^{-1/k}\mathcal{H}^n(\Omega_+^k)\leq e^{-1/k}\mathcal{H}^n(\Omega)\notag
%\end{align}
%holds. Since $\Omega_+=\cup_{k\geq 1}\Omega_+^k$, we have \red{I have mistaken here }$m(\Omega_+)\leq \lim_{k\rightarrow \infty}e^{-1/k}\mathcal{H}^n(\Omega)=\red{\mathcal{H}^1(\Omega)}$. \red{we do not need this $--->$} \bl{Since $m\sim \mathcal{H}^n$, $\mathcal{H}^n(\Omega_+)=0$.} This contradicts (\ref{lem:KNconvex:eq}) \red{contradicts $m(\Omega)>0$ also we notice that in the integral only $\Omega_+$ counts.}. Thus $(V)_+$ is integrable on any bounded Borel set. Since both the positive and negative parts of $V$ is locally integrable, $V$ itself is locally integrable. 
 %%%%%%%%%%%%%%%%%%%%%%%%%%%%%%%%%%%%%%%%%%%%%%%%%%%%%%%%%%%%%%%%%%%%%%%%%%%%%%%%%%%%  
 \par  Take two distinct \emph{Lebesgue points} $x_0,x_1 \in [a , b]$ of $V$ with respect to $\mathcal{H}^1$, that is, to assume
   %%%%%%%%%%%%%%%%%%%%%%%%%%%%%%%%%%%%%%%%%%%%%%%%%%%%%%%%%%%%%%%%%%%%%%%%%%%%%%%%%%%%
   %%%%%%%%%%%%%%%%%%%%%%%%%%%%%%%%%%%%%%%%%%%%%%%%%%%%%%%%%%%%%%%%%%%%%%%%%%%%%% 
   \begin{align}
    V(x_i)=\lim_{r\rightarrow 0}\dashint_{B_r(x_i)}V(x)\,\mathcal{H}^1(dx)\label{eq:lebesguepoints} \quad,\quad i = 1,2.
   \end{align}  
   %%%%%%%%%%%%%%%%%%%%%%%%%%%%%%%%%%%%%%%%%%%%%%%%%%%%%%%%%%%%%%%%%%%%%%%%%%%%%%%%%%%%
   %%%%%%%%%%%%%%%%%%%%%%%%%%%%%%%%%%%%%%%%%%%%%%%%%%%%%%%%%%%%%%%%%%%%%%%%%%%%%% 
Note that by the Lebesgue differentiation theorem,
\be
\mathcal{H}^1([a,b]\setminus \{x\in\R^1\;;\;x\text{ satisfies (\ref{eq:lebesguepoints})}\})=0. \notag
\ee
%%%%%%%%%%%%%%%%%%%%%%%%%%%%%%%%%%%%%%%%%%%%%%%%%%%%%%%%%%%%%%%%%%%%%%%%%%%%%%%%%%%%
   %%%%%%%%%%%%%%%%%%%%%%%%%%%%%%%%%%%%%%%%%%%%%%%%%%%
Set $\mu^r_i:=\left(\mathcal{H}^1(B_r(x_i))\right)^{-1}\mathcal{H}^1\vert_{B_r(x_i)}$, where $r$ is chosen small enough so that $B_r(x_0)\cap B_r(x_1)=\emptyset$. Let $Ent(\mu \vert \nu) : = \int_{X} \; \frac{d \mu }{ d \nu }  \log \left( \frac{d \mu }{ d \nu } \right)  \; d\nu$. Since 
%%%%%%%%%%%%%%%%%%%%%%%%%%%%%%%%%%%%%%%%%%%%%%%%%%%%%%%%%%%%%%%%%%%%%%%%%%%%%%%%%%%%
\be
Ent(\mu\vert e^{-V}\mathcal{H}^1)=Ent(\mu\vert\mathcal{H}^1)+\int V\,d\mu, \notag
\ee
and by factoring in the$(K,N)$-convexity of the Entropy functional, we get 
   %%%%%%%%%%%%%%%%%%%%%%%%%%%%%%%%%%%%%%%%%%%%%%%%%%%%%%%%%%%%%%%%%%%%%%%%%%%%%%%%%%%%
   %%%%%%%%%%%%%%%%%%%%%%%%%%%%%%%%%%%%%%%%%%%%%%%%%%%%%%%%%%%%%%%%%%%%%%%%%%%%%%
   \begin{align}\label{eq:convex}
    &\exp\left(-\frac{1}{N}Ent(\mu^r_t\vert\mathcal{H}^1)\right)\exp\left(-\frac{1}{N}\int V\,d\mu^r_t\right)\notag\\
    &\geq \sigma^{(1-t)}_{K,N}(W_2(\mu_0^r,\mu_1^r))\exp\left(-\frac{1}{N}Ent(\mu^r_0\vert\mathcal{H}^1)\right)\exp\left(-\frac{1}{N}\int V\,d\mu_0^r\right) \\
    &+\sigma^{(t)}_{K,N}(W_2(\mu_0^r,\mu_1^r))\exp\left(-\frac{1}{N}Ent(\mu^r_1\vert\mathcal{H}^1)\right)\exp\left(-\frac{1}{N}\int V\,d\mu_1^r\right). \notag
    \end{align}
   %%%%%%%%%%%%%%%%%%%%%%%%%%%%%%%%%%%%%%%%%%%%%%%%%%%%%%%%%%%%%%%%%%%%%%%%%%%%%%%%%%%%
   %%%%%%%%%%%%%%%%%%%%%%%%%%%%%%%%%%%%%%%%%%%%%%%%%%%%%%%%%%%%%%%%%%%%%%%%%%%%%%
It is easy to see that $W_2(\mu^r_0,\mu^r_1)=d_E(x_0,x_1)$. Moreover, the measure $\mu_t^r$ can be written as $\mu^r_t=\left(\mathcal{H}^1(B_r(x_t))\right)^{-1}\mathcal{H}^1\vert_{B_r(x_t)}$, where $x_t:=(1-t)x_0+tx_1$. Thus, we compute 
   %%%%%%%%%%%%%%%%%%%%%%%%%%%%%%%%%%%%%%%%%%%%%%%%%%%%%%%%%%%%%%%%%%%%%%%%%%%%%%%%%%%%
   %%%%%%%%%%%%%%%%%%%%%%%%%%%%%%%%%%%%%%%%%%%%%%%%%%%%%%%%%%%%%%%%%%%%%%%%%%%%%%
   \begin{align}
    Ent(\mu_t^r\vert\mathcal{H}^1)&=\dashint_{B_r(x_t)}\log\frac{1}{\mathcal{H}^1(B_r(x_t))}\mathcal{H}^1(dx)=\log\frac{1}{\mathcal{H}^1(B_r(x_t))} \notag \\
    &=Ent(\mu^r_0\vert\mathcal{H}^1)=Ent(\mu^r_1\vert\mathcal{H}^1).\notag
   \end{align}
   %%%%%%%%%%%%%%%%%%%%%%%%%%%%%%%%%%%%%%%%%%%%%%%%%%%%%%%%%%%%%%%%%%%%%%%%%%%%%%%%%%%%
   %%%%%%%%%%%%%%%%%%%%%%%%%%%%%%%%%%%%%%%%%%%%%%%%%%%%%%%%%%%%%%%%%%%%%%%%%%%%%%
   Taking the limsup of (\ref{eq:convex}) as $r\rightarrow 0$, one gets 
   %%%%%%%%%%%%%%%%%%%%%%%%%%%%%%%%%%%%%%%%%%%%%%%%%%%%%%%%%%%%%%%%%%%%%%%%%%%%%%%%%%%%
   %%%%%%%%%%%%%%%%%%%%%%%%%%%%%%%%%%%%%%%%%%%%%%%%%%%%%%%%%%%%%%%%%%%%%%%%%%%%%%
   \begin{align}
    &\exp\left(-\frac{1}{N}\liminf_{r\rightarrow0}\int V\,d\mu_t^r\right) \notag \\
    &\geq \sigma^{(1-t)}_{K,N}(d_E(x_0,x_1))\exp\left(-\frac{1}{N}V(x_0)\right)+\sigma^{(t)}_{K,N}(d_E(x_0,x_1))\exp\left(-\frac{1}{N}V(x_1)\right).\notag
   \end{align}
   %%%%%%%%%%%%%%%%%%%%%%%%%%%%%%%%%%%%%%%%%%%%%%%%%%%%%%%%%%%%%%%%%%%%%%%%%%%%%%%%%%%%
   %%%%%%%%%%%%%%%%%%%%%%%%%%%%%%%%%%%%%%%%%%%%%%%%%%%%%%%%%%%%%%%%%%%%%%%%%%%%%%
In particular,
   %%%%%%%%%%%%%%%%%%%%%%%%%%%%%%%%%%%%%%%%%%%%%%%%%%%%%%%%%%%%%%%%%%%%%%%%%%%%%%%%%%%%
   %%%%%%%%%%%%%%%%%%%%%%%%%%%%%%%%%%%%%%%%%%%%%%%%%%%%%%%%%%%%%%%%%%%%%%%%%%%%%%
   \begin{align}\label{eq:Vconvex}
    &\exp\left(-\frac{1}{N}V(x_t)\right)\\
    &\geq \sigma^{(1-t)}_{K,N}(d_E(x_0,x_1))\exp\left(-\frac{1}{N}V(x_0)\right)+\sigma^{(t)}_{K,N}(d_E(x_0,x_1))\exp\left(-\frac{1}{N}V(x_1)\right). \notag
   \end{align}
   %%%%%%%%%%%%%%%%%%%%%%%%%%%%%%%%%%%%%%%%%%%%%%%%%%%%%%%%%%%%%%%%%%%%%%%%%%%%%%%%%%%%
   %%%%%%%%%%%%%%%%%%%%%%%%%%%%%%%%%%%%%%%%%%%%%%%%%%%%%%%%%%%%%%%%%%%%%%%%%%%%%%
   holds if $x_t$ is a Lebesgue point of $V$. Consider the function $W$ which is defined by 
  \begin{eqnarray}
 W(x):= \begin{cases} 
      V(x) & \text{if $x$ is a Lebesgue point of $V$,} \\
      \inf_{\{y_i\}}\left\{\liminf_{y_i\rightarrow x}V(y_i) \right\} & \text{otherwise,} \\
     \end{cases} \notag
\end{eqnarray}   
where, the infimum in the second line, is taken over all sequences $\{y_i\}$ approaching to $x$. By the definition of $W$ and by (\ref{eq:Vconvex}), we obtain 
   %%%%%%%%%%%%%%%%%%%%%%%%%%%%%%%%%%%%%%%%%%%%%%%%%%%%%%%%%%%%%%%%%%%%%%%%%%%%%%%%%%%%
   %%%%%%%%%%%%%%%%%%%%%%%%%%%%%%%%%%%%%%%%%%%%%%%%%%%%%%%%%%%%%%%%%%%%%%%%%%%%%%
   \begin{align}
    &\exp\left(-\frac{1}{N}W(x_t)\right) \notag\\
    &\geq \sigma^{(1-t)}_{K,N}(d_E(x_0,x_1))\exp\left(-\frac{1}{N}W(x_0)\right)+\sigma^{(t)}_{K,N}(d_E(x_0,x_1))\exp\left(-\frac{1}{N}W(x_1)\right).\notag
   \end{align}
   %%%%%%%%%%%%%%%%%%%%%%%%%%%%%%%%%%%%%%%%%%%%%%%%%%%%%%%%%%%%%%%%%%%%%%%%%%%%%%%%%%%%
   %%%%%%%%%%%%%%%%%%%%%%%%%%%%%%%%%%%%%%%%%%%%%%%%%%%%%%%%%%%%%%%%%%%%%%%%%%%%%%
   Also $\mathcal{H}^1(\{V\neq W\})=0$ holds by the Lebesgue differentiation theorem. By \cite[Lemma 2.12]{EKS}, $W$ is a continuous function. The continuity of $W$ implies a lower boundedness of $W$ in any closed bounded convex set in $[a,b]$. This local boundedness together with \cite[Proposition 3.3]{EKS} will imply that the $RCD^*(K,N)$ condition is satisfied by $([a,b],d_E,e^{-W}\mathcal{H}^1)$. On letting $-a,b \rightarrow\infty$, we have 
   \begin{align}
    \mathcal{H}^1\left(\R\setminus \{x\in \R\;;\;x\text{ satisfies }(\ref{eq:lebesguepoints})\}\right)=0.\notag
   \end{align}
   The same argument above leads to the consequence.
  \end{proof}
   %%%%%%%%%%%%%%%%%%%%%%%%%%%%%%%%%%%%%%%%%%%%%%%%%%%%%%%%%%%%%%%%%%%%%%%%%%%%%%%%%%%%
   %%%%%%%%%%%%%%%%%%%%%%%%%%%%%%%%%%%%%%%%%%%%%%%%%%%%%%%%%%%%%%%%%%%%%%%%%%%%%%
   %%%%%%%%%%%%%%%%%%%%%%%%%%%%%%%%%%%%%%%%%%%%%%%%%%%%%%%%%%%%%%%%%%%%%%%%%%%%%%%%%%%%
  \begin{proof}[\textbf{Proof of Theorem \ref{thm:main-1}}]
  It is clear that (2) implies (3). Also (4) immediately implies (1), (2), and (3). According to the fact $m(X\setminus \RE)=0$, we know (3) $\Rightarrow$ (1). Finally Theorem \ref{local} says that (1) implies (2) and (4).

 %%%%%%%%%%%%%%%%%%%%%%%%%%%%%%%%%%%%%%%%%%%%%%%%%%%%%%%%%%%%%%%%%%%%%%%%%%%%%%%%%%%%
   %%%%%%%%%%%%%%%%%%%%%%%%%%%%%%%%%%%%%%%%%%%%%%%%%%%%%%%%%%%%%%%%%%%%%%%%%%%%%%
 Using Proposition \ref{prop:reference} and Lemma \ref{lem:KNconvex}, we know that $(X,d,m)$ is isomorphic to $(X,d,e^{-f}\mathcal{H}^1)$, where $f:X\rightarrow \R$ is a $(K,N)$-convex function provided that $(X,d)$ is isometric to $(\R,d_E)$. However, a similar argument works for $S^1$, $\R_{>0}$, and an open interval. Hence, each $(X,d,m)$ is written in the form $(X,d,e^{-f}\mathcal{H}^1)$, where $f:X\rightarrow \R\cup\{\pm\infty\}$ is $(K,N)$-convex and continuous on the \emph{interior} of $X$. 
   
  \end{proof}
  %%%%%%%%%%%%%%%%%%%%%%%%%%%%%%%%%%%%%%%%%%%%%%%%%%%%%%%%%%%%  
  
  \begin{proof}[\textbf{Proof of Corollary \ref{cor:lessthan2}}]
   By Lemma \ref{empty}, $\RE_j=\emptyset$ for any $j\geq 2$. Thus applying Theorem \ref{thm:main-1} let us obtain the consequence. 
  \end{proof}
 %%%%%%%%%%%%%%%%%%%%%%%%%%%%%%%%%%%%%%%%%%%%%%%%%%%%%%%%%%%%%%%%%%%%%%%%%%%%%%%%%%%%
   %%%%%%%%%%%%%%%%%%%%%%%%%%%%%%%%%%%%%%%%%%%%%%%%%%%%%%%%%%%%%%%%%%%%%%%%%%%%%%  
   %%%%%%%%%%%%%%%%%%%%%%%%%%%%%%%%%%%%%%%%%%%%%%%%%%%%%%%%%%%%%%%%%%%%%%%%%%%%%%%%%%%%%%%%%%%%%%%%%%%%%%%%%%%%%%%%%%%%%%%%%%%%%%%%%%%%%%%%%%%%%%%%%%%%%%%%
   %
   %   Section : Application to the non-possibility of collapsing  
   %
   %%%%%%%%%%%%%%%%%%%%%%%%%%%%%%%%%%%%%%%%%%%%%%%%%%%%%%%%%%%%%%%%%%%%%%%%%%%%%%%%%%%%%%%%%%%%%%%%%%%%%%%%%%%%%%%%%%%%%%%%%%%%%%%%%%%%%%%%%%%%%%%%%%%%%%%%
   \section{Spaces with Ricci curvature $\ge K > 0$, never collapse to circles}
   In~\cite[Section 5]{C-Large}, Colding proves that manifolds with positive Ricci curvature never collapse to a unit sphere of lower dimension. 
   \begin{thm}[Colding~\cite{C-Large}]
    Let $M^n_i$ be $n$-dimensional Riemannian manifolds with positive Ricci curvature $Ric_{M_i}\geq (n-1)$. Assume that $M^n_i$ converges to a unit sphere $S^m$. Then $n=m$.
   \end{thm}
   In this section, we present a totally different proof of this result when $m=1$ by taking advantage of the convexity of the potential function $V$. Moreover, our presented theorem is a bit stronger (see Remark~\ref{rem:colding}).
   \begin{thm}\label{thm:circlecollapse}
   Let $(X,d,m)$ be an $RCD^*(K,N)$ space for $K>0$, $N\in(1,\infty)$. Then $(X,d)$ is not isometric to a circle with its standard metric $(S^1(r),d)$ for any $r>0$. 
   \end{thm}
   \begin{proof}
    Suppose $(X,d,m)$ is isometric to $(S^1(r),d,m)$. For simplicity, we omit $r>0$. By Theorem~\ref{thm:main-1}, we are able to write $m=e^{-V}vol_{S^1}$ for a $(K,N)$-convex function $V$. First to see where the contradiction comes from, we assume $V\in C^2(S^1)$. Then $V$ satisfies the differential inequality $V''\geq K - (V')^2/N$ (see the equation (1.2) in \cite{EKS}). 
    Since $K>0$, we have $V''>0$ at critical points. On the other hand, $V$ has a maximal point $x_0\in S^1$ since $V$ is continuous and $S^1$ is compact. Therefore $V''(x_0)\leq 0$. This contradicts.
%%%%%%%%%%%%%%%%%%%%%%%%%%%%%%%%%%%%%%%%%%%%%%%%%%%%%%%%%%%%%%%%%%%%%%%%%%%%%%%%%%%%%%%%%%%%%%%%%%%%%%%%%%%%%%%%%%%%%%%%%%%%%%%%%%%%%%%%%%%%%%%%%%%%%%%%    
    
\par Now for general case, we know that $V$ is continuous (and in fact Lipschitz). Suppose $\bar x$ is a maximal point for $V$. Take $x_0,x_1$ with $d(x_0,\bar x)=d(x_1,\bar x)=d(x_0,x_1)/2$ and with $d(x_0,x_1)$ is sufficiently small. We may also assume $V(x_0)\leq V(x_1)$. By the definition of $(K,N)$-convexity, 
    \begin{align}
     &\exp\left(-\frac{1}{N}V(\bar x)\right)\notag\\
     &\geq \frac{\sin\left(\frac{d(x_0,x_1)}{2}\sqrt{\frac{K}{N}}\right)}{\sin\left(d(x_0,x_1)\sqrt{\frac{K}{N}}\right)}\exp\left(-\frac{1}{N}V(x_0)\right)+\frac{\sin\left(\frac{d(x_0,x_1)}{2}\sqrt{\frac{K}{N}}\right)}{\sin\left(d(x_0,x_1)\sqrt{\frac{K}{N}}\right)}\exp\left(-\frac{1}{N}V(x_1)\right)\notag\\
     &=\frac{1}{2\cos\left(\frac{d(x_0,x_1)}{2}\sqrt{\frac{K}{N}}\right)}\left(\exp\left(-\frac{1}{N}V(x_0)\right)+\exp\left(-\frac{1}{N}V(x_1)\right)\right)\notag\\
     &\geq \frac{1}{\cos\left(\frac{d(x_0,x_1)}{2}\sqrt{\frac{K}{N}}\right)}\exp\left(-\frac{1}{N}V(x_1)\right)\notag\\
     &\geq \frac{1}{\cos\left(\frac{d(x_0,x_1)}{2}\sqrt{\frac{K}{N}}\right)}\exp\left(-\frac{1}{N}V(\bar x)\right).\notag
    \end{align}  
    Since $0<d(x_0,x_1)\leq \pi\sqrt{N/K}$, 
    \begin{align}
     \cos\left(\frac{d(x_0,x_1)}{2}\sqrt{\frac{K}{N}}\right)<1.\notag
    \end{align}
which is a contradiction. 
   \end{proof}
%%%%%%%%%%%%%%%%%%%%%%%%%%%%%%%%%%%%%%%%%%
%%%%%%%%%%%%%%%%%%%%%%%%%%%%%%%%%%%%%%%%%%
 \begin{rem}\label{rem:colding}
  In Colding~\cite{C-Large}, sequences of $n$-dimensional \emph{closed} Riemannian manifolds with $Ric\geq n-1$ are considered. Our Theorem~\ref{thm:circlecollapse} also applies to weighted Riemannian manifolds with boundary as long as $RCD^*(K,N)$ condition for $K>0$ and $N\in (1,\infty)$ is satisfied. 
 \end{rem}

  %%%%%%%%%%%%%%%%%%%%%%%%%%%%%%%%%%%%%%%%%%
 
 						%% Section : B_G type inequalities %%
 
 %%%%%%%%%%%%%%%%%%%%%%%%%%%%%%%%%%%%%%%%%%
\section{Further information on the measures}

  \subsection{Bishop-Gromov type inequalities}
 
 In this section, we prove useful Bishop-Gromov type inequalities for $RCD^*(K,N)$ spaces.
 %%%%%%%%%%%%%%%%%%%%%%%%%%%%%%%%%%%%%%%%%%%%%
%%%%%%%%%%%%%%%%%%%%%%%%%%%%%%%%%%%%%%%%%%%%%
 \begin{defn}
  Let $(X,d,m)$ be a metric measure space. We define a \emph{boundary measure} ( known as \emph{co-dimension $1$} measure), $m_{-1}$, as follows. Let $\delta>0$ be a sufficiently small number. For a Borel set $A\subset X$, define
   %%%%%%%%%%%%%%%%%%%%%%%%%%%%%%%%%%%%%%%%%%%%%%%%%%%%%%%%%%%%%%%%%%%%%%%%%%%%%%%%%%%%
   %%%%%%%%%%%%%%%%%%%%%%%%%%%%%%%%%%%%%%%%%%%%%%%%%%%%%%%%%%%%%%%%%%%%%%%%%%%%%% 
  \begin{align}
   (m_{-1})_{\delta}(A):=\inf\left\{\sum_{i\in I}r_i^{-1}m(B_{r_i}(x_i))\;;\;r_i\leq \delta,\,\bigcup_{i\in I}B_{r_i}(x_i)\supset A,\,I\text{: countable }\right\}, \notag
  \end{align}
   %%%%%%%%%%%%%%%%%%%%%%%%%%%%%%%%%%%%%%%%%%%%%%%%%%%%%%%%%%%%%%%%%%%%%%%%%%%%%%%%%%%%
and, %%%%%%%%%%%%%%%%%%%%%%%%%%%%%%%%%%%%%%%%%%%%%%%%%%%%%%%%%%%%%%%%%%%%%%%%%%%%%%
  \be 
  m_{-1}(A):=\lim_{\delta\rightarrow0}(m_{-1})_{\delta}(A). \notag
  \ee
   %%%%%%%%%%%%%%%%%%%%%%%%%%%%%%%%%%%%%%%%%%%%%%%%%%%%%%%%%%%%%%%%%%%%%%%%%%%%%%%%%%%%
   %%%%%%%%%%%%%%%%%%%%%%%%%%%%%%%%%%%%%%%%%%%%%%%%%%%%%%%%%%%%%%%%%%%%%%%%%%%%%%
 \end{defn}
  %%%%%%%%%%%%%%%%%%%%%%%%%%%%%%%%%%%%%%%%%%%%%%%%%%%%%%%%%%%%%%%%%%%%%%%%%%%%%%%%%%%%
   %%%%%%%%%%%%%%%%%%%%%%%%%%%%%%%%%%%%%%%%%%%%%%%%%%%%%%%%%%%%%%%%%%%%%%%%%%%%%%
Let $S_{K,N}(t)$ for $N>1$, $K\in \R$ be the following: 
 %%%%%%%%%%%%%%%%%%%%%%%%%%%%%%%%%%%%%%%%%%%%%%%%%%%%%%%%%%%%%%%%%%%%%%%%%%%%%%%%%%%%
   %%%%%%%%%%%%%%%%%%%%%%%%%%%%%%%%%%%%%%%%%%%%%%%%%%%%%%%%%%%%%%%%%%%%%%%%%%%%%%
\begin{align}
 S_{K,N}(t):=\begin{cases}
                       \sqrt{\frac{N-1}{K}}\sin(t\sqrt{\frac{K}{N-1}})&\text{if }K>0,\\
                       t&\text{if }K=0,\\
                       \sqrt{\frac{N-1}{-K}}\sinh(t\sqrt{\frac{-K}{N-1}})&\text{if }K<0.
                     \end{cases} \notag
\end{align}
 %%%%%%%%%%%%%%%%%%%%%%%%%%%%%%%%%%%%%%%%%%%%%%%%%%%%%%%%%%%%%%%%%%%%%%%%%%%%%%%%%%%%
   %%%%%%%%%%%%%%%%%%%%%%%%%%%%%%%%%%%%%%%%%%%%%%%%%%%%%%%%%%%%%%%%%%%%%%%%%%%%%%
Bishop-Gromov type inequalities for boundary measures hold on Ricci limit spaces (see Honda~\cite{HBG}). The same is also true for $RCD^*(K,N)$ spaces. 
%%%%%%%%%%%%%%%%%%%%%%%%%%%%%%%%%%%%%%%%%%%%%
%%%%%%%%%%%%%%%%%%%%%%%%%%%%%%%%%%%%%%%%%%%%%
\begin{thm}\label{thm:BG1}
 Let $(X,d,m)$ be a metric measure space satisfying $RCD^*(K,N)$ condition and $m_{-1}$, the boundary measure. For any point $x_0\in X$ and any $t>0$, we have
  %%%%%%%%%%%%%%%%%%%%%%%%%%%%%%%%%%%%%%%%%%%%%%%%%%%%%%%%%%%%%%%%%%%%%%%%%%%%%%%%%%%%
   %%%%%%%%%%%%%%%%%%%%%%%%%%%%%%%%%%%%%%%%%%%%%%%%%%%%%%%%%%%%%%%%%%%%%%%%%%%%%%
 \begin{align}
  m_{-1}(\partial B_t(x_0))\leq 2 \cdot 5^{N-1} \cdot m(B_t(x_0))\frac{S_{K,N}(t)^{N-1}}{\int_0^tS_{K,N}(r)^{N-1}\,dr}.  \label{eq:BGforbdry}
 \end{align}
\end{thm}
%%%%%%%%%%%%%%%%%%%%%%%%%%%%%%%%%%%%%%%%%%%%%
%%%%%%%%%%%%%%%%%%%%%%%%%%%%%%%%%%%%%%%%%%%%%
\begin{proof}
Let $F(r):=\int_0^rS_{K,N}(s)^{N-1}\,ds$ and fix $x_0\in X$, $t>0$. Let $\delta>0$ be a small positive number satisfying $0<\delta<t/200$. It is trivial that 
 %%%%%%%%%%%%%%%%%%%%%%%%%%%%%%%%%%%%%%%%%%%%%%%%%%%%%%%%%%%%%%%%%%%%%%%%%%%%%%%%%%%%
   %%%%%%%%%%%%%%%%%%%%%%%%%%%%%%%%%%%%%%%%%%%%%%%%%%%%%%%%%%%%%%%%%%%%%%%%%%%%%%
 \begin{align}
  \bigcup_{x\in\partial B_t(x_0)}B_{\delta}(x)\supset \partial B_t(x_0). \notag
 \end{align}
  %%%%%%%%%%%%%%%%%%%%%%%%%%%%%%%%%%%%%%%%%%%%%%%%%%%%%%%%%%%%%%%%%%%%%%%%%%%%%%%%%%%%
   %%%%%%%%%%%%%%%%%%%%%%%%%%%%%%%%%%%%%%%%%%%%%%%%%%%%%%%%%%%%%%%%%%%%%%%%%%%%%%
 Since $\partial B_t(x_0)$ is compact, we can apply a covering lemma argument (as in \cite[Theorem 2.2.3.]{AT}) to get a finite family of points $\{x_i\}_{i\in I}\subset \partial B_t(x_0)$ such that $\{B_{\delta}(x_i)\}_{i\in I}$ are mutually disjoint and $\cup_iB_{5\delta}(x_i)\supset\partial B_t(x_0)$ holds. It is clear that $B_{\delta}(x_i)\subset B_{t+\delta}(x_0)\setminus B_{t-\delta}(x_0)$. By the Bishop-Gromov inequality, we obtain 
  %%%%%%%%%%%%%%%%%%%%%%%%%%%%%%%%%%%%%%%%%%%%%%%%%%%%%%%%%%%%%%%%%%%%%%%%%%%%%%%%%%%%
   %%%%%%%%%%%%%%%%%%%%%%%%%%%%%%%%%%%%%%%%%%%%%%%%%%%%%%%%%%%%%%%%%%%%%%%%%%%%%%
 \begin{align}
  m(B_{t+\delta}(x_0))\leq \frac{F(t+\delta)}{F(t-\delta)}m(B_{t-\delta}(x_0)). \notag
 \end{align}
  %%%%%%%%%%%%%%%%%%%%%%%%%%%%%%%%%%%%%%%%%%%%%%%%%%%%%%%%%%%%%%%%%%%%%%%%%%%%%%%%%%%%
   %%%%%%%%%%%%%%%%%%%%%%%%%%%%%%%%%%%%%%%%%%%%%%%%%%%%%%%%%%%%%%%%%%%%%%%%%%%%%%
 Since $F$ is smooth,
  %%%%%%%%%%%%%%%%%%%%%%%%%%%%%%%%%%%%%%%%%%%%%%%%%%%%%%%%%%%%%%%%%%%%%%%%%%%%%%%%%%%%
   %%%%%%%%%%%%%%%%%%%%%%%%%%%%%%%%%%%%%%%%%%%%%%%%%%%%%%%%%%%%%%%%%%%%%%%%%%%%%% 
 \begin{align}
  1-\frac{F(t-\delta)}{F(t+\delta)}=2\delta\cdot\frac{F'(t-\delta)}{F(t+\delta)}+o(\delta),\label{Taylor}
 \end{align}
  %%%%%%%%%%%%%%%%%%%%%%%%%%%%%%%%%%%%%%%%%%%%%%%%%%%%%%%%%%%%%%%%%%%%%%%%%%%%%%%%%%%%
   %%%%%%%%%%%%%%%%%%%%%%%%%%%%%%%%%%%%%%%%%%%%%%%%%%%%%%%%%%%%%%%%%%%%%%%%%%%%%%
 holds by the Taylor expansion at $t-\delta$. Then from (\ref{Taylor}), we compute
  %%%%%%%%%%%%%%%%%%%%%%%%%%%%%%%%%%%%%%%%%%%%%%%%%%%%%%%%%%%%%%%%%%%%%%%%%%%%%%%%%%%%
   %%%%%%%%%%%%%%%%%%%%%%%%%%%%%%%%%%%%%%%%%%%%%%%%%%%%%%%%%%%%%%%%%%%%%%%%%%%%%%
 \begin{align}
  &m(B_{t+\delta}(x_0)\setminus B_{t-\delta}(x_0))=m(B_{t+\delta}(x_0))-m(B_{t-\delta}(x_0))\notag\\
  &\leq m(B_{t+\delta}(x_0))-\frac{F(t-\delta)}{F(t+\delta)}m(B_{t+\delta}(x_0)) \notag\\
  &=2\delta\cdot\frac{F'(t-\delta)}{F(t+\delta)}m(B_{t+\delta}(x_0))+o(\delta).\notag
 \end{align}
  %%%%%%%%%%%%%%%%%%%%%%%%%%%%%%%%%%%%%%%%%%%%%%%%%%%%%%%%%%%%%%%%%%%%%%%%%%%%%%%%%%%%
   %%%%%%%%%%%%%%%%%%%%%%%%%%%%%%%%%%%%%%%%%%%%%%%%%%%%%%%%%%%%%%%%%%%%%%%%%%%%%%
Therefore, 
  %%%%%%%%%%%%%%%%%%%%%%%%%%%%%%%%%%%%%%%%%%%%%%%%%%%%%%%%%%%%%%%%%%%%%%%%%%%%%%%%%%%%
   %%%%%%%%%%%%%%%%%%%%%%%%%%%%%%%%%%%%%%%%%%%%%%%%%%%%%%%%%%%%%%%%%%%%%%%%%%%%%%
 \begin{align}
   (m_{-1})_{\delta}(\partial B_t(x_0))&\leq \sum_{i\in I}(5\delta)^{-1}m(B_{5\delta}(x_i))\notag\\
  &\leq (5\delta)^{-1}\frac{F(5\delta)}{F(\delta)}\sum_{i\in I}m(B_{\delta}(x_i))\notag\\
  &=(5\delta)^{-1}\frac{F(5\delta)}{F(\delta)}m\left(\bigcup_{i\in I}B_{\delta}(x_i)\right) \notag\\
  &\leq (5\delta)^{-1}\frac{F(5\delta)}{F(\delta)}m(B_{t+\delta}(x_0)\setminus B_{t-\delta}(x_0))\label{eq:comp}\\
  &\leq 2 \cdot  5^{N-1} \cdot\frac{F'(t-\delta)}{F(t+\delta)}m(B_{t+\delta}(x_0))+\frac{o(\delta)}{\delta}\label{eq:b-g}.
 \end{align}
  %%%%%%%%%%%%%%%%%%%%%%%%%%%%%%%%%%%%%%%%%%%%%%%%%%%%%%%%%%%%%%%%%%%%%%%%%%%%%%%%%%%%
   %%%%%%%%%%%%%%%%%%%%%%%%%%%%%%%%%%%%%%%%%%%%%%%%%%%%%%%%%%%%%%%%%%%%%%%%%%%%%%
 Letting $\delta \to 0$ in (\ref{eq:b-g}), we get (\ref{eq:BGforbdry}). 
  %%%%%%%%%%%%%%%%%%%%%%%%%%%%%%%%%%%%%%%%%%%%%%%%%%%%%%%%%%%%%%%%%%%%%%%%%%%%%%%%%%%%
   %%%%%%%%%%%%%%%%%%%%%%%%%%%%%%%%%%%%%%%%%%%%%%%%%%%%%%%%%%%%%%%%%%%%%%%%%%%%%%
\end{proof}
 %%%%%%%%%%%%%%%%%%%%%%%%%%%%%%%%%%%%%%%%%%%%%%%%%%%%%%%%%%%%%%%%%%%%%%%%%%%%%%%%%%%%
   %%%%%%%%%%%%%%%%%%%%%%%%%%%%%%%%%%%%%%%%%%%%%%%%%%%%%%%%%%%%%%%%%%%%%%%%%%%%%%
\begin{rem}
Suppose $\partial B_t(x_0)$ is inside a geodesically convex subset, $X'$ with $\diam(X') \le D$. Then, In the virtue of the L\'{e}vy-Gromov isoperimetric inequality for $RCD^*(K,N)$ spaces that is proven in Cavalletti-Mondino~\cite{CM-1}, one gets
\be
	m^{+}\left( B_t(x_0) \right) \ge \mathcal{I}_{K,N,D} \big( m \left( B_t(x_0) \big) \right). 
\ee
where, $\mathcal{I}_{K,N,D} (\cdot)$ is the Milman's model isoperimetric profile (see~Cavalletti-Mondino~\cite{CM-1} and Milman~\cite{Mil-iso} for the precise definitions). Our Theorem~\ref{thm:BG1}, in contrast to the L\'{e}vy-Gromov isoperimetric inequality, provides an upper bound for the surface measure $m_{-1}\left( \partial B_t(x_0) \right)$ in terms of $m \left( B_t(x_0) \right)$. 
Notice that the two surface measures, $m^+$ and $m_{-1}$ are a priori different but comparable in one direction on spheres. Indeed, by (\ref{eq:comp}), we have 
\begin{align}
 m_{-1}(\partial B_t(x_0))\leq 5^{N-1}\lim_{\delta\rightarrow 0}\frac{m(B_{t+\delta}(x_0)\setminus B_{t-\delta}(x_0))}{\delta}.\notag
\end{align}
Since 
\begin{align}
 m(B_{t+\delta}\setminus B_{t-\delta})&=m(B_{t+\delta})-m(B_t)\notag\\
 &\leq m(B_{t+\delta})-\frac{F(t-\delta)}{F(t)}m(B_t)\notag\\
 &=m(B_{t+\delta}\setminus B_t)+\frac{F(t)-F(t-\delta)}{F(t)}m(B_t),\notag
\end{align}
we obtain 
\begin{align}
 m_{-1}(\partial B_t(x_0))\leq 5^{N-1}\left(m^+(B_t(x_0))+F'(t)\frac{m(B_t(x_0))}{F(t)}\right).\notag
\end{align}
%
%Notice that the two surface measures, $m^+$ and $m_{-1}$ are comparable by Theroem~\ref{thm:comparable} below.
\end{rem}
%%%%%%%%%%%%%%%%%%%%%%%%%%%%%%%%%%%%%%%%%%%%
%%%%%%%%%%%%%%%%%%%%%%%%%%%%%%%%%%%%%%%%%%%%%
%%%%%%%%%%%%%%%%%%%%%%%%%%%%%%%%%%%%%%%%%%%%%
%%%%%%%%%%%%%%%%%%%%%%%%%%%%%%%%%%%%%%%%%%%%%
A direct consequence of the inequality (\ref{eq:BGforbdry}) is the following.
%%%%%%%%%%%%%%%%%%%%%%%%%%%%%%%%%%%%%%%%%%%%%
%%%%%%%%%%%%%%%%%%%%%%%%%%%%%%%%%%%%%%%%%%%%%
\begin{cor}[finiteness of boundary measures]\label{cor:finiteness}
 For an $RCD^*(K,N)$ space $(X,d,m)$, the mass of the boundary of a ball, measured by the boundary measure $m_{-1}$, is always finite.  
\end{cor}
%%%%%%%%%%%%%%%%%%%%%%%%%%%%%%%%%%%%%%%%%%%%%
%%%%%%%%%%%%%%%%%%%%%%%%%%%%%%%%%%%%%%%%%%%%%
\begin{cor}[Bishop-Gromov type inequality]\label{cor:BG}
 Let $(X,d,m)$ be an $RCD^*(K,N)$ space with $\supp m=X$. Assume that $X$ is not the single point space. 
 Fix a point $y\in X$. Then, for any $R>0$ and each $x\in B_R(y)$, there exists a constant $C=C(R,y)$ such that, 
  %%%%%%%%%%%%%%%%%%%%%%%%%%%%%%%%%%%%%%%%%%%%%%%%%%%%%%%%%%%%%%%%%%%%%%%%%%%%%%%%%%%%
   %%%%%%%%%%%%%%%%%%%%%%%%%%%%%%%%%%%%%%%%%%%%%%%%%%%%%%%%%%%%%%%%%%%%%%%%%%%%%%
 \begin{align}
  m(B_s(x))\leq Cs, \label{cor:BG:eq1}
 \end{align}
  %%%%%%%%%%%%%%%%%%%%%%%%%%%%%%%%%%%%%%%%%%%%%%%%%%%%%%%%%%%%%%%%%%%%%%%%%%%%%%%%%%%%
   %%%%%%%%%%%%%%%%%%%%%%%%%%%%%%%%%%%%%%%%%%%%%%%%%%%%%%%%%%%%%%%%%%%%%%%%%%%%%%
 holds for any $s\in (0,1]$. Moreover,
  %%%%%%%%%%%%%%%%%%%%%%%%%%%%%%%%%%%%%%%%%%%%%%%%%%%%%%%%%%%%%%%%%%%%%%%%%%%%%%%%%%%%
   %%%%%%%%%%%%%%%%%%%%%%%%%%%%%%%%%%%%%%%%%%%%%%%%%%%%%%%%%%%%%%%%%%%%%%%%%%%%%% 
 \begin{align}
  m_{-1}(\partial B_s(x))\leq C. \label{cor:BG:eq2}
 \end{align} \\
  %%%%%%%%%%%%%%%%%%%%%%%%%%%%%%%%%%%%%%%%%%%%%%%%%%%%%%%%%%%%%%%%%%%%%%%%%%%%%%%%%%%%
   %%%%%%%%%%%%%%%%%%%%%%%%%%%%%%%%%%%%%%%%%%%%%%%%%%%%%%%%%%%%%%%%%%%%%%%%%%%%%%
\end{cor}
%%%%%%%%%%%%%%%%%%%%%%%%%%%%%%%%%%%%%%%%%%%%%
%%%%%%%%%%%%%%%%%%%%%%%%%%%%%%%%%%%%%%%%%%%%%
\begin{proof}
 Once we prove (\ref{cor:BG:eq1}), (\ref{cor:BG:eq2}) will directly follow by using Theorem \ref{thm:BG1}. Fix $y\in X$ and $R>0$. Take $x_0\in B_R(y)$ with $d(x_0,y)=t$. Choose $0<\delta<t/200$ such that $B_{\delta}(x_0)\subset B_{t+\delta}(y)\setminus B_{t-\delta}(y)$. Since $\partial B_t(y)\setminus B_{3\delta}(x_0)$ is compact, using the same covering argument as in the proof of Theorem \ref{thm:BG1}, we can find a finite family of points $\{x_i\}_{i\in I}\subset \partial B_t(y)\setminus B_{3\delta}(x_0)$ such that $\{B_{\delta}(x_i)\}_{i\in I}$ are mutually disjoint and $ \partial B_t(y)\setminus B_{3\delta}(x_0) \subset \cup_{i\in I}B_{5\delta}(x_i)$. Note that, by the construction, $B_{\delta}(x_0)\cap B_{\delta}(x_i)=\emptyset$ for any $i\in I$ and $\left( \cup_{i\in I}B_{5\delta}(x_i) \right) \cup B_{5\delta}(x_0)\supset \partial B_t(y)$. Thus, repeating the same calculation as in the proof of Theorem \ref{thm:BG1}, we write
 \begin{align}
  \frac{m(B_{5\delta}(x_0))}{5\delta}&\leq \frac{m(B_{5\delta}(x_0))}{5\delta}+\sum_{i\in I}(5\delta)^{-1}m(B_{5\delta}(x_i))\notag\\
  &\leq 2\cdot 5^{N-1}\cdot\frac{F'(t-\delta)}{F(t+\delta)}m(B_{t+\delta}(y))+\frac{o(\delta)}{\delta}.\notag
 \end{align}
 Upon letting $\delta\rightarrow0$, we obtain 
 \begin{align}
  \limsup_{\delta\rightarrow0}\frac{m(B_{5\delta}(x_0))}{5\delta}\leq 2\cdot 5^{N-1}\frac{F'(t)}{F(t)}m(B_t(y)).\label{cor:BG:eq}
 \end{align}
Notice that, these calculations actually imply that the \emph{ small scale } volume growth at any point is \emph{at most linear} so we can write $m(B_t(y))\leq Ct$ for some $C>0$.
%%%%%%%%%%%%%%%%%%%%%%%%%%%%%%%%%%%%%%%%%%%%%
%%%%%%%%%%%%%%%%%%%%%%%%%%%%%%%%%%%%%%%%%%%%%
\par Also notice that
\begin{align}
  \lim_{t\rightarrow0}\frac{tF'(t)}{F(t)}\leq C(K,N)<\infty.\notag
 \end{align}
 %%%%%%%%%%%%%%%%%%%%%%%%%%%%%%%%%%%%%%%%%%%%%
%%%%%%%%%%%%%%%%%%%%%%%%%%%%%%%%%%%%%%%%%%%%%
Therefore, the RHS of (\ref{cor:BG:eq}) is bounded by 
\be
 C(K,N,R) := 2 \cdot 5^{N-1} \sup_{t\in (0,R]}tF'(t)/F(t) <\infty. \notag
\ee
 Hence,
 \begin{align}
  m(B_{\delta}(x_0))\leq C(R,y)\delta ,\notag
 \end{align} 
 holds for small $\delta>0$. 
%%%%%%%%%%%%%%%%%%%%%%%%%%%%%%%%%%%%%%%%%%%%%
%%%%%%%%%%%%%%%%%%%%%%%%%%%%%%%%%%%%%%%%%%%%%
 The inequality (\ref{Taylor}) and the proof of Theorem \ref{thm:BG1} give the conclusion. 
\end{proof}
%%%%%%%%%%%%%%%%%%%%%%%%%%%%%%%%%%%%%%%%%%%%%
%%%%%%%%%%%%%%%%%%%%%%%%%%%%%%%%%%%%%%%%%%%%%
\begin{cor}\label{cor:main1}
Let $\left( X , d , m \right)$ be an $RCD^*(K,N)$ space. Let $\left( W, d_W , m_W , \omega \right)$ be a pointed proper geodesic metric measure space. Assume that 
 %%%%%%%%%%%%%%%%%%%%%%%%%%%%%%%%%%%%%%%%%%%%%%%%%%%%%%%%%%%%%%%%%%%%%%%%%%%%%%%%%%%%
   %%%%%%%%%%%%%%%%%%%%%%%%%%%%%%%%%%%%%%%%%%%%%%%%%%%%%%%%%%%%%%%%%%%%%%%%%%%%%%
\be
\left( \R^k \times W, d_E \times d_W , \mathcal{L}^k \times m_W , \left( 0_E , \omega \right)  \right), \notag
\ee
 %%%%%%%%%%%%%%%%%%%%%%%%%%%%%%%%%%%%%%%%%%%%%%%%%%%%%%%%%%%%%%%%%%%%%%%%%%%%%%%%%%%%
   %%%%%%%%%%%%%%%%%%%%%%%%%%%%%%%%%%%%%%%%%%%%%%%%%%%%%%%%%%%%%%%%%%%%%%%%%%%%%%
is a tangent cone at $x \in X$. Then,
 %%%%%%%%%%%%%%%%%%%%%%%%%%%%%%%%%%%%%%%%%%%%%%%%%%%%%%%%%%%%%%%%%%%%%%%%%%%%%%%%%%%%
   %%%%%%%%%%%%%%%%%%%%%%%%%%%%%%%%%%%%%%%%%%%%%%%%%%%%%%%%%%%%%%%%%%%%%%%%%%%%%%
\be
		\limsup_{\delta \to 0} \frac{m_W ({B_\delta (w)})}{\delta} \le C(d,R) < \infty. \notag
\ee
 %%%%%%%%%%%%%%%%%%%%%%%%%%%%%%%%%%%%%%%%%%%%%%%%%%%%%%%%%%%%%%%%%%%%%%%%%%%%%%%%%%%%
   %%%%%%%%%%%%%%%%%%%%%%%%%%%%%%%%%%%%%%%%%%%%%%%%%%%%%%%%%%%%%%%%%%%%%%%%%%%%%%
\end{cor} 
%%%%%%%%%%%%%%%%%%%%%%%%%%%%%%%%%%%%%%%%%%%%%
%%%%%%%%%%%%%%%%%%%%%%%%%%%%%%%%%%%%%%%%%%%%% 
 \begin{proof}
It is implicit in the splitting theorem applied to $(\R^k\times W,d_E\times d_W,\mathcal{L}^k\times m_W,(0_E,w))$, that $(W,d_W,m_W,w)$ is an $RCD^*(0,N-k)$ space. The desired conclusion, then, follows from Corollary \ref{cor:BG}. 
 \end{proof}
%%%%%%%%%%%%%%%%%%%%%%%%%%%%%%%%%%%%%%%%%%%%%
%%%%%%%%%%%%%%%%%%%%%%%%%%%%%%%%%%%%%%%%%%%%% 

%%%%%%%%%%%%%%%%%%%%%%%%%%%%%%%%%%%%%%%%%%%%%%%%%%%%%%%%%%%%%%%%%%%%%%%%%%%%%%%%%%%%%%%%%%%%%%%%%%%%%%
 \subsection{Higher dimensional case} 
%%%%%%%%%%%%%%%%%%%%%%%%%%%%%%%%%%%%%%%%%%%%%
%%%%%%%%%%%%%%%%%%%%%%%%%%%%%%%%%%%%%%%%%%%%% 
\begin{prop}\label{cor:main2}
	Let $x$ be a point in $ \We{1}$. Then 
	\be
		\liminf_{r \to 0} \frac{m \left( B_r(x) \right)}{r} = 0. \label{cor:main2:eq}
	\ee
\end{prop}
%%%%%%%%%%%%%%%%%%%%%%%%%%%%%%%%%%%%%%%%%%%%%
%%%%%%%%%%%%%%%%%%%%%%%%%%%%%%%%%%%%%%%%%%%%% 
\begin{proof}
%Suppose (\ref{cor:main2:eq}) does not hold. 
By the definition, there exist a sequence of positive numbers $\{r_i\}$ tending to $0$ as $i\rightarrow\infty$ and a proper geodesic space $\left( W , d_W , m_W \right)$ such that 
\be
(X,d_{r_i},m^x_{r_i},x)\rightarrow\left( \R \times W, d_E \times d_W , \mathcal{L}^1 \times m_W , \left( 0_E , \omega \right)  \right), \notag
\ee
 in the measured Gromov-Hausdorff sense. In the virtue of Corollary \ref{cor:main1}, $m_W(B_r(w))\leq Cr$. Since $B_r(0_E,w)\subset B_{\sqrt{2}r}(0_E)\times B_{\sqrt{2}r}(w)$, we obtain 
\begin{align}
 \mathcal{L}^1\times m_W(B_r(0_E,w))\leq \mathcal{L}^1(B_{\sqrt{2}r}(0_E))m_W(B_{\sqrt{2}r}(w))\leq Cr^2.\notag
\end{align}
Note that Corollary \ref{cor:BG} implies $m(B_r(x))\leq Cr$. Therefore for given arbitrary $\epsilon>0$, 
 \begin{align}
 &\liminf_{r\rightarrow 0}\frac{m(B_r(x))}{r}\leq \lim_{i\rightarrow\infty}\frac{m\left(B_{\epsilon r_i}(x)\right)}{\epsilon r_i}\notag\\
 &=\lim_{i\rightarrow\infty}\frac{m^x_{r_i}\left(B^{d_{r_i}}_{\epsilon}(x)\right)}{\epsilon r_i}\cdot\int_{B_{r_i}(x)}1-\frac{1}{r_i}d(x,\cdot)\,dm\notag\\
 &\leq C\lim_{i\rightarrow\infty}\frac{m^x_{r_i}\left(B^{d_{r_i}}_{\epsilon}(x)\right)}{\epsilon}\notag\\
 &=C\frac{\mathcal{L}^1\times m_W\left(B_{\epsilon}(0_E,w)\right)}{\epsilon}\notag\\
 &\leq C'\epsilon , \notag
 \end{align}
 holds. The arbitrariness of $\epsilon$ immediately implies (\ref{cor:main2:eq}). 
\end{proof}
%%%%%%%%%%%%%%%%%%%%%%%%%%%%%%%%%%%%%%%%%%%%%%
%%%%%%%%%%%%%%%%%%%%%%%%%%%%%%%%%%%%%%%%%%%%%%
Consider a set $\mathcal{M}_1$ defined by 
\begin{align}
 \mathcal{M}_1:=\left\{x\in X\;;\;(\ref{cor:main2:eq})\text{ holds at }x\right\}.\notag
\end{align}
\begin{lem}\label{lem:conti}
 For given $r>0$, the function $x\mapsto m(B_r(x))/r$ is locally Lipschitz and in particular, locally uniformly continuous for $r > 0$.
\end{lem}
\begin{proof}
 A similar argument as in the proof of Theorem \ref{thm:BG1} can be applied here too ( also see Lemma 3.1 in \cite{Ki}). For the reader's convenience, we give a proof. The notations below are as in the proof of Theorem \ref{thm:BG1}. Fix a point $x\in X$. Take another point $y\in X$. For simplicity, set $d:=d(x,y)$. Take a midpoint $z\in X$, that is, $d(x,z)=d(z,y)=d(x,y)/2$. We have
 \begin{align}
  m(B_r(x))\leq m(B_{r+d/2}(z))\leq \frac{F(r+d/2)}{F(r-d/2)}m(B_{r-d/2}(z)).\notag
 \end{align} 
 Therefore,
 \begin{align}
  m\left(B_r(x)\setminus B_r(y)\right)&=m(B_r(x))-m(B_r(x)\cap B_r(y))\notag\\
  &\leq m(B_r(x))-m(B_{r-d/2}(z))\notag\\
  &\leq\left\{1-\frac{F(r-d/2)}{F(r+d/2)}\right\}m(B_r(x))\notag\\
  &=\left\{\frac{F'(r-d/2)}{F(r+d/2)}d+o(d)\right\}m(B_r(x)) , \notag
 \end{align}
 for small $d$. Interchanging the role of $x$ and $y$, gives
 \begin{align}
  \left\vert \frac{m(B_r(x))}{r}-\frac{m(B_r(y))}{r}\right\vert&\leq \frac{1}{r}m(B_r(x)\Delta B_r(y))\notag\\
  &\leq \frac{1}{r}\left\{\frac{F'(r-d/2)}{F(r+d/2)}d+o(d)\right\}\left(m(B_r(x))+m(B_r(y))\right)\notag\\
  &\leq C\left\{\frac{F'(r-d/2)}{F(r+d/2)}d+o(d)\right\}.\label{lem:conti:eq}
 \end{align}
 The right-hand side in (\ref{lem:conti:eq}) is independent of the choice of $x$, so using Corollary~\ref{cor:finiteness}, we have the conclusion. 
\end{proof}
%%%%%%%%%%%%%%%%%%%%%%%%%%%%%%%%%%%%%%%%%%%%%%
%%%%%%%%%%%%%%%%%%%%%%%%%%%%%%%%%%%%%%%%%%%%%%
\begin{rem}
In (\ref{lem:conti:eq}), we have $F'(r)/F(r)\rightarrow\infty$ as $r\rightarrow 0$ and therefore, it does not tell us anything about the modulus of continuity of $\frac{m\left( B_r(x)  \right)}{r}$ . If we, a priori, assume the uniform continuity for $r \ge 0$, we can prove that
\be 
\mathcal{M}_k := \left\{ x :  \liminf_{r \to 0}  \frac{m \left( B_r(x) \right)}{r^k} = 0   \right\}, \notag
\ee
 is closed.
\end{rem}
%%%%%%%%%%%%%%%%%%%%%%%%%%%%%%%%%%%%%%%%%%%%%%
%%%%%%%%%%%%%%%%%%%%%%%%%%%%%%%%%%%%%%%%%%%%%%
 \begin{prop}\label{prop:M-closed}
  Suppose $\frac{m\left( B_r(x)  \right)}{r}$ is uniformly continuous for $r \ge 0$, then $\mathcal{M}_1$ is a closed set.
 \end{prop}
 %%%%%%%%%%%%%%%%%%%%%%%%%%%%%%%%%%%%%%%%%%%%%%
%%%%%%%%%%%%%%%%%%%%%%%%%%%%%%%%%%%%%%%%%%%%%%
 \begin{proof}
  Suppose not. Let $x\in\overline{\mathcal{M}}_1\setminus \mathcal{M}_1$. Hence, there exists a constant $C>0$ such that $C\leq \liminf_{r\rightarrow 0}m(B_r(x))/r$. Take a sequence $y_i\in\mathcal{M}_1$ converging to $x$. For sufficiently small $r>0$, we have $C/2\leq m(B_r(x))/r$. By Lemma \ref{lem:conti}, 
\be
\vert m(B_r(x))-m(B_r(y_i))\vert\leq Cr/4 \quad \text{ for large $i$}. \notag
\ee  
 Therefore, we obtain 
  \begin{align}
   \frac{C}{2}\leq \frac{m(B_r(x))}{r}\leq \frac{m\left(B_r(y_i)\right)}{r}+\frac{C}{4} , \notag
  \end{align}
  for any small $r$. This contradicts $y_i\in\mathcal{M}_1$. 
 \end{proof}
%%%%%%%%%%%%%%%%%%%%%%%%%%%%%%%%%%%%%%%%%%%%%%
%%%%%%%%%%%%%%%%%%%%%%%%%%%%%%%%%%%%%%%%%%%%%%
\begin{cor}
 Let $(X,d,m)$ be an $RCD^*(K,N)$ space for $K\in\R$, $N\in(1,\infty)$. Assume that there exists a point $x\in X$ such that 
 \begin{align}
  \liminf_{r\rightarrow 0}\frac{m(B_r(x))}{r}>0.\notag
 \end{align}
 Then $(X,d,m)$ is isomorphic to one of the metric measure spaces given in Theorem \ref{thm:main-1}. 
\end{cor}
\begin{proof} 
 Since $\mathcal{M}_1$ is closed, $X\setminus \mathcal{M}_1$ is open. Therefore a small open neighbourhood, $U$, of $x$ is in $X\setminus \mathcal{M}_1$. Since $m(U)>0$, $m(U\cap \mathcal{R})>0$. However Proposition \ref{cor:main2} implies $U\subset X\setminus \We{1}\subset X\setminus \cup_{j\geq 2}\mathcal{R}_j$. Accordingly, $\mathcal{R}_1\neq\emptyset$. Theorem \ref{thm:main-1} implies the consequence.  
\end{proof}
%%%%%%%%%%%%%%%%%%%%%%%%%%%%%%%%%%%%%%%%%%%%%%
%%%%%%%%%%%%%%%%%%%%%%%%%%%%%%%%%%%%%%%%%%%%%%
We can generalize the statement of above propositions in the following way. Define
\begin{align}
 \mathcal{M}_k:=\left\{x\in X\;;\;\liminf_{r\rightarrow0}\frac{m(B_r(x))}{r^k}=0\right\}.\notag
\end{align}
The closeness of $\mathcal{M}_k$ can be proven just in the same way as in Proposition~\ref{prop:M-closed}. Then we conjecture:
%%%%%%%%%%%%%%%%%%%%%%%%%%%%%%%%%%%%%%%%%%%%%%
%%%%%%%%%%%%%%%%%%%%%%%%%%%%%%%%%%%%%%%%%%%%%%
\begin{conj}\label{conj:main1}
Suppose $\frac{m \left( B_r(x) \right)}{r^k}$ is uniformly continuous for $r \ge 0$, then
 \begin{align}\label{eq:conj}
  \We{k}\subset \mathcal{M}_k.
 \end{align}
\end{conj}
%%%%%%%%%%%%%%%%%%%%%%%%%%%%%%%%%%%%%%%%%%%%%%
%%%%%%%%%%%%%%%%%%%%%%%%%%%%%%%%%%%%%%%%%%%%%%
\begin{rem}
The Conjecture~\ref{conj:main1} is deeply related to a relation between given measure $m$ and Hausdorff measure on regular sets. We speculate that, (\ref{eq:conj}) being true, would imply that $m$ restricted to $\RE_k$ is an Ahlfors $k$-regular measure. (also see the related work by David~\cite{Guy}).
\end{rem}

\appendix
\section{Explicit details of the proof of Claim \ref{clm:RS}} \label{app:A}

Here, we will show that the $K-$ convexity of the entropy fails under the branching phenomenon (even when the branching time is not the same but rather within a short time interval) as in Claim~\ref{clm:RS}. One should keep the tripod example in mind while reading these computations. 
%%%%%%%%%%%%%%%%%%%%%%%%%%%%%%%%%%%%%%%%%%%%%%
%%%%%%%%%%%%%%%%%%%%%%%%%%%%%%%%%%%%%%%%%%%%%%
\par We will be using the same notations as in the Claim \ref{clm:RS} and almost the same calculations as in \cite{RS}.
%%%%%%%%%%%%%%%%%%%%%%%%%%%%%%%%%%%%%%%%%%%%%%
%%%%%%%%%%%%%%%%%%%%%%%%%%%%%%%%%%%%%%%%%%%%%%
\par As we observed in Claim~\ref{clm:RS}, one obtains two mutually singular measures $\pi^u$ and $\pi^d$ (in the tripod space analogy, the superscripts $d$ and $u$ mean \emph{up} and \emph{down} referring to the plans supported on either the \emph{upper} or \emph{lower}  branches). 
%%%%%%%%%%%%%%%%%%%%%%%%%%%%%%%%%%%%%%%%%%%%%%
%%%%%%%%%%%%%%%%%%%%%%%%%%%%%%%%%%%%%%%%%%%%%%
\par The trick is to write the $K-$ convexity of Entropy along the measure curve $\frac{\rho^d_t+\rho^u_t}{2\beta}$ at \emph{"fixed"} times, $t=b$ (a very small positive number less than $a$), $t = a$ and $t=a' = a + \epsilon$ and along the measure curves  $\frac{\rho^d_t}{\beta}$ and $\frac{\rho^u_t}{\beta}$ at \emph{"fixed"} times $t=a$, $t=a' = a + \epsilon$ and $t=1$. Recall that all the branching is happening within the tiny time interval $(a,a')$ and hence, these two measure curves coincide for times $t\le a$. The computations are similar to Step 7 in Rajala-Sturm~\cite{RS}:
%%%%%%%%%%%%%%%%%%%%%%%%%%%%%%%%%%%%%%%%%%%%%%
%%%%%%%%%%%%%%%%%%%%%%%%%%%%%%%%%%%%%%%%%%%%%%
$K-$ convexity of Entropy along the measure curve $\frac{\rho^d_t+\rho^u_t}{2\beta}$ implies that at times, $t=b$ , $t = a$ and $t=a' = a + \epsilon$, one has
\begin{align}
& Ent \left( \tfrac{\left(e_b \right)_\sharp \pi_a^d}{\beta} \vert m \right) =\int\frac{\rho^d_a}{\beta}\log\frac{\rho^d_a}{\beta}\,dm\notag\\
 &\leq \frac{\epsilon}{\epsilon+a-b}\int\frac{\rho^d_b}{\beta}\log\frac{\rho^d_b}{\beta}\,dm+\frac{a-b}{a + \epsilon -b}\int\frac{\rho^d_{a + \epsilon}+\rho^u_{a + \epsilon}}{2\beta}\log\frac{\rho^d_{a + \epsilon}+\rho^u_{a + \epsilon}}{2\beta}\,dm\notag\\
 &+\frac{\vert K\vert}{2}\frac{\epsilon(a-b)}{(a+\epsilon -b)^2}W_2^2\left((e_b)_\sharp\left(\frac{\pi^d+\pi^u}{2\beta}\right),(e_{a + \epsilon})_\sharp\left(\frac{\pi^d+\pi^u}{2\beta}\right)\right) =: \mathbf{II}\notag.
\end{align}

 %%%%%%%%%%%%%%%%%%%%%%%%%%%%%%%%%%%%%%%%%%%%%%
%%%%%%%%%%%%%%%%%%%%%%%%%%%%%%%%%%%%%%%%%%%%%%
Since $\supp \rho^u_{a+\epsilon}$ and $\supp \rho^d_{a+\epsilon}$ are mutually disjoint, we can proceed as
\small
\begin{align}
 &\int\frac{\rho^u_{a+\epsilon}+\rho^d_{a+\epsilon}}{2\beta}\log\frac{\rho^u_{a+\epsilon}+\rho^d_{a+\epsilon}}{2\beta}\,dm\notag\\
 &=\frac{1}{2}\int_{\supp\mu^u_{a+\epsilon}}\frac{\rho^u_{a+\epsilon}}{\beta}\log\frac{\rho^u_{a+\epsilon}}{2\beta}\,dm+\frac{1}{2}\int_{\supp\mu^d_{a+\epsilon}}\frac{\rho^d_{a+\epsilon}}{\beta}\log\frac{\rho^d_{a+\epsilon}}{2\beta}\,dm\notag\\
 &=\frac{1}{2}\int_{\supp\mu^u_{a+\epsilon}}\frac{\rho^u_{a+\epsilon}}{\beta}\log\frac{\rho^u_{a+\epsilon}}{\beta}\,dm+\frac{1}{2}\int_{\supp\mu^d_{a+\epsilon}}\frac{\rho^d_{a+\epsilon}}{\beta}\log\frac{\rho^d_{a+\epsilon}}{\beta}\,dm-\log2\notag\\
 &=\frac{1}{2}\int\frac{\rho^u_{a+\epsilon}}{\beta}\log\frac{\rho^u_{a+\epsilon}}{\beta}\,dm+\frac{1}{2}\int\frac{\rho^d_{a+\epsilon}}{\beta}\log\frac{\rho^d_{a+\epsilon}}{\beta}\,dm-\log2.\notag
\end{align}
\normalsize
 %%%%%%%%%%%%%%%%%%%%%%%%%%%%%%%%%%%%%%%%%%%%%%
%%%%%%%%%%%%%%%%%%%%%%%%%%%%%%%%%%%%%%%%%%%%%%

\par Both $\supp(e_0)_*((\pi^u+\pi^d)/2\beta)$ and $\supp(e_1)_*((\pi^u+\pi^d)/2\beta)$ are contained  in $B_{\eta/n^2}(x)$. For large $n$, every geodesic connecting a point in the former and a point in the latter is also contained in $B_{\eta}(x)$. This means the Wasserstein distance between these measures is at most $\eta$. Thus,
\small
\begin{align}
 &W_2^2\left( (e_b)_*\left(\frac{\pi^u+\pi^d}{2\beta}\right),(e_{a+\epsilon})_*\left(\frac{\pi^u+\pi^d}{2\beta}\right)\right)\notag\\
 &\leq (a+\epsilon-b)^2W_2^2\left((e_0)_*\left(\frac{\pi^u+\pi^d}{2\beta}\right),(e_1)_*\left(\frac{\pi^u+\pi^d}{2\beta}\right)\right)\notag\\
 &\leq (a+\epsilon-b)^2\eta^2.\notag
\end{align}
\normalsize
\par Now, using the density estimate
\small
 \be
 	\frac{d \left(e_b \right)_\sharp \pi^d}{d m } \le C, \notag
 \ee
 \normalsize
 we can continue as follows
 %%%%%%%%%%%%%%%%%%%%%%%%%%%%%%%%%%%%%%%%%%%%%%
%%%%%%%%%%%%%%%%%%%%%%%%%%%%%%%%%%%%%%%%%%%%%% 
 \small
 \begin{align}
 \mathbf{II} &\leq \frac{\epsilon}{a + \epsilon -b}\log\frac{C}{\beta}+\frac{\vert K\vert}{2}\epsilon(a-b)\eta^2-\frac{a-b}{a + \epsilon -b}\log 2\notag\\
 &+\frac{a-b}{2(a + \epsilon -b)}\left(\int\frac{\rho^d_{a + \epsilon}}{\beta}\log\frac{\rho^d_{a + \epsilon}}{\beta}\,dm+\int\frac{\rho^u_{a + \epsilon}}{\beta}\log\frac{\rho^u_{a + \epsilon}}{\beta}\,dm\right)\notag\\
 &\leq \frac{\epsilon}{a + \epsilon -b}\log\frac{C}{\beta}+ \frac{\vert K\vert}{2}\epsilon(a-b)\eta^2 -\frac{a-b}{a + \epsilon -b}\log 2\notag\\
 &+\frac{a-b}{2(a + \epsilon -b)}\left(\frac{\epsilon}{1-a}\int\frac{\rho_1^d}{\beta}\log\frac{\rho^d_1}{\beta}\,dm+\frac{1-a-\epsilon}{1-a}\int\frac{\rho^d_a}{\beta}\log\frac{\rho^d_a}{\beta}\,dm+\frac{\vert K\vert}{2}\epsilon (1-a-\epsilon)\eta^2 \right)\notag\\
 &+\frac{a-b}{2(a + \epsilon -b)}\left(\frac{\epsilon}{1-a}\int\frac{\rho_1^u}{\beta}\log\frac{\rho^u_1}{\beta}\,dm+\frac{1-a-\epsilon}{1-a}\int\frac{\rho^u_a}{\beta}\log\frac{\rho^u_a}{\beta}\,dm+\frac{\vert K\vert}{2}\epsilon(1-a-\epsilon)\eta^2\right).\notag
\end{align}
\normalsize
%%%%%%%%%%%%%%%%%%%%%%%%%%%%%%%%%%%%%%%%%%%%%%
%%%%%%%%%%%%%%%%%%%%%%%%%%%%%%%%%%%%%%%%%%%%%%
The above is equivalent to 
\small
\begin{align}\label{app:eq1}
 \int\frac{\rho^d_a}{\beta}\log\frac{\rho^d_a}{\beta}\,dm\leq \frac{\epsilon}{1-a}\log\frac{C}{\beta}-\frac{(1-a)(a-b)}{\epsilon(1-b)}\log 2+(a-b)(1-a)\frac{\vert K\vert}{2}\eta^2. 
\end{align}
\normalsize
Taking into account (\ref{app:eq1}), We can approximate the entropy of $\pi^u/\beta$ at time $a + \epsilon$ by 
\small
\begin{align}
 \int\frac{\rho^u_{a + \epsilon}}{\beta}\log\frac{\rho^u_{a + \epsilon}}{\beta}\,dm\leq \log\frac{C}{\beta}-(1-a-\epsilon)\log 2\left(\frac{a-\epsilon}{\epsilon(1-b)}-\frac{a+\epsilon-b}{3}\right).\label{app:eq:up}
\end{align}
\normalsize
On combining the upper estimate (\ref{app:eq:up}) and the lower estimate (\ref{lower}), we obtain 
\small
\begin{align}\label{app:eq2}
 \epsilon \left(\log\frac{\epsilon}{10m(B(x,\eta/2))}-\log C\right)\leq -(1-a-\epsilon)\log 2\left(\frac{a-b}{1-b}-\frac{a(a + \epsilon -b)}{3}\right).
\end{align} 
\normalsize
The right-hand side of (\ref{app:eq2}) is strictly negative (think of $b \searrow 0 $) while the left-hand side approaches to $0$ as $\epsilon$ goes to $0$ (recall that $\epsilon \to 0$ as $n \to \infty$). This is a contradiction. 
%%%%%%%%%%%%%%%%%%%%%%%%%%%%%%%%%%%%%%%%%%%%%%
%%%%%%%%%%%%%%%%%%%%%%%%%%%%%%%%%%%%%%%%%%%%%%
\par It is easier (computation-wise) to get the contradiction using the R\'{e}nyi entropy instead of the Shannon entropy as we demonstrate in below. For simplicity we assume $\beta = 1$ and $K \ge 0$. For general $K$, one would need to also incorporate torsion coefficients in the $K-$ convexity estimates and the contradiction will follow by letting $\epsilon \to 0$ and then, $N \to \infty$ (notice that for any $K$, the torsion coefficients, $\sigma_N(t)$ converge to $t$ as $N \to \infty$). These computations are similar to those carried out in Rajala~\cite{Rajala-poin}.

%%%%%%%%%%%%%%%%%%%%%%%%%%%%%%%%%%%%%%%%%%%%%%
%%%%%%%%%%%%%%%%%%%%%%%%%%%%%%%%%%%%%%%%%%%%%%
\small
\begin{align}
& \int\left( \rho^d_a \right)^{1 - \frac{1}{N}}\,dm\notag\\
 &\ge \frac{\epsilon}{\epsilon+a-b}\int\left( \rho^d_b \right)^{1 - \frac{1}{N}} \,dm+ 2^{ \frac{1}{N} - 1} \cdot \frac{a-b}{a + \epsilon -b} \int\; \left( \rho^d_{a + \epsilon}+\rho^u_{a + \epsilon} \right)^{1 - \frac{1}{N}} \,dm\notag\\
 & - \frac{\vert K\vert}{2}\frac{\epsilon(a-b)}{(a+\epsilon -b)^2}W_2^2\left((e_b)_\sharp\left(\frac{\pi^d+\pi^u}{2}\right),(e_{a + \epsilon})_\sharp\left(\frac{\pi^d+\pi^u}{2}\right)\right) \notag\\ & >   - \frac{\vert K\vert}{2}\epsilon(a-b)\eta^2\notag\\
 &+2^{\frac{1}{N}-1} \cdot\frac{a-b}{(a + \epsilon -b)}\left(\frac{\epsilon}{1-a}\int\left( \rho^d_1 \right)^{1 - \frac{1}{N}}\,dm+\frac{1-a-\epsilon}{1-a}\int\left( \rho^d_a \right)^{1 - \frac{1}{N}}\,dm - \frac{\vert K\vert}{2}\epsilon (1-a-\epsilon)\eta^2\right)\notag\\
 &+ 2^{\frac{1}{N}-1} \cdot \frac{a-b}{(a + \epsilon -b)}\left(\frac{\epsilon}{1-a}\int\left( \rho^u_1 \right)^{1 - \frac{1}{N}}\,dm+\frac{1-a-\epsilon}{1-a}\int\left( \rho^u_a \right)^{1 - \frac{1}{N}}\,dm - \frac{\vert K\vert}{2}\epsilon(1-a-\epsilon)\eta^2 \right) \notag \\ & =  2^{\frac{1}{N}} \left( \frac{a-b}{a + \epsilon -b} \right) \left( \frac{1-a-\epsilon}{1 -a} \right) \int\left( \rho^d_a \right)^{1 - \frac{1}{N}}\,dm - 2^{\frac{1}{N}} \left( \frac{a-b}{a + \epsilon -b} \right) \frac{\vert K\vert}{2}\epsilon(1-a-\epsilon)\eta^2 \notag
\end{align}
\normalsize
 Now, on letting $\epsilon \to 0$, we get
\small
\be
	\int\left( \rho^d_a \right)^{1 - \frac{1}{N}}\,dm \ge 2^{\frac{1}{N}} \int\left( \rho^d_a \right)^{1 - \frac{1}{N}}\,dm, \notag
\ee
\normalsize
which is an obvious contradiction for any $N$.

\subsection*{Acknowledgements}
  The authors are deeply grateful to Professor Shouhei Honda for suggesting this problem and for many fruitful discussions he provided. The authors would also like to thank Professors Tapio Rajala, Kota Hattori, and Qintao Deng for their valuable advice. This research project was conducted during the Junior Hausdorff Trimester Program on "Optimal Transportation" in Hausdorff Research Institute for Mathematics. We would like to thanks the organizers of this program and the staff at HIM. 

The authors are also very grateful to the anonymous reviewers of this paper whose valuable comments helped the authors to immensely improve these notes.

%\bibliographystyle{plain}
%\bibliography{reference2015}

\begin{bibdiv}
\begin{biblist}


















\bib{AMS-2}{article}{
   author={Ambrosio, L.},
   author={Mondino, A.},
   author={Savar\'{e}, G.},
   title={Nonlinear diffusion equations and curvature conditions in metric measure spaces},
   journal={arXiv:1509.07273},
   date={2015},
}







%\bib{AD}{article}{
%   author={Ambrosio, Luigi},
%   author={Di Marino, Simone},
%   title={Equivalent definitions of $BV$ space and of total variation on
%   metric measure spaces},
%   journal={J. Funct. Anal.},
%   volume={266},
%   date={2014},
%   number={7},
%%   pages={4150--4188},
%   issn={0022-1236},
%   review={\MR{3170206}},
%   doi={10.1016/j.jfa.2014.02.002},
%}




%\bib{AGSbook}{book}{
%   author={Ambrosio, Luigi},
%   author={Gigli, Nicola},
%   author={Savar{\'e}, Giuseppe},
%   title={Gradient flows in metric spaces and in the space of probability
%   measures},
%   series={Lectures in Mathematics ETH Z\"urich},
%   edition={2},
%   publisher={Birkh\"auser Verlag, Basel},
%   date={2008},
%   pages={x+334},
%   isbn={978-3-7643-8721-1},
%   review={\MR{2401600 (2009h:49002)}},
%}

%\bib{AGSbercd}{article}{
%   author={Ambrosio, Luigi},
%   author={Gigli, Nicola},
%   author={Savar{\'e}, Giuseppe},
%   title={Bakry-\'Emery curvature-dimension condition and Riemannian Ricci curvature bounds},
%   journal={arXiv:1209.5786},
   %volume={195},
   %date={2014},
   %number={2},
   %pages={289--391},
   %issn={0020-9910},
   %review={\MR{3152751}},
   %doi={10.1007/s00222-013-0456-1},
%}

%\bib{AGScal}{article}{
%   author={Ambrosio, Luigi},
%   author={Gigli, Nicola},
%   author={Savar{\'e}, Giuseppe},
%   title={Calculus and heat flow in metric measure spaces and applications
%   to spaces with Ricci bounds from below},
%   journal={Invent. Math.},
%   volume={195},
%   date={2014},
%   number={2},
%   pages={289--391},
%   issn={0020-9910},
%   review={\MR{3152751}},
%   doi={10.1007/s00222-013-0456-1},
%}

\bib{AGSRiem}{article}{
   author={Ambrosio, Luigi},
   author={Gigli, Nicola},
   author={Savar{\'e}, Giuseppe},
   title={Metric measure spaces with Riemannian Ricci curvature bounded from
   below},
   journal={Duke Math. J.},
   volume={163},
   date={2014},
   number={7},
   pages={1405--1490},
   issn={0012-7094},
   review={\MR{3205729}},
   doi={10.1215/00127094-2681605},
}

\bib{AGMR}{article}{
   author={Ambrosio, Luigi},
   author={Gigli, Nicola},
   author={Mondino, Andrea},
   author={Rajala, Tapio},
   title={Riemannian Ricci curvature lower bounds in metric measure spaces
   with $\sigma$-finite measure},
   journal={Trans. Amer. Math. Soc.},
   volume={367},
   date={2015},
   number={7},
   pages={4661--4701},
   issn={0002-9947},
   review={\MR{3335397}},
   doi={10.1090/S0002-9947-2015-06111-X},
}


%\bib{AMP}{article}{
%   author={Ambrosio, L.},
%   author={Miranda, M., Jr.},
%   author={Pallara, D.},
%   title={Special functions of bounded variation in doubling metric measure
%   spaces},
%   conference={
%      title={Calculus of variations: topics from the mathematical heritage
%      of E.\ De Giorgi},
%   },
%   book={
%      series={Quad. Mat.},
%      volume={14},
%      publisher={Dept. Math., Seconda Univ. Napoli, Caserta},
%   },
%   date={2004},
%   pages={1--45},
%   review={\MR{2118414 (2005j:49036)}},
%}


\bib{AT}{book}{
   author={Ambrosio, Luigi},
   author={Tilli, Paolo},
   title={Topics on analysis in metric spaces},
   series={Oxford Lecture Series in Mathematics and its Applications},
   volume={25},
   publisher={Oxford University Press, Oxford},
   date={2004},
   pages={viii+133},
   isbn={0-19-852938-4},
   review={\MR{2039660 (2004k:28001)}},
}

\bib{BC}{article}{
   author={Bianchini, S.},
   author={Cavalletti, F.},
   title={The Monge problem for distance cost in geodesic spaces},
   journal={Comm. Math. Phys.},
   volume={318},
   date={2013},
   number={3},
   pages={615--673},
   issn={0010-3616},
   review={\MR{3027581}},
   doi={10.1007/s00220-013-1663-8},
}

%\bib{BE}{article}{
%   author={Bakry, D.},
%   author={{\'E}mery, Michel},
%   title={Diffusions hypercontractives},
%   language={French},
%   conference={
%      title={S\'eminaire de probabilit\'es, XIX, 1983/84},
%   },
%   book={
%      series={Lecture Notes in Math.},
%      volume={1123},
%      publisher={Springer, Berlin},
%   },
%   date={1985},
%   pages={177--206},
%   review={\MR{889476 (88j:60131)}},
%   doi={10.1007/BFb0075847},
%}

%\bib{BGL}{book}{
%   author={Bakry, Dominique},
%   author={Gentil, Ivan},
%   author={Ledoux, Michel},
%   title={Analysis and geometry of Markov diffusion operators},
%   series={Grundlehren der Mathematischen Wissenschaften [Fundamental
%   Principles of Mathematical Sciences]},
%   volume={348},
%   publisher={Springer, Cham},
%   date={2014},
%   pages={xx+552},
%   isbn={978-3-319-00226-2},
%   isbn={978-3-319-00227-9},
%   review={\MR{3155209}},
%   doi={10.1007/978-3-319-00227-9},
%}

%\bib{BH}{book}{
%   author={Bouleau, Nicolas},
%   author={Hirsch, Francis},
%   title={Dirichlet forms and analysis on Wiener space},
%   series={de Gruyter Studies in Mathematics},
%   volume={14},
%   publisher={Walter de Gruyter \& Co., Berlin},
%   date={1991},
%   pages={x+325},
%   isbn={3-11-012919-1},
%   review={\MR{1133391 (93e:60107)}},
%   doi={10.1515/9783110858389},
%}

\bib{BS}{article}{
   author={Bacher, Kathrin},
   author={Sturm, Karl-Theodor},
   title={Localization and tensorization properties of the
   curvature-dimension condition for metric measure spaces},
   journal={J. Funct. Anal.},
   volume={259},
   date={2010},
   number={1},
   pages={28--56},
   issn={0022-1236},
   review={\MR{2610378 (2011i:53050)}},
   doi={10.1016/j.jfa.2010.03.024},
}

%\bib{C}{article}{
%   author={Cheeger, J.},
%   title={Differentiability of Lipschitz functions on metric measure spaces},
%   journal={Geom. Funct. Anal.},
%   volume={9},
%   date={1999},
%   number={3},
%   pages={428--517},
%   issn={1016-443X},
%   review={\MR{1708448 (2000g:53043)}},
%   doi={10.1007/s000390050094},
%}







Geometric and Functional Analysis
April 2014, Volume 24, Issue 2, pp 493-551


\bib{Cav-decomp}{article}{
   author={ Cavalletti, F.},
   title={Decomposition of geodesics in the Wasserstein space and the globalization property},
   journal={Geometric and Functional Analysis},
     volume={24},
     number={2}, 
     date={2014},
      pages={493--551},
  }
  



\bib{FM}{article}{
   author={Cavalletti, Fabio},
   author={Mondino, Andrea},
   title={Measure rigidity of Ricci curvature lower bounds},
  journal={Advances in Mathematics},
       volume={286},
       date={2016},
        pages={430--480},
    }



\bib{CM-1}{article}{
   author={Cavalletti, F.},
   author={Mondino, A.},
   title={Sharp and rigid isoperimetric inequalities in metric-measure spaces with lower Ricci curvature bounds},
   journal={arXiv:1502.06465},
}



\bib{CS-2}{article}{
   author={Cavalletti, Fabio},
   author={Sturm, Karl Theodor},
   title={Local curvature-dimension condition implies
measure-contraction property},
   journal={J. Funct. Anal.},
   volume={262},
   date={2012},
    pages={5110--5127},
}










\bib{CCwarped}{article}{
   author={Cheeger, Jeff},
   author={Colding, Tobias H.},
   title={Lower bounds on Ricci curvature and the almost rigidity of warped
   products},
   journal={Ann. of Math. (2)},
   volume={144},
   date={1996},
   number={1},
   pages={189--237},
   issn={0003-486X},
   review={\MR{1405949 (97h:53038)}},
   doi={10.2307/2118589},
}

\bib{CC1}{article}{
   author={Cheeger, Jeff},
   author={Colding, Tobias H.},
   title={On the structure of spaces with Ricci curvature bounded below. I},
   journal={J. Differential Geom.},
   volume={46},
   date={1997},
   number={3},
   pages={406--480},
   issn={0022-040X},
   review={\MR{1484888 (98k:53044)}},
}

\bib{CC2}{article}{
   author={Cheeger, Jeff},
   author={Colding, Tobias H.},
   title={On the structure of spaces with Ricci curvature bounded below. II},
   journal={J. Differential Geom.},
   volume={54},
   date={2000},
   number={1},
   pages={13--35},
   issn={0022-040X},
   review={\MR{1815410 (2003a:53043)}},
}

\bib{CC3}{article}{
   author={Cheeger, Jeff},
   author={Colding, Tobias H.},
   title={On the structure of spaces with Ricci curvature bounded below.
   III},
   journal={J. Differential Geom.},
   volume={54},
   date={2000},
   number={1},
   pages={37--74},
   issn={0022-040X},
   review={\MR{1815411 (2003a:53044)}},
}

%\bib{CGY}{article}{
%   author={Chung, F. R. K.},
%   author={Grigor{\cprime}yan, A.},
%   author={Yau, S.-T.},
%   title={Eigenvalues and diameters for manifolds and graphs},
%   conference={
%      title={Tsing Hua lectures on geometry \& analysis},
%      address={Hsinchu},
%      date={1990--1991},
%   },
%   book={
%      publisher={Int. Press, Cambridge, MA},
%   },
%   date={1997},
%   pages={79--105},
%   review={\MR{1482032 (98h:58206)}},
%}

\bib{LinaChen}{article}{
   author={Chen, Lina},
   title={A remark on regular points of Ricci limit spaces},
   journal={Frontiers of Mathematics in China},
   volume={11},
   date={2016},
   number={1},
   pages={21--26},
}


\bib{C-Large}{article}{
   author={Colding, Tobias H.},
   title={Large manifolds with positive Ricci curvature},
   journal={Inventiones mathematicae},
   volume={124},
   date={1996},
   number={1},
   pages={193--214},
}


\bib{CN}{article}{
   author={Colding, Tobias Holck},
   author={Naber, Aaron},
   title={Sharp H\"older continuity of tangent cones for spaces with a lower
   Ricci curvature bound and applications},
   journal={Ann. of Math. (2)},
   volume={176},
   date={2012},
   number={2},
   pages={1173--1229},
   issn={0003-486X},
   review={\MR{2950772}},
   doi={10.4007/annals.2012.176.2.10},
}


%\bib{EG}{book}{
%   author={Evans, Lawrence C.},
%   author={Gariepy, Ronald F.},
%   title={Measure theory and fine properties of functions},
%   series={Studies in Advanced Mathematics},
%   publisher={CRC Press, Boca Raton, FL},
%   date={1992},
%   pages={viii+268},
%   isbn={0-8493-7157-0},
%   review={\MR{1158660 (93f:28001)}},
%}


\bib{Guy}{article}{
   author={David, Guy C.},
   title={Tangents and Rectifiability of Ahlfors Regular Lipschitz Differentiability Spaces},
   journal={Geom. Funct. Anal.},
   volume={25},
   date={2015},
   pages={553--579},
}



\bib{EKS}{article}{
   author={Erbar, Matthias},
   author={Kuwada, Kazumasa},
   author={Sturm, Karl-Theodor},
   title={On the equivalence of the entropic curvature-dimension condition and Bochner's Inequality on metric measure spaces},
   journal={Inventiones mathematicae},
   volume={201},
   number={3},
      date={2015},
      pages={993--1071},
}

\bib{Fo}{book}{
   author={Folland, Gerald B.},
   title={Real analysis},
   series={Pure and Applied Mathematics (New York)},
   edition={2},
   note={Modern techniques and their applications;
   A Wiley-Interscience Publication},
   publisher={John Wiley \& Sons, Inc., New York},
   date={1999},
   pages={xvi+386},
   isbn={0-471-31716-0},
   review={\MR{1681462 (2000c:00001)}},
}



%\bib{FS}{article}{
%   author={Cavalletti, Fabio},
%   author={Sturm, Karl-Theodor},
%   title={Local curvature-dimension condition implies measure-contraction
%   property},
%   journal={J. Funct. Anal.},
%   volume={262},
%   date={2012},
%   number={12},
%   pages={5110--5127},
%   issn={0022-1236},
%   review={\MR{2916062}},
%   doi={10.1016/j.jfa.2012.02.015},
%}

\bib{Goverview}{article}{
   author={Gigli, Nicola},
   title={An overview of the proof of the splitting theorem in spaces with
   non-negative Ricci curvature},
   journal={Anal. Geom. Metr. Spaces},
   volume={2},
   date={2014},
   pages={169--213},
   issn={2299-3274},
   review={\MR{3210895}},
   doi={10.2478/agms-2014-0006},
}

\bib{Gsplit}{article}{
   author={Gigli, Nicola},
   title={The splitting theorem in non-smooth context},
   journal={arXiv:1302.5555},
   %volume={2},
   %date={2014},
   %pages={169--213},
   %issn={2299-3274},
   %review={\MR{3210895}},
   %doi={10.2478/agms-2014-0006},
}

%\bib{GH}{article}{
%   author={Gigli, Nicola},
%   author={Han, Bang-Xian},
%%   title={Independence on $p$ of weak upper gradients on RCD spaces},
%   journal={arXiv:1407.7350},
%}   

%\bib{GKO}{article}{
%   author={Gigli, Nicola},
%   author={Kuwada, Kazumasa},
%   author={Ohta, Shin-Ichi},
%   title={Heat flow on Alexandrov spaces},
%   journal={Comm. Pure Appl. Math.},
%   volume={66},
%   date={2013},
%   number={3},
%   pages={307--331},
%   issn={0010-3640},
%   review={\MR{3008226}},
%   doi={10.1002/cpa.21431},
%}

\bib{GMR}{article}{
   author={Gigli, Nicola},
   author={Mondino, Andrea},
   author={Rajala, Tapio},
   title={Euclidean spaces as weak tangents of infinitesimally Hilbertian metric spaces with Ricci curvature bounded below},
   journal={J. Reine Angew. Math.},
 volume={2015},
   date={2015},
   pages={233--244},
}



\bib{GMS}{article}{
   author={Gigli, Nicola},
   author={Mondino, Andrea},
   author={Savar{\'e}, Giuseppe},
   title={Convergence of pointed non-compact metric measure spaces and stability of Ricci curvature bounds and heat flows},
    journal={ Journal of Geometric Analysis },
      volume={(online first)},
      date={2015},
      pages={1--16},
   }

\bib{GRS}{article}{
   author={Gigli, Nicola},
   author={Rajala, Tapio},
   author={Sturm, Karl-Theodor},
   title={Optimal maps and exponentiation on finite dimensional spaces with Ricci curvature bounded from below},
   journal={arXiv:1305.4849},
}


%\bib{GW}{article}{
%   author={Gong, Fu-Zhou},
%   author={Wang, Feng-Yu},
%   title={Heat kernel estimates with application to compactness of
%   manifolds},
%   journal={Q. J. Math.},
%   volume={52},
%   date={2001},
%   number={2},
%   pages={171--180},
%   issn={0033-5606},
%   review={\MR{1838361 (2002c:58039)}},
%   doi={10.1093/qjmath/52.2.171},
%}








\bib{HBG}{article}{
   author={Honda, Shouhei},
   title={Bishop-Gromov type inequality on Ricci limit spaces},
   journal={J. Math. Soc. Japan},
   volume={63},
   date={2011},
   number={2},
   pages={419--442},
   issn={0025-5645},
   review={\MR{2793106 (2012d:53123)}},
}

\bib{Hlow}{article}{
   author={Honda, Shouhei},
   title={On low-dimensional Ricci limit spaces},
   journal={Nagoya Math. J.},
   volume={209},
   date={2013},
   pages={1--22},
   issn={0027-7630},
   review={\MR{3032136}},
}

%\bib{HK}{article}{
%   author={Haj{\l}asz, Piotr},
%   author={Koskela, Pekka},
%   title={Sobolev met Poincar\'e},
%   journal={Mem. Amer. Math. Soc.},
%   volume={145},
%   date={2000},
%   number={688},
%   pages={x+101},
%   issn={0065-9266},
%   review={\MR{1683160 (2000j:46063)}},
%   doi={10.1090/memo/0688},
%}




\bib{KT}{article}{
   author={Ketterer, C.},
   author={Rajala, T.},
   title={Failure of Topological Rigidity Results for the Measure Contraction Property},
   journal={Potential Analysis},
   volume={42},
   number={3}
   date={2015},
   pages={645--655},
}








%\bib{KeObata}{article}{
%   author={Ketterer, Christian},
%   title={Obata's rigidity theorem for metric measure spaces},
%   journal={arXiv:1410.5210},
%}

\bib{Ki}{article}{
   author={Kitabeppu, Yu},
   title={Lower bound of coarse Ricci curvature on metric measure spaces and
   eigenvalues of Laplacian},
   journal={Geom. Dedicata},
   volume={169},
   date={2014},
   pages={99--107},
   issn={0046-5755},
   review={\MR{3175238}},
   doi={10.1007/s10711-013-9844-3},
}

%\bib{K}{article}{
%   author={Kuwada, Kazumasa},
%   title={Duality on gradient estimates and Wasserstein controls},
%   journal={J. Funct. Anal.},
%   volume={258},
%   date={2010},
%   number={11},
%   pages={3758--3774},
%   issn={0022-1236},
%   review={\MR{2606871 (2011d:35109)}},
%   doi={10.1016/j.jfa.2010.01.010},
%}

%\bib{Lremarks}{article}{
%   author={Ledoux, M.},
%   title={Remarks on logarithmic Sobolev constants, exponential
%   integrability and bounds on the diameter},
%   journal={J. Math. Kyoto Univ.},
%   volume={35},
%   date={1995},
%   number={2},
%   pages={211--220},
%   issn={0023-608X},
%   review={\MR{1346225 (97m:58209)}},
%}

%\bib{L}{book}{
%   author={Ledoux, Michel},
%   title={The concentration of measure phenomenon},
%   series={Mathematical Surveys and Monographs},
%   volume={89},
%   publisher={American Mathematical Society, Providence, RI},
%   date={2001},
%   pages={x+181},
%   isbn={0-8218-2864-9},
%   review={\MR{1849347 (2003k:28019)}},
%}

\bib{LV}{article}{
   author={Lott, John},
   author={Villani, C{\'e}dric},
   title={Ricci curvature for metric-measure spaces via optimal transport},
   journal={Ann. of Math. (2)},
   volume={169},
   date={2009},
   number={3},
   pages={903--991},
   issn={0003-486X},
   review={\MR{2480619 (2010i:53068)}},
   doi={10.4007/annals.2009.169.903},
}





\bib{Mil-iso}{article}{
   author={Milman, Emanuel},
   title={Sharp isoperimetric inequalities and model spaces for the Curvature-Dimension-Diameter condition},
   journal={J. Eur. Math. Soc. (JEMS)},
   volume={17},
   date={2015},
   number={5},
   pages={1041--1078},
}




\bib{MN}{article}{
   author={Mondino, Andrea},
   author={Naber, Aaron},
   title={Structure theory of metric-measure spaces with lower Ricci curvature bounds I},
   journal={arXiv:1405.2222v2},
   %volume={163},
   %date={2014},
   %number={7},
   %pages={1405--1490},
   %issn={0012-7094},
   %review={\MR{3205729}},
   %doi={10.1215/00127094-2681605},
}



%\bib{Obata}{article}{
%   author={Obata, Morio},
%   title={Certain conditions for a Riemannian manifold to be isometric with a sphere},
%   journal={J. Math. Soc. Japan},
%   volume={14},
%   date={1962},
%   number={3},
%   pages={333--340},
%}
















\bib{OhFinsler}{article}{
   author={Ohta, Shin-ichi},
   title={Finsler interpolation inequalities},
   journal={Calc. Var. Partial Differential Equations},
   volume={36},
   date={2009},
   number={2},
   pages={211--249},
   issn={0944-2669},
   review={\MR{2546027 (2011m:58027)}},
   doi={10.1007/s00526-009-0227-4},
}










\bib{Ohmcp}{article}{
   author={Ohta, Shin-ichi},
   title={On the measure contraction property of metric measure spaces},
   journal={Comment. Math. Helv.},
   volume={82},
   date={2007},
   number={4},
  pages={805--828},
   issn={0010-2571},
   review={\MR{2341840 (2008j:53075)}},
   doi={10.4171/CMH/110},
}

%\bib{OSheat}{article}{
%   author={Ohta, Shin-Ichi},
%   author={Sturm, Karl-Theodor},
%   title={Heat flow on Finsler manifolds},
%   journal={Comm. Pure Appl. Math.},
%   volume={62},
%   date={2009},
%   number={10},
%   pages={1386--1433},
%   issn={0010-3640},
%   review={\MR{2547978 (2010j:58058)}},
%   doi={10.1002/cpa.20273},
%}

%\bib{OSnoncontract}{article}{
%   author={Ohta, Shin-ichi},
%   author={Sturm, Karl-Theodor},
%   title={Non-contraction of heat flow on Minkowski spaces},
%   journal={Arch. Ration. Mech. Anal.},
%   volume={204},
%   date={2012},
%   number={3},
%   pages={917--944},
%   issn={0003-9527},
%   review={\MR{2917125}},
%   doi={10.1007/s00205-012-0493-8},
%}

%\bib{O}{article}{
%   author={Ollivier, Yann},
%   title={Ricci curvature of Markov chains on metric spaces},
%   journal={J. Funct. Anal.},
%   volume={256},
%   date={2009},
%   number={3},
%   pages={810--864},
%   issn={0022-1236},
%   review={\MR{2484937 (2010j:58081)}},
%   doi={10.1016/j.jfa.2008.11.001},
%}

%\bib{R}{article}{
%   author={Rajala, Tapio},
%   title={Interpolated measures with bounded density in metric spaces
%   satisfying the curvature-dimension conditions of Sturm},
%   journal={J. Funct. Anal.},
%   volume={263},
%   date={2012},
%   number={4},
%   pages={896--924},
%   issn={0022-1236},
%   review={\MR{2927398}},
%   doi={10.1016/j.jfa.2012.05.006},
%}



\bib{Rajala-poin}{article}{
   author={Rajala, Tapio},
   title={Local Poincar\'{e} inequalities from stable curvature conditions on metric spaces},
   journal={Calculus of Variations and Partial Differential Equations},
   volume={44},
   date={2012},
   number={3-4},
   pages={447--494},
}




\bib{RS}{article}{
   author={Rajala, Tapio},
   author={Sturm, Karl-Theodor},
   title={Non-branching geodesics and optimal maps in strong
   $CD(K,\infty)$-spaces},
   journal={Calc. Var. Partial Differential Equations},
   volume={50},
   date={2014},
   number={3-4},
   pages={831--846},
   issn={0944-2669},
   review={\MR{3216835}},
   doi={10.1007/s00526-013-0657-x},
}

%\bib{vRS}{article}{
%   author={von Renesse, Max-K.},
%   author={Sturm, Karl-Theodor},
%   title={Transport inequalities, gradient estimates, entropy, and Ricci
%   curvature},
%   journal={Comm. Pure Appl. Math.},
%   volume={58},
%   date={2005},
%   number={7},
%   pages={923--940},
%   issn={0010-3640},
%   review={\MR{2142879 (2006j:53048)}},
%   doi={10.1002/cpa.20060},
%}

%\bib{SC}{article}{
%   author={Saloff-Coste, L.},
%   title={Convergence to equilibrium and logarithmic Sobolev constant on
%   manifolds with Ricci curvature bounded below},
%   journal={Colloq. Math.},
%   volume={67},
%   date={1994},
%   number={1},
%   pages={109--121},
%   issn={0010-1354},
%   review={\MR{1292948 (95g:58258)}},
%}

%\bib{S}{article}{
%   author={Savar{\'e}, Giuseppe},
%   title={Self-improvement of the Bakry-\'Emery condition and Wasserstein
%   contraction of the heat flow in ${\rm RCD}(K,\infty)$ metric measure
%   spaces},
%   journal={Discrete Contin. Dyn. Syst.},
%   volume={34},
%   date={2014},
%   number={4},
%   pages={1641--1661},
%   issn={1078-0947},
%   review={\MR{3121635}},
%   doi={10.3934/dcds.2014.34.1641},
%}

%\bib{Stanal}{article}{
%   author={Sturm, Karl-Theodor},
%   title={Analysis on local Dirichlet spaces. II. Upper Gaussian estimates
%   for the fundamental solutions of parabolic equations},
%   journal={Osaka J. Math.},
%   volume={32},
%   date={1995},
%   number={2},
%   pages={275--312},
%   issn={0030-6126},
%   review={\MR{1355744 (97b:35003)}},
%}

%\bib{Stanal3}{article}{
%   author={Sturm, K. T.},
%   title={Analysis on local Dirichlet spaces. III. The parabolic Harnack
%   inequality},
%   journal={J. Math. Pures Appl. (9)},
%   volume={75},
%   date={1996},
%   number={3},
%   pages={273--297},
%   issn={0021-7824},
%   review={\MR{1387522 (97k:31010)}},
%}

%\bib{Shubin}{book}{
%   author={Shubin, M. A.},
%   title={Pseudodifferential Operators and Spectral Theory},
%   series={Mathematical Surveys and Monographs},
%   volume={89},
%   publisher={Springer-Verlag Berlin Heidelberg},
%   date={2001},
%   pages={288},
%   isbn={978-3-540-41195-6},
%}



\bib{Stmms1}{article}{
   author={Sturm, Karl-Theodor},
   title={On the geometry of metric measure spaces. I},
   journal={Acta Math.},
   volume={196},
   date={2006},
   number={1},
   pages={65--131},
   issn={0001-5962},
   review={\MR{2237206 (2007k:53051a)}},
   doi={10.1007/s11511-006-0002-8},
}

\bib{Stmms2}{article}{
   author={Sturm, Karl-Theodor},
   title={On the geometry of metric measure spaces. II},
   journal={Acta Math.},
   volume={196},
   date={2006},
   number={1},
   pages={133--177},
   issn={0001-5962},
   review={\MR{2237207 (2007k:53051b)}},
   doi={10.1007/s11511-006-0003-7},
}

%\bib{V}{book}{
%   author={Villani, C{\'e}dric},
%   title={Optimal transport},
%   series={Grundlehren der Mathematischen Wissenschaften [Fundamental
%   Principles of Mathematical Sciences]},
%   volume={338},
%   note={Old and new},
%   publisher={Springer-Verlag, Berlin},
%   date={2009},
%   pages={xxii+973},
%   isbn={978-3-540-71049-3},
%   review={\MR{2459454 (2010f:49001)}},
%   doi={10.1007/978-3-540-71050-9},
%}

%\bib{Vtopics}{book}{
%   author={Villani, C{\'e}dric},
%   title={Topics in optimal transportation},
%   series={Graduate Studies in Mathematics},
%   volume={58},
%   publisher={American Mathematical Society, Providence, RI},
%   date={2003},
%   pages={xvi+370},
%   isbn={0-8218-3312-X},
%   review={\MR{1964483 (2004e:90003)}},
%}


\end{biblist}
\end{bibdiv}
\end{document}